\documentclass[reqno,11pt]{amsart}
\usepackage[colorlinks=true, linkcolor=blue, citecolor=blue]{hyperref}

\usepackage{amssymb}
\usepackage{amsmath, graphicx, rotating}
\usepackage{color}     
\usepackage{soul}
\usepackage[dvipsnames]{xcolor}   
\usepackage{tcolorbox}

\usepackage[T1]{fontenc}
\usepackage{lmodern}
\usepackage[english]{babel}

\usepackage{ upgreek }
\usepackage{stmaryrd}
\SetSymbolFont{stmry}{bold}{U}{stmry}{m}{n}
\usepackage{amsthm}
\usepackage{float}

\usepackage{ bbm }
\usepackage{ stmaryrd }
\usepackage{ mathrsfs }
\usepackage{ frcursive }
\usepackage{ comment }

\usepackage{pgf, tikz}
\usetikzlibrary{shapes}
\usepackage{varioref}
\usepackage{enumitem}
\usepackage{longtable}

\usepackage{mathtools}

\usepackage{dsfont}

\definecolor{rouge}{rgb}{0.7,0.00,0.00}
\definecolor{vert}{rgb}{0.00,0.5,0.00}
\definecolor{bleu}{rgb}{0.00,0.00,0.8}

\usepackage[margin=1.3in]{geometry}

\newtheorem{theorem}{Theorem}[section]
\newtheorem*{theorem*}{Theorem}
\newtheorem{lemma}[theorem]{Lemma}

\newtheorem{corollary}[theorem]{Corollary}
\newtheorem{proposition}[theorem]{Proposition}

\labelformat{hypothesis}{\textbf{M\kern-0.1mm#1}}

\renewcommand\dots{\hbox to 1em{.\hss.\hss.}}

\theoremstyle{definition}

\newtheorem{remark}[theorem]{Remark}

\numberwithin{equation}{section}

\def\bb#1{\mathbb{#1}}

\def\bf#1{\mathbf{#1}}
\def\scr#1{\mathscr{#1}}

\def\geq{\geqslant}
\def\leq{\leqslant}
\def\vphi{\varphi}

\def\ddd{\ldots}
 
\newcommand\ee{\varepsilon}

\def\geq{\geqslant}
\def\leq{\leqslant}
\def\Rd {\mathbb{R}^d}
\def\Rd*{(\mathbb{R}^d)^*}
\def\Pd{{\mathbb{P}}^{d-1}}
\def\Pd*{(\mathbb{P}^{d-1})^*}

\def\bb#1{\mathbb{#1}}

\begin{document}
\title[Conditioned random walks on linear groups]{Conditioned random walks on linear groups I: construction of the target harmonic measure} 

\author{Ion Grama}
\author{Jean-Fran\c cois Quint}
\author{Hui Xiao}

\curraddr[Grama, I.]{Univ Bretagne Sud, CNRS UMR 6205, LMBA, Vannes, France.}
\email{ion.grama@univ-ubs.fr}

\curraddr[Quint, J.-F.]{IMAG, Univ Montpellier, CNRS UMR 5149, Montpellier, France.}
\email{Jean-Francois.Quint@umontpellier.fr}

\curraddr[Xiao, H.]{State Key Laboratory of Mathematical Sciences, Academy of Mathematics and Systems Science, Chinese Academy of Sciences, Beijing 100190, China.}
\email{xiaohui@amss.ac.cn}

\date{\today }

\begin{abstract}
Our objective is to explore random walks on the general linear group, constrained to a specific domain, 
with a primary focus on establishing the conditioned local limit theorem. 
This paper represents the first step toward achieving this goal, 
specifically entailing the construction of a novel entity -- the target harmonic measure. 
This measure, together with the harmonic function, serves as a pivotal component in establishing
 the conditioned local limit theorem.
Using a reversal identity, we introduce a reversed sequence characterized as a dual random walk with a perturbation depending on future observations. The investigation of such walks, which rely on future information, lies at the heart of this paper.
To carry out this study, we develop an approach grounded in the finite-size approximation of perturbations, 
enabling us to simplify the investigation to an array of Markov chains with increasing dimensions.
\end{abstract}

\maketitle

\tableofcontents


\section{Introduction and results} \label{sec-Introduction and results}

\subsection{Notation and background} \label{sec-background and motiv}
We define $\bb R$ and $\bb R_+=[0,\infty)$ as the real line and  the non-negative half-line, respectively. 
The set of non-negative integers is denoted by $\bb N$ and that of positive integers by $\bb N^*$.
Let $\bb V$ be a finite-dimensional vector space over the field $\bb R$ and
let $\bb G = \mathrm{GL}(\bb V)$ be the group of linear automorphisms of $\bb V$.
We equip $\bb V$ with a Euclidean norm $\| \cdot \|$. 
If $g$ is a linear endomorphism of $\bb V$, 
we write $\|g\| = \sup_{v \in \bb V \smallsetminus \{0\}} \frac{\| gv \|}{\| v \|}$ for the operator norm of $g$. 
The group $\bb G$ acts on the projective space $\bb P(\bb V)$ of $\bb V$ through the formula
$ g ( \bb R v ) = \bb R (gv)$, where $g \in \bb G$ and $v \in \bb V \smallsetminus \{0\}$. 

Let $\mu$ be a Borel probability measure on the group $\bb G$. 
Denote by $\Gamma_{\mu}$ the closed subsemigroup of $\bb G$ spanned by the support of the measure $\mu$. 
We will make use of the condition that $\Gamma_{\mu}$ is proximal, meaning that   
$\Gamma_{\mu}$ contains an element $g$ such that 
the characteristic polynomial of $g$ admits a unique root of maximal modulus, and that this root is simple. 
We will additionally assume that $\Gamma_{\mu}$ is strongly irreducible, meaning that no finite union of proper non-zero subspaces 
of $\bb V$ is $\Gamma_{\mu}$-invariant. 
Finally, we will say that $\mu$ admits a finite exponential moment if 
there exists a constant  $\alpha >0$ such that 
\begin{align}\label{Exponential-moment}
\int_{\bb G} \max \{ \|g\|,  \|g^{-1} \| \}^{\alpha} \mu(dg) < \infty. 
\end{align}
For $g \in \bb G$ and $x = \bb R v \in \bb P(\bb V)$, we write 
\begin{align*}
\sigma(g, x) = \log \frac{\|gv\|}{\|v\|}. 
\end{align*}
The function $\sigma: \bb G \times \bb P(\bb V) \to \bb R$ satisfies the cocycle identity: for any $g, h \in \bb G$ and $x \in \bb P(\bb V)$, we have 
$\sigma(g h, x) = \sigma(g, h x) + \sigma(h, x).$ 

Below we will consider random variables with values in $\bb R$ and $\bb G$-valued random elements  that all are assumed to be 
 defined on some probability space $(\Omega, \scr A, \bb P)$. 
The expectation with respect to the probability measure $\bb P$ is denoted by $\bb E$.
We denote by $\mathds 1_{B}$ the indicator function of the event $B\in \scr A$. 
 For brevity, given a random variable  $X$  and an event $B\in \scr A$, we will write $\bb E (X; B)$ 
 for the expectation $\bb E (X \mathds 1_{B})$.  

Let $g_1, g_2, \ldots$ be a sequence of independent and identically distributed random elements of $\bb G$ with law $\mu$.
For any starting point $x\in \bb P(\bb V)$, consider the random walk
\begin{align} \label{def of direct RW-001}
 \sigma(g_n \cdots g_1,x) = \sum_{k=1}^{n}  \sigma(g_k, g_{k-1} \cdots g_1 x) , 
 \quad n\geq 1,
\end{align}
where, by a convention applied throughout the paper, 
whenever $m<1$, the empty left product $g_m \cdots g_1$ is
identified with the identity matrix.  
Under the assumptions that the measure $\mu$ 
admits an exponential moment and its associated semigroup $\Gamma_\mu$ is proximal and strongly irreducible, 
it is well known that there exists a real number $\lambda_\mu$, 
called the first Lyapunov exponent of $\mu$, 
such that for any $x \in \bb P(\bb V)$, 
$\bb P$-almost surely, as $n \to \infty$, 
\begin{align*}
\frac{1}{n}  \sigma(g_n \cdots g_1, x) \to \lambda_{\mu}. 
\end{align*}
We assume that the random walk \eqref{def of direct RW-001} is centered, meaning that  $\lambda_{\mu} = 0$.  
For any $x \in \bb P(\bb V)$ and $t\in \bb R$, consider 
the first time when the random walk 
$(t+\sigma(g_k \cdots g_1, x))_{k\geq 1}$
leaves the non-negative half-line $\bb R_+$: 
\begin{align} \label{stopping time tau x t-001}
\tau_{x, t}   = \min \{ k \geq 1: t + \sigma(g_k \cdots g_1, x) < 0 \}, 
\end{align}
where by convention $\min \emptyset =\infty$.
A similar time can be associated with the random walk $(t-\sigma(g_k \cdots g_1, x))_{k\geq 1}$: 
\begin{align} \label{stopping time tau x t-002}
\check \tau_{x, t}   = \min \{ k \geq 1: t - \sigma(g_k \cdots g_1, x) < 0 \}.   
\end{align}
These stopping times have been studied in \cite{GLP17}, 
where the asymptotics of the probability $\bb P(\tau_{x,t} >n)$ 
as well as the law of the random walk $t + \sigma(g_n \cdots g_1, x)$ 
conditioned on the event $\{ \tau_{x,t} >n \}$ have been determined. 
Our focus in this paper lies in examining the corresponding local limit theorem within the same context.
 
To clarify the concepts required for such a conditioned local limit theorem, 
we turn briefly to the case of random walks in $\bb R$. 
Let $S_n =\sum_{k=1}^n \xi_k $ for $n\geq 1$, where $(\xi_k)_{k\geq 1}$ 
is a sequence of independent and identically distributed real-valued random variables of mean $0$ and finite variance $\upsilon^2 >0$. 
For any $t\in \bb R$, we define the stopping times 
\begin{align*}
& \tau_{t}=\min \{ k \geq 1: t + S_k < 0 \}, \quad \check{\tau}_{t}=\min \{ k \geq 1: t - S_k < 0 \}, 
\end{align*}
where, as before, $\min \emptyset =\infty$. 
Assume  that the random walk $(S_n)_{n\geq 1}$ is non-lattice and that $\xi_1$ has a moment of order $2$.
Then, from the results in \cite{VatWacht09, Don12, GQX23, GX2022IID},  
the following asymptotic holds, which we reformulate in a suitable way: 
for any $t\in \bb R$ and any continuous compactly supported function $h: \bb R \to \bb R$, 
\begin{align} \label{loc limit theorem finite MC-001}
\lim_{n\to\infty} n^{3/2} \bb E \Big(h( t+ S_n ); \tau_{t} >n-1  \Big) 
& = \frac{2V(t)}{\sqrt{2\pi} \upsilon^3} \int_{\bb R} h(t') \check{V}(t')dt'  \notag\\
& = \frac{2V(t)}{\sqrt{2\pi} \upsilon^3} \int_{\bb R} h(t') \rho(dt'). 
\end{align}
Here $V(t)= \lim_{n\to\infty} \bb E (t+S_n; \tau_{t}>n)$, $t\in \bb R,$
is the harmonic function pertaining to the random walk $(S_n)_{n\geq 1}$,
$\check{V}(t) = \lim_{n\to\infty}    \bb E (t +\check S_n;  \check{\tau}_{t}>n)$, $t\in \bb R,$
is the harmonic function pertaining to the reversed random walk $\check S_n = (-S_n)_{n\geq 1}$,
and $\rho$ is the absolutely continuous Radon measure on $\bb R$ defined by $\rho(dt) = \check{V}(t)dt$.
The measure $\rho$ may be thought of as the exit target measure 
in the conditioned local limit theorem \eqref{loc limit theorem finite MC-001}.
Using the following reversal identity 
\begin{align} \label{duality-ident-001}
\int_{\bb R} h(t)  \bb E \Big( t + \check S_n; \check{\tau}_{t}>n \Big) dt = \int_{\bb R_+} t \, \bb E \Big( h(t+S_n); \tau_{t}>n-1 \Big) dt,
\end{align}
the target measure $\rho$ can be redefined in terms of the direct  random walk $(S_n)_{n\geq 1}$ as follows:
\begin{align} \label{alternative form 001}
\int _{\bb R} h(t) \rho(dt)= \lim_{n\to\infty} \int_{\bb R_+} t \, \bb E \Big( h(t+S_n); \tau_{t}>n-1 \Big) dt.
\end{align}

The identities \eqref{alternative form 001} and  \eqref{duality-ident-001} 
serve as the starting point for extending the concept of the target measure 
to the random walk $(\sigma(g_n \cdots g_1, x))_{n\geq 1}$.
 This extension, still denoted by $\rho$, 
 will henceforth be referred to as the target harmonic measure  associated with the random walk 
$(\sigma(g_n \cdots g_1, x))_{n\geq 1}$ killed upon exiting $\bb R_+$.

The construction of the target harmonic measure $\rho$ for the walk $(\sigma(g_n \cdots g_1, x))_{n\geq 1}$
is the main purpose of this paper.
A key step in this construction is to identify a suitable adaptation of the concept 
of a reversed random walk for $\sigma(g_n \cdots g_1, x)$. 
The reversed random walk is defined using an analog of the reversal identity  \eqref{duality-ident-001}, 
as formulated in Lemma \ref{lemma-duality-for-product-001}, 
and is expressed through the relation \eqref{reversed RWfor products-001}.  
One of the main challenges tackled in this work lies in the fact that the resulting reversed random walk 
 no longer constitutes a Markov chain. 
However, using the cohomological identity  \eqref{cohomological eq -vers 002}, 
we interpret this reversed process as a Markov chain perturbed by dependencies on future coordinates.
It is precisely this dependency that renders the construction of the associated harmonic function, as
developed in \cite{DW15, GLP17}, unsuitable for the analysis of the exit time $\tau_{x,t}$. 

To address this issue, we construct a sequence of Markov chains with increasing dimension,
leading to  a corresponding sequence of harmonic functions that depend on the dimension. 
Through a refined approximation procedure, this framework ultimately enables the
construction of the harmonic measure $\rho$.

The use of the term \textit{harmonic} is justified by the fact that $\rho$ satisfies a harmonicity property, 
as demonstrated in Corollary \ref{Thm-measure-rho-v003}. 
Further insights into this property are provided in Subsection \ref{sec-heuristic for harmonic measure}.

The significance of the target measure $\rho$  becomes particularly evident in the formulation of the conditioned local limit theorem, 
which we prove in a separate paper \cite{GQX24}. 
For the reader's convenience, it is restated at the end of the next section, 
following the statement of the main results, see Theorem \ref{Thm-CLLT-cocycle}.

Finally, the theory developed here will also prove to be useful for studying the conditioned 
random walk of the form $\log \|g_n\cdots g_1\|$, where $\|\cdot \|$ is any norm on the group $\bb G$.  
The results related to this theory will be presented in a forthcoming work \cite{GQX25}.

\subsection{Statement of the main results}
We begin by recalling the following existence result \cite[Theorem 2.1]{GLP17}. 

\begin{theorem} \label{Theorem GLPL2017} 
Assume that $\Gamma_{\mu}$ is proximal and strongly irreducible, 
the measure $\mu$ admits an exponential moment and the Lyapunov exponent $\lambda_{\mu}$ is zero. 
Then, for any $x \in \bb P(\bb V)$ and $t \in \bb R$, the following limits exist:
\begin{align} 
& \lim_{n \to \infty} \bb E \Big( t + \sigma(g_n \cdots g_1, x); \tau_{x, t} > n \Big) = V(x, t),  \label{harm function GLPP-001} \\
 & \lim_{n \to \infty} \bb E \Big( t - \sigma(g_n \cdots g_1, x); \check{\tau}_{x, t} > n \Big) = \check{V}(x, t).  \label{harm function GLPP-002}
\end{align}
Moreover, 
uniformly in $x \in \bb P(\bb V)$, it holds 
\begin{align*} 
\lim_{t \to \infty} \frac{V(x, t)}{t}  =\lim_{t \to \infty}\frac{\check{V}(x, t)}{t} = 1.
\end{align*} 
\end{theorem}
In \cite{GLP17} this theorem is stated only for $t \geq 0$, 
but it can be extended to any $t\in \bb R$. 
We refer to Section \ref{sec-results without perturb} for an alternative proof in a more general context. 
The function $V$, which is harmonic for the transition operator of the 
Markov random walk \eqref{def of direct RW-001} conditioned to stay positive, 
is the central object to describe the asymptotic behavior of the probabilities $\bb P (\tau_{x, t} > n)$ 
and the conditioned central limit theorems for the walk $t + \sigma(g_n \cdots g_1, x)$. 
For the corresponding statements, we refer to \cite{GLP17}.
Let us note that these types of results become possible following the groundbreaking work of Denisov and Wachtel \cite{DW15}
on random walks in cones.

In order to state a conditioned local limit theorem, in complement to the harmonic functions in \eqref{harm function GLPP-001} and \eqref{harm function GLPP-002}, 
it is necessary to go further and  prove the existence of a target Radon measure
$\rho$ on the locally compact space $\bb P(\bb V) \times \bb R$ related to the 
reversed random walk for $\sigma(g_n \cdots g_1, x)$.
It turns out that the construction of such a measure $\rho$ is highly technical,
and cannot be derived using the techniques developed in \cite{DW15}, nor from the extensions provided in \cite{GLP17, GLL18Ann}.
 In the related article \cite{GQX23}, we tackled this difficulty 
by employing a sequence of approximated reversals in the context of Birkhoff sums over hyperbolic dynamical systems conditioned to stay positive.

The following theorem states the existence of the target harmonic measure $\rho$  for random walks on linear groups. 

\begin{theorem}\label{Thm-measure-rho}
Assume that $\Gamma_{\mu}$ is proximal and strongly irreducible, 
the measure $\mu$ admits an exponential moment and the Lyapunov exponent $\lambda_{\mu}$ is zero. 
Then, there exist Radon measures $\rho$ and $\check{\rho}$ 
on $\bb P(\bb V) \times \bb R$ such that,  
for any continuous compactly supported function $h$ on $\bb P(\bb V) \times \bb R$, uniformly in $x \in \bb P(\bb V)$, the following limits exist
and are independent of $x$, 
\begin{align}
  \lim_{n \to \infty}  \int_{0}^{\infty} t  \, \bb E \Big( h (g_n \cdots g_1 x, & \  t + \sigma(g_n \cdots g_1, x)); \, \tau_{x, t} > n -1 \Big) dt
  \notag \\
   & \qquad\qquad\quad  = \int_{\bb P(\bb V) \times \bb R} h(x', t') \rho(dx', dt'),  \label{exist of measure rho-001} \\ 
   \lim_{n \to \infty}  \int_{0}^{\infty} t  \,  \bb E \Big( h (g_n \cdots g_1 x,  & \  t - \sigma(g_n \cdots g_1, x) ); \, \check{\tau}_{x, t} > n -1 \Big) dt 
\notag \\ 
 &  \qquad\qquad\quad  = \int_{\bb P(\bb V) \times \bb R} h(x', t') \check{\rho}(dx', dt') \label{exist of measure rho-002}. 
\end{align}
\end{theorem}

We continue with some properties of the measures $\rho$ and $\check \rho$.
The following property states that the measures $\rho$ and $\check \rho$ have absolutely continuous marginals on $\bb R$.

\begin{corollary}\label{Thm-measure-rho-v002}
Assume that $\Gamma_{\mu}$ is proximal and strongly irreducible, 
the measure $\mu$ admits an exponential moment and the Lyapunov exponent $\lambda_{\mu}$ is zero. 
Then, the marginals of the measures $\rho$ and $\check \rho$ on $\bb R$  
are absolutely continuous with respect to the Lebesgue measure with non-decreasing densities. 
\end{corollary}

The next property is the analog of the harmonicity property of the functions $V$ and $\check V$.

\begin{corollary}\label{Thm-measure-rho-v003}
Assume that $\Gamma_{\mu}$ is proximal and strongly irreducible, 
the measure $\mu$ admits an exponential moment and the Lyapunov exponent $\lambda_{\mu}$ is zero. 
Then, for any non-negative measurable function $h$ on $\bb P(\bb V) \times \bb R$, we have 
\begin{align*} 
& \int_{\bb P(\bb V) \times \bb R} h(x, t) \rho(dx, dt) 
= \int_{\bb P(\bb V) \times [0,\infty)} \bb E \Big( h\left(g_1 x,t + \sigma \left(g_1, x \right)  \right)  \Big) \rho(dx,dt), \notag\\ 
& \int_{\bb P(\bb V) \times \bb R} h(x, t) \check \rho(dx, dt) 
= \int_{\bb P(\bb V) \times [0,\infty)} \bb E \Big( h\left(g_1 x,t - \sigma \left(g_1, x \right)  \right)  \Big) \check \rho(dx,dt). 
\end{align*} 
\end{corollary}

Throughout this paper, by measurable functions we mean Borel measurable functions.

Denote by $\nu$ the unique $\mu$-stationary Borel probability measure on $\bb P(\bb V)$, 
see \cite{Boug-Lacr85, BQ16b}. 
The next statement describes the behaviour of the measures $\rho$ and $\check \rho$ at $+\infty$:  
it says that close to $+\infty$ these measures look like the product measure $\nu \otimes t \, dt$.

\begin{corollary}\label{Thm-measure-rho-v004}
Assume that $\Gamma_{\mu}$ is proximal and strongly irreducible, 
the measure $\mu$ admits an exponential moment and the Lyapunov exponent $\lambda_{\mu}$ is zero. 
Then, for any continuous compactly supported function $h$ on $\bb P(\bb V) \times \bb R$, the following limits exist:
\begin{align}
\lim_{t \to \infty} \frac{1}{t} \int_{ \bb P(\bb V) \times \bb R } h(x', t'-t) \rho(dx', dt') = \int_{ \bb P(\bb V) \times \bb R } h(x', t') \nu(dx') dt', \label{limit of rho-001} \\
\lim_{t \to \infty} \frac{1}{t} \int_{ \bb P(\bb V) \times \bb R } h(x', t'-t) \check{\rho}(dx', dt') = \int_{ \bb P(\bb V) \times \bb R } h(x', t') \nu(dx') dt'. \label{limit of rho-002}
\end{align}
As a consequence, the marginal densities $W(t)= \frac{\rho(\bb P(\bb V),dt)}{dt}$ and $\check W(t)= \frac{\check \rho(\bb P(\bb V),dt)}{dt}$  
satisfy 
\begin{align} \label{non-degeneracy of W-001}
\lim_{t \to \infty} \frac{W(t)}{t} = \lim_{t \to \infty} \frac{\check{W}(t)}{t} = 1.  
\end{align}
\end{corollary}

Finally we discuss the behaviour of the measures $\rho$ and $\check \rho$ at $-\infty$.

\begin{corollary}\label{Thm-measure-rho-v005}
Assume that $\Gamma_{\mu}$ is proximal and strongly irreducible, 
the measure $\mu$ admits an exponential moment and the Lyapunov exponent $\lambda_{\mu}$ is zero. 
Then, there exists a constant $c >0$ such that for any $t \leq 0$, one has 
\begin{align}\label{property-W-negative-t}
W(t) \leq c e^{-\alpha |t|} \quad \mbox{and} \quad  \check{W}(t) \leq c e^{-\alpha |t|}, 
\end{align}
where $\alpha>0$ is the exponent from \eqref{Exponential-moment}. 
Moreover, we have 
\begin{align}\label{bound-rho-infty}
\rho \big( \bb P(\bb V) \times (-\infty, 0] \big) \in (0,\infty) 
\quad \mbox{and} \quad  
\check{\rho} \big( \bb P(\bb V) \times (-\infty, 0] \big) \in (0,\infty). 
\end{align}
\end{corollary}

Note that, due to the property  \eqref{non-degeneracy of W-001} or \eqref{bound-rho-infty}, 
the Radon measures $\rho$ and $\check \rho$ are non-zero. 

In a forthcoming paper \cite{GQX24}, we will make use of the harmonic measures $\rho$ and $\check \rho$ constructed above 
to prove the following conditioned local limit theorem for random walks on linear groups.
To formulate the corresponding result we recall that,  for any $x\in \bb P(\bb V)$, the following asymptotic variance 
$$
\upsilon_{\mu}^2 = \lim_{n \to \infty} \frac{1}{n} \bb E  \left[ (\sigma(g_n \cdots g_1, x))^2 \right]
$$ 
exists, does not depend on $x$ and is positive, see for instance \cite[Chapter 13]{BQ16b}.  
The following is Theorem 1.1 from \cite{GQX24}. 

\begin{theorem}\label{Thm-CLLT-cocycle}
Assume that $\Gamma_{\mu}$ is proximal and strongly irreducible, 
the measure $\mu$ admits an exponential moment and the Lyapunov exponent $\lambda_{\mu}$ is zero. 
Then, for any fixed $t\in \bb R$ and for any continuous compactly supported function $h$ on $\bb P(\bb V) \times \bb R$, 
 we have, uniformly in $x \in \bb P(\bb V)$,
\begin{align*}
  \lim_{n \to \infty}  n^{3/2}  \bb E \Big( h (g_n \cdots g_1 x, & \  t + \sigma(g_n \cdots g_1, x));  \tau_{x, t} > n -1 \Big)  \\
&\qquad = \frac{2 V(x, t)}{ \sqrt{2 \pi} \upsilon_{\mu}^3 } \int_{\bb P(\bb V) \times \bb R} h (x',t') \rho(dx',dt'), 
 \notag\\
 \lim_{n \to \infty}  n^{3/2}  \bb E \Big( h (g_n \cdots g_1 x,  & \  t - \sigma(g_n \cdots g_1, x) );   \check{\tau}_{x, t} > n -1 \Big)  \\
&\qquad = \frac{2 \check{V}(x, t)}{ \sqrt{2 \pi} \upsilon_{\mu}^3 } \int_{\bb P(\bb V) \times \bb R} h (x',t') \check{\rho}(dx',dt'). 
\end{align*}
\end{theorem}

All the above results can be stated for the stopping times defined by using large inequality $\leq$ instead of strict inequality $<$.
In these formulations the harmonic functions $V$ and $\check V$ will get replaced by the 
harmonic functions corresponding to the new stopping times, while  
the target harmonic measures $\rho$ and $\check \rho$ will remain the same.  
The respective results can be obtained by the same methods.


\subsection{Heuristics for the harmonicity of the target measure} \label{sec-heuristic for harmonic measure}

In which sense the target measure $\rho$ is harmonic 
 may be a posteriori  explained through the following formalism. 
The data of the measure $\mu$ on $\bb G$ and of the cocycle $\sigma: \bb G \times \bb P(\bb V) \to \bb R$ 
define a Markov chain on the space $\bb P(\bb V) \times \bb R$, which is represented by the operator $P$ defined, 
for any bounded Borel measurable function $h$ on $\bb P(\bb V) \times \bb R$ and for any $(x, t) \in \bb P(\bb V) \times \bb R$, by 
\begin{align*}
P h(x, t) = \int_{\bb G} h \Big( gx, t + \sigma(g, x) \Big) \mu(dg). 
\end{align*}
On $\bb P(\bb V) \times \bb R$ we also define a killing operator $M$ as follows:
for any bounded Borel measurable function $h$ on $\bb P(\bb V) \times \bb R$ and for any $(x, t) \in \bb P(\bb V) \times \bb R$, 
we set 
\begin{align*}
M h(x, t) =  h(x, t) \mathds 1_{\bb R_+} (t). 
\end{align*}
Then the operator $Q = PM$ represents the process $(t + \sigma(g_n \cdots g_1, x))_{n\geq 1}$ conditioned to stay non-negative: for $n \geq 1$, 
it holds 
\begin{align*}
\bb E \Big( h (g_n \cdots g_1 x,  t + \sigma(g_n \cdots g_1, x)); \tau_{x, t} > n \Big) 
= Q^n h(x, t). 
\end{align*}
We can also define the operator $R = MP$ which does not seem to have an obvious probabilistic interpretation. 
With the help of the operators $P$, $Q$ and $R$, one can write   
\begin{align}\label{expect in terms of Q and R-001}
\bb E \Big( h (g_n \cdots g_1 x,  t + \sigma(g_n \cdots g_1, x)); \tau_{x, t} > n -1 \Big) &= Q^{n-1} P h(x, t) = P R^{n-1} h(x, t).
\end{align}
It is formally evident that if, as in Theorem \ref{Thm-CLLT-cocycle}, 
the expectation in \eqref{expect in terms of Q and R-001} is equivalent to 
 $$
 \frac{2 V(x, t)}{ \sqrt{2 \pi} \upsilon_{\mu}^3  n^{3/2}} \int_{\bb P(\bb V) \times \bb R} h d \rho,
 $$ 
then the function $V$ should be $Q$-harmonic, that is, $QV = V$, and the measure $\rho$ should be $R$-harmonic, that is, 
for any bounded Borel measurable function $h$ on $\bb P(\bb V) \times \bb R$, 
\begin{align} \label{def of harminic measure-001}
\int_{\bb P(\bb V) \times \bb R} R h d \rho = \int_{\bb P(\bb V) \times \bb R}  h d \rho. 
\end{align}

Based on the above rationale, it is reasonable to construct the measure $\rho$ 
as the limit of the sequence $(R^*)^n \eta$ where $\eta$ is an initial Radon measure on $\bb P(\bb V) \times \bb R$.
This construction is precisely what is addressed in Theorem \ref{Thm-measure-rho}.
Indeed, let us fix $x \in \bb P(\bb V)$. Choose $\eta=\eta_x := \delta_x \otimes (t \mathds 1_{\bb R_+}(t) dt)$, 
which means that $\eta_x$ is the Radon measure satisfying, 
for any continuous compactly supported function $h$ on $\bb P(\bb V) \times \bb R$, 
\begin{align}\label{def-eta-x}
\int_{ \bb P(\bb V) \times \bb R } h d \eta_x = \int_{\bb R_+} t h(x, t) dt. 
\end{align}
Then, the conclusion of Theorem \ref{Thm-measure-rho} exactly says that the sequence of measures $(R^*)^n \eta_x$ 
converges vaguely to $\rho$, as $n \to \infty$. 
Indeed, using \eqref{expect in terms of Q and R-001}, we have, uniformly over $x \in \bb P(\bb V)$, as $n \to \infty$, 
\begin{align*}
&\int_{\bb R} h d( (R^*)^n \eta_x)  
 = \int_{\bb R} t R^{n} h(x, t) dt  =  \int_{\bb R} t M P R^{n-1} h(x, t) dt  \notag\\
&\qquad  = \int_{0}^{\infty} t \,  \bb E \Big( h (g_n \cdots g_1 x,  t + \sigma(g_n \cdots g_1, x));  \, \tau_{x, t} > n -1 \Big) dt 
 \to  \int_{ \bb P(\bb V) \times \bb R } h d \rho. 
\end{align*}
From this result we deduce in Corollary \ref{Thm-measure-rho-v002} that $\rho$ satisfies \eqref{def of harminic measure-001}.

Note that the analogs of the operators $P, M, Q=PM, R=MP$ can also be considered   
in the setting with random walks based on independent and identically distributed random variables, 
leading to the same conclusions about the measure $\rho$.

\subsection{Extension to the case of flag manifolds}

The methods of this paper can be applied to formulate intrinsic results for random walks on reductive groups.  
Let $\bf G$ be a real connected reductive group. Denote by $K$ a maximal compact subgroup of $G = \bf G(\bb R)$
and by $\bf A$ a maximal $\bb R$-split torus of $\bf G$ so that the Cartan involution of $\bf G$ associated with $K$ equals $-1$ on the Lie algebra $\mathfrak a$
of  $A = \bf A(\bb R)$. Let $\mathfrak a^+ \subset \mathfrak a$ be a Weyl chamber. Then we have the Cartan decomposition $G = K \exp (\mathfrak a^+) K$
and the associated Cartan projection $\kappa: G \to \mathfrak a^+$.

We let $\bf P$ be the unique minimal  $\bb R$-parabolic subgroup of $\bf G$ whose Lie algebra contains the root spaces associated with the elements 
of $\mathfrak a^+$. We have the Iwasawa decomposition $G = K P$, where $P = \bf P(\bb R)$. 
Let $\bf U$ be the unipotent radical of $\bf P$ so that $P = AU$, where $U = \bf U(\bb R)$,
and hence $G = KAU$. 
More precisely, for $g \in G$, the set $KgU$ contains a unique element of $\exp (\mathfrak a)$. 

We also denote by $\mathcal P = G/P$ the flag manifold of $G$, that is, the set of minimal  $\bb R$-parabolic subgroups of $\bf G$, 
and by $\xi_0$ the unique fixed point of $P$ in $\mathcal P$. 
The set $\mathcal P$ is a compact homogeneous space of $G$.
For $g \in G$ and $\xi \in \mathcal P$, choose $k \in K$ such that $\xi = k \xi_0$,
and denote by $\sigma(g, \xi)$ the unique element of $\mathfrak a$ such that  
\begin{align*}
\exp(\sigma(g, \xi)) \in K g kU. 
\end{align*}
The map $\sigma: G \times \mathcal P \to \mathfrak a$ 
is a smooth cocycle which is usually called the Iwasawa cocycle.

Let $\mu$ be a Borel probability measure on $G$. Assume that the first moment of $\mu$ is finite, meaning that $\int_{G} \| \kappa(g) \| \mu(dg) < \infty$
for some norm $\| \cdot \|$ on the vector space $\mathfrak a$. 
Then, the limit $\lim_{n \to \infty} \frac{1}{n} \bb E \kappa(g_n \cdots g_1)$ exists and is called the Lyapunov vector $\lambda_{\mu} \in \mathfrak a^+$. 
Let $\phi$ be a linear functional on $\mathfrak a$ such that $\phi(\lambda_{\mu}) = 0$. 
Then for $t \in \bb R$ and $\xi \in \mathcal P$, we define the following stopping time 
\begin{align*}
\tau_{\xi, t}  = \min \{ k \geq 1: t + \phi( \sigma(g_k \cdots g_1, \xi) ) < 0 \}. 
\end{align*}
Denote by $\Gamma_{\mu}$ the subsemigroup of $G$ spanned by the support of $\mu$. 

One can show the following analog of Theorem \ref{Theorem GLPL2017}.

\begin{theorem} \label{Theorem GLPL2017-reductive} 
Assume that $\Gamma_{\mu}$ is Zariski dense in $G$,  
the measure $\mu$ admits an exponential moment (i.e. $\int_{G} e^{\alpha \|\kappa(g) \|} \mu(dg) < \infty$ for some $\alpha >0$) 
and $\phi(\lambda_{\mu}) = 0$. 
Then, for any $\xi \in \mathcal P$ and $t \in \bb R$, the following limit exists:
\begin{align} 
& \lim_{n \to \infty} \bb E \Big( t + \phi( \sigma(g_k \cdots g_1, \xi) ); \tau_{\xi, t} > n \Big) = V(\xi, t). 
\end{align}
Moreover, uniformly in $\xi \in \mathcal P$, it holds $\lim_{t \to \infty} \frac{V(\xi, t)}{t} = 1.$ 
\end{theorem}

The methods of this paper also allow to get the following analog of Theorem \ref{Thm-measure-rho}.

\begin{theorem}\label{Thm-measure-rho-reductive}
Assume that $\Gamma_{\mu}$ is Zariski dense in $G$,  
the measure $\mu$ admits an exponential moment 
and $\phi(\lambda_{\mu}) = 0$. 
Then, there exists a Radon measure $\rho$ 
on $\mathcal P \times \bb R$ such that,  
for any continuous compactly supported function $h$ on $\mathcal P \times \bb R$, uniformly in $\xi \in \mathcal P$, the following limit exists 
and is independent of $\xi$, 
\begin{align*}
  \lim_{n \to \infty}  \int_{0}^{\infty} t  \bb E \Big( h (g_n \cdots g_1 \xi,  t + \phi( \sigma(g_n \cdots g_1, \xi) ) );  \tau_{\xi, t} > n -1 \Big) dt
 = \int_{\mathcal P \times \bb R} h(\xi', t') \rho(d\xi', dt'). 
\end{align*}
The marginal  of the measure $\rho$ on $\bb R$ is absolutely continuous with respect to the Lebesgue measure
with non-decreasing density. 

Moreover, for any continuous compactly supported function $h$ on $\mathcal P \times \bb R$, the following limit exists:
\begin{align*}
\lim_{t \to \infty} \frac{1}{t} \int_{ \mathcal P \times \bb R } h(\xi', t'-t) \rho(d\xi', dt') = \int_{ \mathcal P \times \bb R } h(\xi', t') \nu(d\xi') dt', 
\end{align*}
where $\nu$ is the unique $\mu$-stationary probability measure on $\mathcal P$. 
In particular, the marginal density $W(t) = \frac{\rho(\mathcal P, dt)}{dt}$
satisfies $\lim_{t \to \infty} \frac{W(t)}{t} =  1$. 
\end{theorem}

The measure $\rho$ satisfies the same harmonicity property as in Corollary \ref{Thm-measure-rho-v003}. 

We will not prove these results explicitly, but they can be obtained in the same way as in Theorems \ref{Theorem GLPL2017} and \ref{Thm-measure-rho},
by using the language of \cite{BQ16b}. 

\subsection{Extension to the case of local fields}

Theorems \ref{Theorem GLPL2017} and \ref{Thm-measure-rho} are stated for random walks 
with values in the linear group $\mathrm{GL}(d,\bb R)$ over the field of real numbers $\bb R$.
However, they can be directly extended to the case of random walks with values in the linear group $\mathrm{GL}(d,\bb K)$, 
where $\bb K$ is a local field, that is a locally compact topological field.   

Indeed, when working over the complex field $\bb C$, we can consider the cocycle associated with the data of a Hermitian norm
on a given finite-dimensional complex vector space $\bb V$ isomorphic to $\bb C^d$.
When working over a non-Archimedean local field $\bb K$,  we can consider the cocycle associated with the data of an  
ultrametric norm on a given finite-dimensional $\bb K$-vector space $\bb V$ isomorphic to $\bb K ^d$.


\subsection{Proof strategy and organization of the paper} 
The main results of the paper are stated in Section \ref{sec-Introduction and results}, precisely in
Theorem \ref{Thm-measure-rho}. Their proofs will rely on 
studying random walks conditioned to stay non-negative with some perturbations depending on the future.
The conclusion of the latter study is summarized in Theorem \ref{Pro-Appendix-Final2-Inequ} of Section \ref{sec invar func},
which will be formulated in an abstract framework where some general group acts on a general locally compact space.

In Section \ref{Sec-Approximation-properties}, we introduce a random walk featuring an ideal perturbation 
depending on the entire future, which will later be employed to define the target harmonic measure $\rho$. 
Such perturbed random walks will be studied later in an abstract setting in Section \ref{sec invar func}.
At this point, our aim is to verify that our concrete example satisfies the assumptions  
of the general Theorem \ref{Pro-Appendix-Final2-Inequ} -- specifically, 
 that the ideal perturbation can be effectively approximated by perturbations 
depending on a finite number of coordinates. 
We conclude the section by addressing a comparable issue concerning a random walk with a perturbation varying with $n$.

In Section \ref{Sec-harmonic-measure}, we begin by  establishing the existence of the target harmonic measure $\rho$ 
through the utilization of the ideally perturbed random walk introduced in Section \ref{Sec-Approximation-properties}. 
The proof relies on the two-sided approximations articulated in Theorem \ref{Pro-Appendix-Final2-Inequ}. 
In the latter part of the section, 
we define an appropriate reversed random walk for $\sigma(g_n \cdots g_1, x)$  
through a reversal identity similar to \eqref{duality-ident-001}, 
with the precise formulation given in Lemma \ref{lemma-duality-for-product-001}.  
By applying to the reversed walk approximation techniques analogous to those in the first part of the section, 
we derive the conclusions stated in Theorem \ref{Thm-measure-rho}.

The subsequent sections, namely Sections \ref{sec invar func}, \ref{sec-results without perturb}, 
\ref{sec quasi-increasingness}, and \ref{sec-quasi-decreasing property}, 
are devoted to establishing Theorem \ref{Pro-Appendix-Final2-Inequ}. 
This theorem addresses random walks with perturbations 
that lend themselves to approximations by functions depending on a finite set of coordinates. 
For the proof, we develop an approach that involves replacing these walks with 
suitably chosen Markov chains of increasing dimension.


\subsection{Incentive and related work} 
The study of the conditioned limit theorems for random walks with independent and identically distributed jumps on the real line 
has been initiated by Spitzer \cite{Spitzer} and Feller \cite{Feller} and  has attracted the interest of many authors
\cite{Bolth, Igle74, Eppel-1979, Doney85, BertDoney94, Carav05}.  
Random walks in cones have been studied intensely in \cite{Den Wacht 2008, DW15, DW19, PW21, DRTW22, DW24}.
For a historical overview and a comprehensive list of references, we refer readers to \cite{GQX23, GX2022IID}.

The case of sums of dependent random variables is considerably less explored. 
In the setting of additive functionals associated with finite state Markov chains, 
a conditioned local limit theorem has been established in \cite{GLL20}. 
The methodology highlighted in \cite{GLL20} revolves around the existence of a reversed Markov chain, 
which is intricately linked to the original chain via a reversal identity. 
This connection facilitated the formulation of a dual harmonic function and, 
consequently, the implicit demonstration of the existence of the measure $\rho$.
While such an approach is not directly applicable to the general case of products of random matrices, 
it proves effective in specific instances -- particularly when dealing with matrices 
possessing a density with respect to the Haar measure on $\textrm{GL}(d, \mathbb{R})$. 
This specialized case has been explored in \cite{GXM2024ExtrPos} 
where the underlying concept aligns with that in \cite{GLL20}.

A closely related issue concerning positive matrices has been recently addressed in \cite{PP23, GX2024CLLTposMatr}. However, the proof methods employed in these papers fell short of providing the exact asymptotic. Instead, they yielded only two-sided bounds for the local probabilities.

Exact asymptotics in conditioned local limit theorems have been studied within the 
framework of hyperbolic dynamical systems in \cite{GQX23}. 
This work specifically addresses the subshift of finite type setting. 
Notably, this research revealed that for dependent random walks characterized by dependencies extending beyond the Markov type, the traditional harmonic function needs to be replaced by a more comprehensive entity known as the harmonic measure.
Our paper can be viewed as an adaptation of this methodology to the setting with products of random matrices. 
Along with this, our approach crucially employs the techniques developed recently in the 
works  such as \cite{Den Wacht 2008, DW15, GLP17, GLL18Ann}.  
  
The comprehension of conditioned local limit theorems for products of random matrices
is crucial in addressing various problems. 
For instance, it is instrumental in examining random walks on affine groups in the critical case \cite{BBE97, BPP20, ABP24},
the reflected random walks \cite{Lalley95, EsPeRa2013},    
 multitype branching processes in random environment \cite{LPPP18}, 
branching random walks on the linear group \cite{BDGM14, Mentem16, GXM2024ExtrPos}.

 The problem of finding the suitable target Radon measure to achieve a  local limit theorem, 
 similar to the issue encountered in our scenario, also emerges in the study 
 of non-centered random walks on nilpotent Lie groups.
 Relevant insights into this problem can be found in \cite{Bre05, Hou19, DH21, BB23}.


\section{Approximation properties for perturbations}\label{Sec-Approximation-properties}

\subsection{A random walk with the ideal perturbation} 
In this section, we introduce the dual random walk and a perturbation function denoted as $f$, 
which play important roles in constructing the Radon measure $\rho$ featured 
in the primary result of the paper -- Theorem \ref{Thm-measure-rho}. 
Note that the perturbation function $f$ emerges as the limit of a sequence of perturbations $f_n^{x,m}$, $m\geq n\geq 1$, which come into play when investigating the reversed random walk in Section \ref{Sec-harmonic-measure} below. 
For this reason $f$ will be called the ideal perturbation.

Without loss of generality we shall assume in the following that   
 $\Omega = \bb G^{\bb N^*}$, that $\Omega$ is equipped with the Borel $\sigma$-algebra $\scr A$ and with the  
probability measure $\bb P = \mu^{\otimes \bb N^*}$. 
A typical element of $\Omega$ is written as $\omega = (g_1, g_2, \ldots)$.
Then, the sequence of coordinate maps, $\omega \mapsto g_k$, $k=1,2,\ldots$ on the probability space 
$(\Omega,\mathscr A, \bb P)$ forms a sequence of independent and identically distributed elements of $\bb G$ with law $\mu$.
Let $T: \Omega \to \Omega$ be the shift map $\omega = (g_1, g_2, \ldots) \mapsto T \omega = (g_2, g_3, \ldots)$.
We also introduce the shift map $\widetilde T $ on $\Omega\times \bb P(\bb V^*)$, defined as follows: 
for $\omega=(g_1, g_2, \ldots) \in \Omega$ and $y \in \bb P(\bb V^*)$, 
\begin{align*} 
\widetilde T(\omega, y)= \left( (g_2, g_3, \ldots), g^{-1}_{1} y \right).
\end{align*}

Recall that we have chosen a Euclidean norm $\| \cdot \|$ on $\bb V$.  
We equip the projective space $\bb P(\bb V)$ with the sine distance 
$d(x, x')= \frac{\| v \wedge v' \| }{\| v \| \| v' \| }$, where $x=\bb R v \in \bb P(\bb V)$ and $ x'=\bb R v' \in \bb P(\bb V)$.
Consider the dual vector space $\bb V^*$ of $\bb V$ and denote by $\bb P(\bb V^*)$ the projective space of $\bb V^*$. 
We let $\bb G$ act on $\bb V^*$ and $\bb P(\bb V^*)$ in the standard way: 
for $g \in \bb G$ and $\varphi \in \bb V^*$, the action $g \varphi$ is defined as the linear functional that acts on $v \in \bb V$ by
\begin{align*}
g \varphi (v) = \varphi (g^{-1} v). 
\end{align*}
We also equip $\bb V^*$ with the Euclidean norm dual to the norm $\|\cdot\|$ on  $\bb V$, 
defined as follows: for $\varphi \in \bb V^*$, 
\begin{align*}
\| \varphi \| = \sup_{v \in \bb V \smallsetminus \{0\}} \frac{| \varphi(v) |}{\|v\|}. 
\end{align*}
Define a cocycle $\sigma^*: \bb G \times \bb P(\bb V^*) \to \bb R$ by:
for any $g \in \bb G$ and $y = \bb R \varphi \in \bb P(\bb V^*)$, 
\begin{align}\label{dual cocycle-001}
\sigma^*(g, y) = \log \frac{\| g \varphi \|}{\| \varphi \|}. 
\end{align}

The proof of Theorem \ref{Thm-measure-rho} relies on the study of the perturbed dual random walk 
$(\tilde S_n(\omega, y) )_{n\geq 1}$ defined as follows.
Let $f$ be a real-valued measurable function on $\Omega \times \bb P(\bb V)$, referred to as the ideal perturbation. 
For any $\omega=(g_1,g_2,\ldots)\in \Omega$, $y\in \mathbb{P}(\bb V)$ and $n\geq 1$, we set
\begin{align} \label{introduc of perturb RW with f-001}
\tilde S_n(\omega, y) 
= -\sigma^*(g_n^{-1} \cdots g_1^{-1}, y) 
+ f\circ \widetilde T^n (\omega,y) - f(\omega,y).
\end{align}
This walk can be viewed as the random walk $(\sigma^*(g_n^{-1} \cdots g_1^{-1}, y))_{n\geq 1}$ altered by the functions $f\circ \widetilde T^n - f.$
Note that the function $\omega\mapsto f(\omega,y)$ depends on the future coordinates of $\omega$,
which introduces the main challenge in analyzing the properties of the walk \eqref{introduc of perturb RW with f-001}. 
The study of random walks with perturbations depending on the future is at the core of the abstract framework established in Theorem  \ref{Pro-Appendix-Final2-Inequ}. 

Below, we precisely define the perturbation function $f$ that will be appropriate for proving the existence of the target harmonic measure $\rho$. 
We refer to this function as the \textit{ideal perturbation}.
Moreover, we will show that the ideal perturbation $f$ satisfies an approximation property 
that is consistent with the assumptions of Theorem \ref{Pro-Appendix-Final2-Inequ}. 

To define the ideal perturbation $f$ which is appropriate for our study, 
we use classical results on products of random matrices, see \cite{Boug-Lacr85, BQ16, BQ16b}. 
Recall that $\Gamma_{\mu}$ is the closed subsemigroup of $\bb G$ spanned by the support of $\mu$. 
Since $\Gamma_{\mu}$ contains some proximal element and the action of $\Gamma_{\mu}$ on $\bb V$ is strongly irreducible,   
the space $\bb P(\bb V)$ carries a unique $\mu$-stationary probability measure $\nu$. 
By a classical result of Furstenberg \cite{Fur63}, 
there exists a unique measurable map $\xi$ from $\Omega$ into $\bb P(\bb V)$ with the following equivariance property:
for $\bb P$-almost every $\omega = ( g_1, g_2, \ldots) \in \Omega$, 
\begin{align}\label{equivariance-xi}
\xi(\omega) = g_1 \xi(T\omega),  
\end{align}
and hence, by iteration, for any $p\geq 1$,
\begin{align}\label{equivariance-xi-002}
\xi(\omega) = g_1\cdots g_p \xi(T^p \omega). 
\end{align}
Moreover, the law of the random point $\xi$ in $\bb P(\bb V)$ is the stationary measure $\nu$.
The map in \eqref{equivariance-xi} will play an important role in the subsequent analysis.

The duality between $\bb V^*$ and $\bb V$ allows to define a function $\delta$ on the set 
\begin{align*} 
\Delta: = \big\{ (x, y) \in \bb P(\bb V) \times \bb P(\bb V^*):  x = \bb R v,  \  y = \bb R \varphi,  \  \varphi(v) \neq 0  \big\}. 
\end{align*}
For any $(x, y) \in \Delta$ with $x = \bb R v \in \bb P(\bb V)$ and $y = \bb R \varphi \in \bb P(\bb V^*)$, let
\begin{align} \label{def of delta func-001}
\delta(x, y) = - \log \frac{ |\varphi(v)| }{ \| \varphi \| \| v\| }  \geq 0. 
\end{align}
The function $\delta$ and the cocycles $\sigma$ and $\sigma^*$ are related by the cohomological formula:
for $g \in \bb G$ and $(x, y) \in \Delta$, 
\begin{align*}
\delta(gx, gy) = \delta(x, y) + \sigma(g, x) + \sigma^*(g, y). 
\end{align*}
In other words,
\begin{align} \label{cohomological eq -vers 002}
\sigma(g, x) - \delta(gx, y)  = \sigma^*(g^{-1}, y) - \delta (x, g^{-1}y). 
\end{align}
With the notation introduced above, the ideal perturbation function $f$ is defined as follows: 
for $\omega \in \Omega$ and $y\in \bb P(\bb V^*)$,  
\begin{align} \label{perturbed function in bb Y-001}
f(\omega, y) = \delta ( \xi (\omega), y ). 
\end{align}
Note that this function is undefined 
when $\xi(\omega)$ lies in the projective hyperplane in $\bb P(\bb V)$ that is orthogonal to $y$, 
because the denominator in the definition of $\delta$ becomes zero. 
However, 
the law of the random point $\xi \in \bb P(\bb V)$ is the unique $\mu$-stationary measure $\nu$ on $\bb P(\bb V)$,
which assigns zero mass to such projective hyperplanes. Consequently, 
 for any $y\in \bb P(\bb V^*)$, the function $\omega \mapsto f(\omega, y)$ 
is defined almost everywhere with respect to $\bb P$ on $\Omega$.

\subsection{Uniform approximations for the ideal perturbation}\label{sec-Uniform approximations}

For $p \geq 1$, we denote by $\scr A_p$ the $\sigma$-algebra on $\Omega$ 
spanned by the random elements $g_1,\ldots,g_p$, 
and we set $\scr A_0 = \{ \emptyset, \Omega \}$ for the trivial $\sigma$-algebra. 
Henceforth, the symbols $c$ and $C$ denote positive constants whose values may change from line to line. 

We begin by acknowledging that the function $f$ possesses an exponential moment (refer to \cite[Theorem 14.1]{BQ16b}):
there exists a constant $\alpha >0$ such that 
\begin{align}\label{Regularity-nu}
\sup_{y \in \bb P(\bb V^*)} \int_{\Omega}  e^{ \alpha f(\omega, y) } \bb P(d \omega)
 =  \sup_{y \in \bb P(\bb V^*)} \int_{\bb P(\bb V)}  e^{ \alpha \delta(x, y) } \nu(dx)  < \infty. 
\end{align}
Note that our notation differs from that in \cite{BQ16b}: 
what we denoted here by $\delta(x, y)$ would be $-\log \delta(x, y)$ there. 

The subsequent result asserts that the ideal perturbation function $f$, which depends on the entire future, 
can itself be very well approximated by functions that rely solely on a finite number of coordinates.

\begin{proposition} \label{Prop-projection approx-001}
Assume that $\mu$ has a finite exponential moment and that $\Gamma_{\mu}$ is proximal and strongly irreducible.
Then the function $f$ defined by \eqref{perturbed function in bb Y-001} 
satisfies  the following approximation property: 
 there exist constants $\alpha, \beta, c >0$ 
 such that, for any $p \geq 1$, 
\begin{align} \label{approxim rate for gp-001-for bb Y} 
\sup_{y\in\bb P(\bb V^*)} \int_{\Omega} e^{\alpha |f(\omega, y) - \bb E (f(\cdot, y) | \scr A_p)(\omega) |} \bb P(d \omega) \leq 1+ ce^{-\beta p}.
\end{align}
\end{proposition}

The proof of this proposition will rely on the following concentration estimate for the trajectories 
of random walks on the projective space. 

\begin{lemma}\label{Lem-random-measure}
Assume that $\mu$ has a finite exponential moment and that $\Gamma_{\mu}$ is proximal and strongly irreducible.
Then there exist constants $a, b, c >0$ and $C>0$ such that for any $x \in \bb P(\bb V)$, any Borel probability measure $\varrho$
on $\bb P(\bb V)$ and any $n \geq 1$, 
\begin{align*}
\bb P \left( \varrho \Big\{ x' \in \bb P(\bb V):  d(g_n \cdots g_1 x, g_n \cdots g_1 x') > e^{ - an }  \Big\} > e^{ - bn } \right)  \leq C e^{-cn}. 
\end{align*}
\end{lemma}

\begin{proof}
By \cite[Proposition 14.3]{BQ16b} (see equation (14.6)), we may find constants $a, b, C >0$ such that for any $x \in \bb P(\bb V)$
and any $n \geq 1$, 
\begin{align*}
\bb P \left( d (g_n \cdots g_1 x, x_{g_n \cdots g_1}^M) > e^{-an} \right) \leq C e^{-bn},
\end{align*}
where, for $g \in \bb G$, $x_g^M \in \bb P(\bb V)$ is the density point of $g$, see page 224 of \cite{BQ16b}. 
By the triangle inequality, we get that, for any $x, x' \in \bb P(\bb V)$, 
\begin{align*}
\bb P \Big( d (g_n \cdots g_1 x, g_n \cdots g_1 x') > e^{-an} \Big) \leq 2C e^{-bn}.  
\end{align*}
Using Fubini's theorem and integrating over $x'$ yield
\begin{align*}
& \bb E \Big( \varrho \Big\{ x' \in \bb P(\bb V):  d(g_n \cdots g_1 x, g_n \cdots g_1 x') > e^{ - an }  \Big\} \Big)  \notag\\
& = \int_{\bb P(\bb V)} \bb P \Big( d (g_n \cdots g_1 x, g_n \cdots g_1 x') > e^{-an} \Big) \varrho(dx')  \notag\\
& \leq 2C e^{-bn}. 
\end{align*}
By Chebyshev's inequality, we obtain
\begin{align*}
\bb P \left( \varrho \Big\{ x' \in \bb P(\bb V):  d(g_n \cdots g_1 x, g_n \cdots g_1 x') > e^{ - an }  \Big\} > e^{ - \frac{b}{2}n } \right)
\leq  2C e^{ - \frac{b}{2}n },  
\end{align*}
completing the proof of the lemma. 
\end{proof}

Before proving Proposition \ref{Prop-projection approx-001}, 
we use Lemma \ref{Lem-random-measure} to derive the following corollary, which will also be useful to check the assumptions of 
Theorem \ref{Pro-Appendix-Final2-Inequ}.

\begin{corollary} \label{corollary f_y-001}
Assume that $\mu$ has a finite exponential moment and that $\Gamma_{\mu}$ is proximal and strongly irreducible.
Let $\varphi$ be a H\"older continuous function on $\bb P(\bb V)$, 
and define $\theta(\omega) = \varphi(\xi(\omega))$ for $\omega \in \Omega$.
Also, for $x \in \bb P(\bb V)$, $m \geq 1$ and $\omega=(g_1,g_2,\ldots) \in \Omega$, 
define $\theta^{x, m}(\omega) = \varphi(g_1 \cdots g_m x)$.
Then, the functions $\theta$ and $\theta^{x, m}$  satisfy the following approximation property: 
there exist constants $c>0$ and $\beta >0$ such that, for any $p\geq 1$,
\begin{align} \label{approx property of theta-in bb Y-001}
\| \theta -  \bb E (\theta | \scr A_p) \|_{1} = \bb E | \theta -  \bb E (\theta | \scr A_p) | \leq  c e^{-\beta p }
\end{align}
and, for any $m \geq p$ and $x \in \bb P(\bb V)$, 
\begin{align}\label{approx property of theta-in bb Y-002}
\| \theta^{x, m} -  \bb E (\theta^{x, m} | \scr A_p)  \|_{1} = 
\bb E | \theta^{x, m} -  \bb E (\theta^{x, m} | \scr A_p) | \leq  c e^{-\beta p }. 
\end{align}
\end{corollary}

\begin{proof}
As the function $\varphi$ is H\"older continuous, there exist constants $\alpha, C >0$ such that, for any $x,x' \in \bb P(\bb V)$, 
\begin{align*} 
| \varphi(x) - \varphi(x') | \leq C d(x,x')^{\alpha}.
\end{align*}
Taking the conditional expectation in \eqref{equivariance-xi-002} gives, for $\bb P$-almost all $\omega=(g_1,g_2,\ldots) \in \Omega$, 
\begin{align*}
\bb E (\theta | \scr A_p)(\omega) = \int_{\bb P(\bb V)}  \varphi(g_1 \cdots g_p x) \nu(dx). 
\end{align*}
Letting $a, b, c >0$ be as in Lemma \ref{Lem-random-measure}, we obtain that, for any $p \geq 1$, 
\begin{align*}
& \bb E \left| \theta- \bb E (\theta | \scr A_p) \right|  \notag\\
&=\int_{\bb P(\bb V)} \bb E \left| \varphi (g_p \cdots g_1 x)- \int_{\bb P(\bb V)} \varphi (g_p \cdots g_1 x') \nu(dx')  \right| \nu(dx) \\
& =  \int_{\bb P(\bb V)} \bb E \left| \int_{\bb P(\bb V)} \mathds 1_{ \{ d(g_p \cdots g_1 x, g_p \cdots g_1 x') \leq e^{-a p} \} }  
  \left( \varphi (g_p \cdots g_1 x) - \varphi (g_p \cdots g_1 x')\right) \nu(dx')  \right| \nu(dx) \\
& \quad +  \int_{\bb P(\bb V)} \bb E \left| \int_{\bb P(\bb V)} \mathds 1_{ \{ d(g_p \cdots g_1 x, g_p \cdots g_1 x') > e^{-a p} \} }  
 \left( \varphi (g_p \cdots g_1 x) - \varphi (g_p \cdots g_1 x')\right) \nu(dx')  \right| \nu(dx) \\
&\leq C e^{-a \alpha p}  
+ 2 \|\varphi\|_{\infty} 
\int_{\bb P(\bb V)}  \bb E \;  \nu \left\{ x' \in \bb P(\bb V):  d(g_p \cdots g_1 x, g_p \cdots g_1 x') > e^{ - ap }  \right\} \nu(dx) \\
&\leq C e^{-a \alpha p}  + 2 \|\varphi\|_{\infty}  e^{-b p}  \\
& \quad+ 2 \|\varphi\|_{\infty} \int_{\bb P(\bb V)} \bb P \left( \nu \left\{ x' \in \bb P(\bb V):  d(g_p \cdots g_1 x, g_p \cdots g_1 x') > e^{ - ap }  \right\} > e^{ - bp } \right)  \nu(dx) \\
&\leq C e^{-a \alpha p}  + 2 \|\varphi\|_{\infty}  e^{-b p}  + C e^{-cp},  
\end{align*}
where in the last inequality we apply Lemma \ref{Lem-random-measure} with $\varrho = \nu$. This proves \eqref{approx property of theta-in bb Y-001}.

In the same way, for $m \geq p$ and any $x \in \bb P(\bb V)$, we have, $\bb P$-almost surely, 
\begin{align*}
\bb E (\theta^{x, m} | \scr A_p) = \int_{\bb G^{m-p}}  \varphi(g_1 \cdots g_p g'_{p+1} \cdots g'_{m} x) \mu(dg'_{p+1}) \ldots \mu(dg'_{m}), 
\end{align*}
which gives 
\begin{align*}
& \bb E | \theta^{x, m} -  \bb E (\theta^{x, m} | \scr A_p) |   \notag\\
& =\int_{\bb G^m}    \left| \varphi (g_1 \cdots g_m x)
- \int_{\bb G^{m-p}} \varphi (g_1 \cdots g_p g'_{p+1} \cdots g'_{m} x) \mu(dg'_{p+1}) \ldots \mu(dg'_{m})  \right| \notag\\
&\qquad\qquad \mu(dg_{1}) \ldots \mu(dg_{m}). 
\end{align*}
Following the same proof as above, we get \eqref{approx property of theta-in bb Y-002}.  
 \end{proof}

\begin{proof}[Proof of Proposition \ref{Prop-projection approx-001}]
By \eqref{equivariance-xi-002} and \eqref{perturbed function in bb Y-001}, for any $y \in \bb P(\bb V^*)$, we have, $\bb P$-almost surely, 
\begin{align*}
\bb E (f(\cdot, y) | \scr A_p) = \int_{\bb P(\bb V)}  \delta(g_1 \cdots g_p x', y) \nu(dx'). 
\end{align*}
As mentioned before, by \eqref{Regularity-nu}, there exists a constant $\alpha_0 > 0$ such that 
\begin{align} \label{new expmoment001}
\sup_{y \in \bb P(\bb V^*)} \int_{\bb P(\bb V)}  e^{ \alpha_0 \delta(x', y) } \nu(dx')  < \infty. 
\end{align}
This yields that, for any $\alpha \in (0,\alpha_0/3]$, 
\begin{align} \label{eq-projappr-001}
 \int_{\Omega} e^{\alpha |f(\omega, y) - \bb E (f(\cdot, y) | \scr A_p)(\omega) |} \bb P(d \omega)  
 = \bb E  \left(  \int_{\bb P(\bb V)}  \exp \left( \alpha \left| \bar{\delta}(G_px, y)  \right|  \right)  \nu(dx) \right), 
\end{align}
where, for short, we denote $G_p = g_p \cdots g_1$ and 
\begin{align*}
\bar{\delta}(G_px, y) = \delta(G_px, y) - \int_{\bb P(\bb V)}  \delta(G_px', y) \nu(dx'). 
\end{align*}
Note that, by H\"older's inequality and \eqref{new expmoment001},  
the expectation on the right-hand side of \eqref{eq-projappr-001} can be shown to be finite.
Therefore, by using Fubini's theorem,  we get,  for any $\alpha \in (0,\alpha_0/3]$,
\begin{align*}
\int_{\Omega} e^{\alpha |f(\omega, y) - \bb E (f(\cdot, y) | \scr A_p)(\omega) |} \bb P(d \omega) 
 =   \int_{\bb P(\bb V)}  \bb E \exp \left( \alpha \left| \bar{\delta}(G_px, y)  \right|  \right)  \nu(dx). 
\end{align*}
We fix $\ee \in (0, a/2)$, where $a$ is as in Lemma \ref{Lem-random-measure}.  
Set $A_{p,x,y} = \{ \delta(G_px, y) > \ee p \}$.
By \cite[Proposition 14.3]{BQ16b}, we have uniformly in $x \in \bb P(\bb V)$ and $y \in \bb P(\bb V^*)$,
\begin{align} \label{exp bound Apx-001}
\bb P \left( A_{p, x, y} \right) \leq C e^{ - cp }. 
\end{align}
On the set $A_{p, x, y}^c$, we have 
\begin{align*}
 \left| \bar{\delta}(G_px, y)  \right| 
& \leq  \left|   \int_{ B_{p, x}(\omega) } \left( \delta(G_px, y) -  \delta(G_px', y) \right) \nu(dx')  \right|  \notag\\
& \quad   + \left|   \int_{ B_{p, x}(\omega)^c } \left( \delta(G_px, y) -  \delta(G_px', y) \right) \nu(dx')  \right|, 
\end{align*}
where $B_{p, x}(\omega)$ is the random set $\{x' \in \bb P(\bb V): d(G_px, G_px') \leq e^{-ap} \}$. 
Here and in the rest of this proof, we omit $\omega$ in $G_p(\omega)$ for short. 
On $A_{p,x,y}^c$, for $x' \in B_{p, x}(\omega)$, we have, by the mean value theorem, 
\begin{align*}
|\delta( G_px, y ) - \delta( G_px', y ) | \leq C e^{\delta (G_px, y)} e^{- ap}
\leq  C e^{ \ee p } e^{- ap} \leq  C e^{- \frac{a}{2} p}. 
\end{align*}
Hence, we obtain that on $A_{p,x,y}^c$,
\begin{align*}
\left| \bar{\delta}(G_px, y)  \right|  
& \leq  C  e^{- \frac{a}{2} p}
+ \left|   \int_{ B_{p, x}(\omega)^c } \left( \delta(G_px, y) -  \delta(G_px', y) \right) \nu(dx')  \right|   \notag\\
& \leq   C e^{- \frac{a}{2} p}  +  \left( \nu \left(  B_{p, x}(\omega)^c \right) \right)^{1/2} 
  \left( \int_{ \bb P(\bb V) } \left( \delta(G_px, y) -  \delta(G_px', y) \right)^2 \nu(dx') \right)^{1/2}  \notag\\
& \leq  C e^{- \frac{a}{2} p} +  \left( \nu \left(  B_{p, x}(\omega)^c \right) \right)^{1/2} 
    \left[ \delta(G_px, y) +  \left(  \int_{ \bb P(\bb V) }   \delta(G_px', y)^2 \nu(dx') \right)^{1/2}  \right]   \notag\\
& \leq   C e^{- \frac{a}{2} p} +  \left( \nu \left(  B_{p, x}(\omega)^c \right) \right)^{1/2} 
    \left[ \ee p +  \left( \int_{ \bb P(\bb V) }   \delta(G_px', y)^2 \nu(dx') \right)^{1/2}  \right]. 
\end{align*}
We set $E_{p,x} = \{ \omega \in \Omega: \nu (  B_{p, x}(\omega)^c ) > e^{ -bp } \}$. 
By Lemma \ref{Lem-random-measure}, we have that 
\begin{align} \label{exp bound Epx-001}
\bb P ( E_{p,x} ) \leq Ce^{-cp}. 
\end{align}
Therefore, for any $\alpha \in (0,\alpha_0/3]$, 
\begin{align*}
& \bb E \exp \left\{ \alpha \left| \bar{\delta}(G_px, y)  \right|  \right\}  
 \leq  \bb E  \mathds 1_{ E_{p,x} \cup A_{p,x,y} } \exp \left\{ \alpha \left| \bar{\delta}(G_px, y) \right|  \right\}  \notag\\
 & \quad +  \bb E \mathds 1_{ E^c_{p,x} \cap A^c_{p,x} } \exp \left( \alpha C e^{- \frac{a}{2} p}  +  \alpha e^{- b p/2} 
    \left[ \ee p +  \left(  \int_{ \bb P(\bb V) }   \delta(G_px', y)^2 \nu(dx') \right)^{1/2}  \right] \right) . 
\end{align*}
As the measure $\nu$ is $\mu$-stationary, we get
\begin{align*}
\bb E  \int_{ \bb P(\bb V) }   \delta(G_px', y)^2 \nu(dx')  
=  \int_{ \bb P(\bb V) }   \delta(x', y)^2 \nu(dx')
\leq C,
\end{align*}
where $C$ does not depend on $y \in \bb P(\bb V^*)$ due to \eqref{new expmoment001}. 
Set 
\begin{align*}
F_{p,y} = \left\{ \int_{ \bb P(\bb V) }   \delta(G_px', y)^2 \nu(dx')  > e^{\frac{bp}{2}}  \right\}. 
\end{align*}
By Chebyshev's inequality, we obtain
\begin{align} \label{exp bound Fp-001}
\bb P\left( F_{p,y} \right) \leq  C e^{- \frac{bp}{2}}. 
\end{align}
Now we get, for any $\alpha \in (0,\alpha_0/3]$,
\begin{align*}
& \bb E \exp \left\{ \alpha \left| \bar{\delta}(G_px, y)  \right|  \right\}  \notag\\
& \leq  \bb E  \mathds 1_{ F_{p,y}^c \cap  E^c_{p,x} \cap A^c_{p,x,y}  }  \exp \left( \alpha C e^{- \frac{a}{2} p}  +  \alpha e^{- b p/2} 
    \left\{ \ee p +  e^{bp /4}  \right\} \right)   \notag\\
 & \quad +  \bb E  \mathds 1_{ E_{p,x} \cup A_{p,x,y} \cup F_{p,y} } \exp \left( \alpha \left| \bar{\delta}(G_px, y)  \right|  \right) \notag \\
& \leq   \exp \left( \alpha C e^{- \frac{a}{2} p}  +  \alpha e^{- b p/2} 
    \left\{ \ee p +  e^{bp /4}  \right\} \right)   \notag\\
 & \quad +  \bb E  \mathds 1_{ E_{p,x} \cup A_{p,x,y} \cup F_{p,y} } \exp \left( \alpha \left| \delta(G_px, y) \right| 
 + \alpha  \left| \int_{\bb P(\bb V)}  \delta(G_px', y) \nu(dx')  \right|  \right). 
\end{align*}
As $p \to \infty$, the first term is bounded by $1+e^{-cp}$, for some constant $c>0$.
We claim that if $\alpha>0$ is chosen small enough, then the expectation of the second term tends to $0$ at an exponential rate.
Indeed, by H\"older's inequality, we have, for any $\alpha \in (0,\alpha_0/3]$,
\begin{align}  \label{Final bound projection-001001}
&\int_{\bb P(\bb V)} 
\bb E  \mathds 1_{ E_{p,x} \cup A_{p,x,y} \cup F_{p,y} } \exp \left( \alpha \left| \delta(G_px, y) \right| 
 + \alpha \left| \int_{\bb P(\bb V)}  \delta(G_px', y) \nu(dx')  \right|  \right) \nu(dx) \notag \\
 &\leq 
 \left( \int_{\bb P(\bb V)}  \bb P ( E_{p,x} \cup A_{p,x,y} \cup F_{p,y} ) \nu(dx)\right)^{1/3} 
 \left(\int_{\bb P(\bb V)}  \bb E  \exp \left( 3 \alpha \left| \delta(G_px, y) \right| \right) \nu(dx)\right)^{1/3} \notag \\
&\qquad\qquad\times  \left[ \bb E  \exp \left( 3 \alpha  \left| \int_{\bb P(\bb V)}  \delta(G_px', y) \nu(dx')  \right| \right) \right]^{1/3}.
 \end{align}
From \eqref{exp bound Apx-001}, \eqref{exp bound Epx-001} and \eqref{exp bound Fp-001}, we have that, for some constant $c>0$,
uniformly in $x\in \bb P(\bb V)$, $y \in \bb P(\bb V^*)$ and $p$ large enough,
\begin{align*} 
\bb P ( E_{p, x} \cup A_{p, x, y} \cup F_{p, y} ) \leq e^{-cp}.
\end{align*}
By Jensen's inequality, the third factor on the right-hand side of  \eqref{Final bound projection-001001}  
is less than the second factor. 
As the measure $\nu$ is $\mu$-stationary, the latter is equal to
\begin{align*} 
\int_{\bb P(\bb V)}  \exp \Big( 3 \alpha \left| \delta(x, y) \right| \Big) \nu(dx),
\end{align*}
which is finite by \eqref{new expmoment001}, as soon as $3\alpha \leq \alpha_0$.
\end{proof}

\subsection{Integral approximation for perturbations depending on $n$} \label{sec-integral type approx for perturb}
In this section, we establish an approximation property 
similar to that demonstrated in Proposition \ref{Prop-projection approx-001},
but for a different type of perturbations that arises in the proof of Theorem \ref{Thm-measure-rho}. 
We will show that these perturbations also satisfy the assumption of Theorem \ref{Pro-Appendix-Final2-Inequ}, 
thanks to the proposition below which can be seen as an integral version of Proposition \ref{Prop-projection approx-001}. 

Let $\mu^{-1}$ be the image of the measure $\mu$ under the inverse map $g \mapsto g^{-1}$ of $\bb G$;
this is the law of the increments of the random walk $g_n^{-1} \cdots g_1^{-1}$.
The semigroup spanned by the support of $\mu^{-1}$ is the set $\Gamma_{\mu}^{-1}$ of inverses of elements of $\Gamma_{\mu}$. 
Since the semigroup $\Gamma_{\mu}^{-1}$ is proximal and strongly irreducible in $\bb V^*$, 
the probability measure $\mu^{-1}$ on $\Gamma_{\mu}^{-1}$ admits a unique stationary probability measure $\nu^*$ on $\bb P(\bb V^*)$. 
Let  $x \in \bb P(\bb V)$. 
We define the perturbation functions on $\Omega \times \bb P(\bb V^*)$ as follows: for
$\omega=(g_1,g_2,\ldots) \in \Omega$ and $y \in \bb P(\bb V^*)$,  
\begin{align}\label{perturbed function in bb Y-inte-001}
f_{n}^{x,m}(\omega, y) = \delta(g_1 \cdots g_{m-n} x, y), \quad 0 \leq n\leq m,  
\end{align}
with the convention $f_{m}^{x,m}(\omega, y) = \delta(x, y)$ for $m \geq 0$.

As in the case of the ideal perturbation, the function $f_{m}^{x,m}$ has a finite exponential moment 
with respect to the probability measure $\bb P \otimes \nu^*$:
there exists a constant $\alpha > 0$ such that for any $0 \leq n \leq m$ and $x \in \bb P(\bb V)$,
\begin{align}\label{First-important-Property}
& \int_{\bb P(\bb V^*)}  \int_{\Omega} e^{\alpha  f_{n}^{x,m}(\omega, y)} \bb P(d\omega) \nu^*(dy) < \infty.
\end{align}
 The constant $\alpha$ can be taken to be the same as in the condition \eqref{Regularity-nu}. 
Indeed, for any $0 \leq n \leq m$ and $x \in \bb P(\bb V)$, we have 
\begin{align*}
& \int_{\bb P(\bb V^*)}  \int_{\Omega} e^{\alpha  f_{n}^{x,m}(\omega, y)} \bb P(d\omega) \nu^*(dy)  \notag\\
& = \int_{\bb G^{m-n}}  \int_{\bb P(\bb V^*)}   e^{\alpha  \delta(g_1 \cdots g_{m-n} x, y)} \nu^*(dy)  \mu(dg_1)\ldots \mu(dg_{m-n}) \notag\\
& \leq  \sup_{x' \in \bb P(\bb V)} \int_{\bb P(\bb V^*)}  e^{ \alpha \delta(x', y) } \nu^*(dy) < \infty,
\end{align*}
where the finiteness of the last quantity  is due to \cite[Theorem 14.1]{BQ16b}.

The following result is an integral version (with respect to the variable $y\in \bb P(\bb V^*)$) of Proposition \ref{Prop-projection approx-001}.
It asserts that the function $f^{x, m}_{n}$ can  be well approximated by functions 
that rely on a finite number of coordinates.

\begin{proposition} \label{Prop-projection approx-inte001} 
Assume that $\mu$ has a finite exponential moment and that $\Gamma_{\mu}$ is proximal and strongly irreducible.
Then the function $f^{x,m}$ defined by \eqref{perturbed function in bb Y-inte-001} satisfies 
 the following approximation property: 
 there are constants $\alpha, \beta, c >0$ such that, for any $0 \leq n \leq m$, $x \in \bb P(\bb V)$ and $p \geq 1$, 
\begin{align*}
\int_{\bb P(\bb V^*)} \int_{\Omega}   e^{\alpha |f^{x,m}_{n}(\omega, y) - \bb E (f^{x,m}_{n}(\cdot, y) | \scr A_p)(\omega) |} \bb P(d\omega) \nu^*(dy) 
\leq 1+ ce^{-\beta p}.
\end{align*}
\end{proposition}

\begin{proof}
Our proof follows a similar approach as the one used in Proposition \ref{Prop-projection approx-001}.
Note that, by \eqref{perturbed function in bb Y-inte-001}, 
whenever  $p \geq m-n$, we have, for $\bb P$-almost all $\omega \in \Omega$, 
\begin{align*}
f^{x,m}_{n}(\omega, y) - \bb E (f^{x,m}_{n}(\cdot, y) | \scr A_p)(\omega)  = 0,  
\end{align*}
so the assertion of the proposition becomes evident. 
Therefore, in the remainder of the proof, we can assume that $p < m-n$. 
For fixed $x \in \bb P(\bb V)$, 
still by \eqref{perturbed function in bb Y-inte-001}, 
we have, for $\nu^*$-almost every $y \in \bb P(\bb V^*)$ and $\bb P$-almost every $\omega=(g_1,g_2,\ldots) \in \Omega$, 
\begin{align*}
\bb E (f^{x,m}_{n}(\cdot, y) | \scr A_p)(\omega)
 = \int_{\bb G^{m-n-p}}  \delta \Big(g_1 \cdots g_p g_{p+1}' \cdots g_{m-n}' x, y \Big)  \mu(dg_{p+1}') \ldots \mu(dg_{m-n}'). 
\end{align*}
For short, set
\begin{align}\label{def-tilde-delta}
 \widetilde{\delta}(\omega, x, y)  
& =  \delta(g_1 \cdots g_{m-n} x, y)  \notag\\
& \quad - \int_{\bb G^{m-n-p}} \delta \Big( g_1 \cdots g_p g_{p+1}' \cdots g_{m-n}' x, y \Big)  \mu(dg_{p+1}') \ldots \mu(dg_{m-n}'). 
\end{align}
As before, by \eqref{Regularity-nu}, 
there exists a constant $\alpha_0 > 0$ satisfying \eqref{new expmoment001}. 
Applying Fubini's theorem,  for $\alpha\in (0,\alpha_0/3]$, this yields 
\begin{align*}
& \int_{\bb P(\bb V^*)} \int_{\Omega}   e^{\alpha |f^{x,m}_{n}(\omega, y) - \bb E (f^{x,m}_{n}(\cdot, y) | \scr A_p)(\omega) |} \bb P(d\omega) \nu^*(dy) \notag\\
& =   \int_{\Omega} \int_{\bb P(\bb V^*)}  \exp \left\{ \alpha \left| \widetilde{\delta}(\omega, x, y) \right|  \right\}  \nu^*(dy) \bb P(d\omega)  \notag\\
& =   \int_{\Omega} \int_{\bb P(\bb V^*)}  \exp \bigg\{ \alpha \bigg|  \int_{\bb G^{m-n-p}} 
 \left( \delta(g_1 \cdots g_{m-n} x, y) - \delta(g_1 \cdots g_p g_{p+1}' \cdots g_{m-n}' x, y) \right)  \notag\\
& \qquad \mu(dg_{p+1}') \ldots \mu(dg_{m-n}')  \bigg|  \bigg\}   \nu^*(dy) \bb P(d\omega), 
\end{align*}
where the finiteness of the integrals is guaranteed by \eqref{new expmoment001}. 
Let $\ee > 0$, whose value will be determined later. 
For $y \in \bb P(\bb V^*)$, define the set $A_{p,x,y} = \{ \omega \in \Omega: \delta (g_1 \cdots g_{m-n} x, y) > \ee p \}$. 
Since $p < m-n$, 
by \cite[Proposition 14.3]{BQ16b}, we have uniformly in $x \in \bb P(\bb V)$ and $y \in \bb P(\bb V^*)$,
\begin{align} \label{exp bound Apx-001-modi}
\bb P \left( A_{p,x,y} \right) \leq C e^{ - cp }. 
\end{align}
On the set $A_{p,x,y}^c$, we have 
\begin{align*}
& \left| \int_{\bb G^{m-n-p}}   \left( \delta(g_1 \cdots g_{m-n} x, y) - \delta(g_1 \cdots g_p g_{p+1}' \cdots g_{m-n}' x, y) \right) 
 \mu(dg_{p+1}') \ldots \mu(dg_{m-n}') \right|  \notag\\
& \leq  \left| \int_{ B_{p,x}(\omega) } 
\left( \delta(g_1 \cdots g_{m-n} x, y) - \delta(g_1 \cdots g_p g_{p+1}' \cdots g_{m-n}' x, y) \right)   \mu(dg_{p+1}') \ldots \mu(dg_{m-n}') \right| \notag\\
& \quad   +  \left| \int_{ B_{p,x}(\omega)^c } 
\left( \delta(g_1 \cdots g_{m-n} x, y) - \delta(g_1 \cdots g_p g_{p+1}' \cdots g_{m-n}' x, y) \right)
 \mu(dg_{p+1}') \ldots \mu(dg_{m-n}') \right|, 
\end{align*}
where $B_{p,x}(\omega)$ is the random set 
\begin{align*}
\Big\{ (g_{p+1}', \cdots, g_{m-n}') \in \bb G^{m-n-p}: d \Big( g_1 \cdots g_{m-n} x, g_1 \cdots g_p g_{p+1}' \cdots g_{m-n}' x \Big) &\leq e^{-ap} \Big\} \\
&\subseteq \bb G^{m-n-p}.
\end{align*}
By Lemma \ref{Lem-random-measure}, we have 
\begin{align}\label{inequa-random-measure-modified}
\bb P \left( \mu^{\otimes (m-n-p)} \left( B_{p,x}(\omega)  \right) > e^{ - bp } \right)  \leq C e^{-cp}. 
\end{align}
On $A_{p,x,y}^c$, for $(g_{p+1}', \ldots, g_{m-n}') \in B_{p, x}(\omega)$, we get, by the mean value theorem, 
\begin{align*}
& \left|  \delta(g_1 \cdots g_{m-n} x, y) - \delta(g_1 \cdots g_p g_{p+1}' \cdots g_{m-n}' x, y) \right|  \notag\\
&  \leq C e^{\delta (g_1 \cdots g_{m-n} x, y)} e^{- ap}
\leq  C e^{ \ee p } e^{- ap} \leq  C e^{- \frac{a}{2} p}, 
\end{align*}
where we have assumed that $\ee \leq a/2$. 
Hence, we obtain that for any $\omega \in A_{p,x,y}^c$,
\begin{align*}
& \left| \int_{\bb G^{m-n-p}}  \left( \delta(g_1 \cdots g_{m-n} x, y) - \delta(g_1 \cdots g_p g_{p+1}' \cdots g_{m-n}' x, y) \right)
 \mu(dg_{p+1}') \ldots \mu(dg_{m-n}') \right|    \notag\\
& \leq  C  e^{- \frac{a}{2} p} \\
& \quad + \left| \int_{ B_{p,x}(\omega)^c }  
 \left( \delta(g_1 \cdots g_{m-n} x, y) - \delta(g_1 \cdots g_p g_{p+1}' \cdots g_{m-n}' x, y) \right) 
  \mu(dg_{p+1}') \ldots \mu(dg_{m-n}') \right|   \notag\\
& \leq   C e^{- \frac{a}{2} p}  +  \left( \mu^{\otimes (m-n-p)} \left(  B_{p, x}(\omega)^c \right) \right)^{1/2}   \notag\\ 
& \quad \times   \left[ \int_{ \bb G^{m-n-p} }  \left( \delta(g_1 \cdots g_{m-n} x, y) - \delta(g_1 \cdots g_p g_{p+1}' \cdots g_{m-n}' x, y) \right)^2  \mu(dg_{p+1}') \ldots \mu(dg_{m-n}')  \right]^{1/2}  \notag\\
& \leq  C e^{- \frac{a}{2} p} +  \left( \mu^{\otimes (m-n-p)} \left(  B_{p, x}(\omega)^c \right) \right)^{1/2}  \notag\\ 
& \quad \times    \left\{ \delta(g_1 \cdots g_{m-n} x, y) +  \left[  \int_{ \bb G^{m-n-p} }   
     \delta(g_1 \cdots g_p g_{p+1}' \cdots g_{m-n}' x, y)^2  \mu(dg_{p+1}') \ldots \mu(dg_{m-n}')  \right]^{1/2}  \right\}   \notag\\
& \leq   C e^{- \frac{a}{2} p} +  \left( \mu^{\otimes (m-n-p)} \left(  B_{p, x}(\omega)^c \right) \right)^{1/2}  \notag\\ 
& \quad \times    \left\{ \ee p +  \left[  \int_{ \bb G^{m-n-p} }   
     \delta(g_1 \cdots g_p g_{p+1}' \cdots g_{m-n}' x, y)^2  \mu(dg_{p+1}') \ldots \mu(dg_{m-n}')  \right]^{1/2}  \right\}. 
\end{align*}
We set $E_{p,x} = \{ \mu^{\otimes (m-n-p)} (  B_{p, x}(\omega)^c ) > e^{ -bp } \}$. 
By \eqref{inequa-random-measure-modified}, we have 
\begin{align} \label{exp bound Epx-001-modi}
\bb P ( E_{p,x} ) \leq Ce^{-cp}. 
\end{align}
Therefore, for $\alpha\in (0,\alpha_0/3]$,
\begin{align*}
& \int_{\bb P(\bb V^*)} \int_{\Omega}   e^{\alpha |f^{x,m}_{n}(\omega, y) - \bb E (f^{x,m}_{n}(\cdot, y) | \scr A_p)(\omega) |} \bb P(d\omega) \nu^*(dy)  \notag\\
& \leq  \int_{\bb P(\bb V^*)} \bb E \mathds 1_{ E^c_{p,x} \cap A^c_{p,x,y} } \exp \Bigg( \alpha C e^{- \frac{a}{2} p}
+  \alpha e^{- b p/2}   \Bigg\{ \ee p   \notag\\
& \quad   +  \bigg[  \int_{ \bb G^{m-n-p} }   
     \delta(g_1 \cdots g_p g_{p+1}' \cdots g_{m-n}' x, y)^2  \mu(dg_{p+1}') \ldots \mu(dg_{m-n}')  \bigg]^{1/2}  \Bigg\} \Bigg) \nu^*(dy)  \notag\\
 & \quad + \int_{\bb P(\bb V^*)} \int_{\Omega}  \mathds 1_{ E_{p,x} \cup A_{p,x,y} }  
   \exp \left\{ \alpha \left| \widetilde{\delta}( \omega, x, y) \right|  \right\}  \bb P(d \omega) \nu^*(dy), 
\end{align*}
where $\widetilde{\delta}(\omega, x, y)$ is defined by \eqref{def-tilde-delta}. 
Now we have
\begin{align}\label{Triw-modified-inequa}
& \int_{\bb P(\bb V^*)} \bb E \int_{ \bb G^{m-n-p} }   
     \delta(g_1 \cdots g_p g_{p+1}' \cdots g_{m-n}' x, y)^2  \mu(dg_{p+1}') \ldots \mu(dg_{m-n}')  \nu^*(dy) \notag\\
& \leq  \sup_{x' \in \bb P(\bb V)} \int_{\bb P(\bb V^*)}   \delta(x', y)^2 \nu^*(dy) 
\leq C,
\end{align}
where the last inequality holds due to \cite[Theorem 14.1]{BQ16b}. 
For $y \in \bb P(\bb V^*)$, set 
\begin{align*}
F_{p,y} = \left\{ \omega \in \Omega:   \int_{ \bb G^{m-n-p} }   
     \delta(g_1 \cdots g_p g_{p+1}' \cdots g_{m-n}' x, y)^2  \mu(dg_{p+1}') \ldots \mu(dg_{m-n}')  > e^{\frac{bp}{2}}  \right\}. 
\end{align*}
By Chebyshev's inequality and \eqref{Triw-modified-inequa}, we obtain
\begin{align} \label{exp bound Fp-001-modi}
\int_{\bb P(\bb V^*)} \bb P \left( F_{p,y} \right) \nu^*(dy) \leq  C e^{- \frac{bp}{2}}. 
\end{align}
Then, we derive that, for $\alpha\in (0,\alpha_0/3]$, 
\begin{align*}
&  \int_{\bb P(\bb V^*)} \int_{\Omega}   e^{\alpha |f^{x,m}_{n}(\omega, y) - \bb E (f^{x,m}_{n}(\omega, y) | \scr A_p) |} \bb P(d\omega) \nu^*(dy)  \notag\\
& \leq  \int_{\bb P(\bb V^*)} \bb E  \mathds 1_{ F_{p,y}^c \cap  E^c_{p,x} \cap A^c_{p,x,y}  }  \exp \left( \alpha C e^{- \frac{a}{2} p}  +  \alpha e^{- b p/2} 
    \left\{ \ee p +  e^{bp /4}  \right\} \right)  \nu^*(dy)  \notag\\
 & \quad + \int_{\bb P(\bb V^*)} 
   \int_{\Omega}  \mathds 1_{ E_{p,x} \cup A_{p,x,y} \cup F_{p,y} }   
    \exp \left\{ \alpha \left| \widetilde{\delta}(\omega, x, y) \right|  \right\} \bb P(d \omega) \nu^*(dy) \notag \\
& \leq   \exp \left( \alpha C e^{- \frac{a}{2} p}  +  \alpha e^{- b p/2} 
    \left\{ \ee p +  e^{bp /4}  \right\} \right)   \notag\\
 & \quad + \int_{\bb P(\bb V^*)}   \bb E  \mathds 1_{ E_{p,x} \cup A_{p,x,y} \cup F_{p,y} } \exp \Bigg\{ \alpha \left| \delta(g_1 \cdots g_{m-n} x, y) \right|  \notag\\
& \qquad\qquad 
  + \alpha  \left|  \int_{\bb G^{m-n-p}} \delta(g_1 \cdots g_p g_{p+1}' \cdots g_{m-n}' x, y)  \mu(dg_{p+1}') \ldots \mu(dg_{m-n}')  \right|  \Bigg\} \nu^*(dy).
\end{align*}
As $p \to \infty$, the first term is bounded by $1+e^{-cp}$, for some constant $c>0$.
We claim that if $\alpha>0$ is chosen small enough, then the expectation of the second term  goes to $0$ with an exponential rate.
Indeed, by H\"older's inequality, we have
\begin{align}  \label{Final bound projection-001001-modi}
& \int_{\bb P(\bb V^*)}   \bb E  \mathds 1_{ E_{p,x} \cup A_{p,x,y} \cup F_{p,y} } \exp \Bigg\{ \alpha \left| \delta(g_1 \cdots g_{m-n} x, y) \right|  \notag\\
& \qquad\qquad\qquad 
  + \alpha  \left|  \int_{\bb G^{m-n-p}} \delta(g_1 \cdots g_p g_{p+1}' \cdots g_{m-n}' x, y)  \mu(dg_{p+1}') \ldots \mu(dg_{m-n}')  \right|  \Bigg\} \nu^*(dy) \notag \\
 &\leq 
 \left( \int_{\bb P(\bb V^*)}  \bb P ( E_{p,x} \cup A_{p,x,y} \cup F_{p,y} ) \nu^*(dy)\right)^{1/3}   \notag\\
& \quad \times \left( \int_{\bb P(\bb V^*)}  \bb E  \exp \left\{ 3 \alpha \left| \delta(g_1 \cdots g_{m-n} x, y) \right| \right\} \nu^*(dy) \right)^{1/3} \notag \\
&\quad  \times  \Big( \int_{\bb P(\bb V^*)} \bb E  \exp \Big\{ 3 \alpha  \Big| \int_{\bb G^{m-n-p}} \delta(g_1 \cdots g_p g_{p+1}' \cdots g_{m-n}' x, y)   \notag \\
&  \qquad\qquad  \mu(dg_{p+1}') \ldots \mu(dg_{m-n}')  \Big| \Big\}  \nu^*(dy)  \Big)^{1/3}.
 \end{align}
From \eqref{exp bound Apx-001-modi}, \eqref{exp bound Epx-001-modi} and \eqref{exp bound Fp-001-modi}, we have that, for some $c>0$,
uniformly in $x\in \bb P(\bb V)$ and $p$ large enough,
\begin{align*} 
\int_{\bb P(\bb V^*)} \bb P ( E_{p,x} \cup A_{p,x,y} \cup F_{p,y} ) \nu^*(dy) \leq e^{-cp}.
\end{align*}
By Jensen's inequality, the third factor on the right-hand side of  \eqref{Final bound projection-001001-modi}  
is less than the second factor. 
The latter is bounded by 
\begin{align*} 
\sup_{x' \in \bb P(\bb V)} \int_{\bb P(\bb V^*)}  \exp \left\{ 3 \alpha \left| \delta(x', y) \right| \right\} \nu^*(dy),
\end{align*}
which, by \eqref{new expmoment001}, is finite as soon as $3\alpha\leq \alpha_0$. 
The conclusion of the proposition follows. 
\end{proof}


\section{Existence of the target measure}\label{Sec-harmonic-measure}
In this section, we prove Theorem \ref{Thm-measure-rho}, which establishes the existence of a Radon measure $\rho$ on $\bb P(\bb V) \times \bb R$ 
that satisfies \eqref{exist of measure rho-001}. 
This will be achieved as an application of a more general theorem concerning random walks with perturbations depending on the future,  
which will be discussed in Section \ref{sec invar func}. 

\subsection{From the perturbed random walk toward the target measure} \label{sec-ideal preturb-001}
Following the heuristics presented in Section \ref{sec-heuristic for harmonic measure},
we start the construction of the measure $\rho$ by relating  the random walk \eqref{def of direct RW-001} 
to a convenient reversed random walk. 
This reversed random walk is determined by the dual cocycle $\sigma^*$ (defined in \eqref{dual cocycle-001})
which acts on the space $\bb G \times \bb P(\bb V^*)$, and by a perturbation function on the same space. 
The perturbation function that we need is the limiting ideal perturbation defined in \eqref{perturbed function in bb Y-001}, 
which takes the form $f(\omega, y) = \delta ( \xi (\omega), y )$ for  $\omega \in \Omega=\bb G^{\bb N^*}$ and $y\in \bb P(\bb V^*)$.

In other words, the existence of the target harmonic measure $\rho$ is related with the perturbed 
random walk $(\tilde S_n(\cdot,y))_{n\geq1}$ defined as follows:
for $\omega=(g_1,g_2,\ldots) \in \Omega$ and $y\in \bb P(\bb V^*)$,
\begin{align} \label{eq-idealRWaaa-001}
\tilde S_n(\omega,y) := - \sigma^*( g^{-1}_n\cdots g^{-1}_1, y)  + f\circ \widetilde T^n (\omega, y) - f(\omega, y), 
\quad n\geq 1.
\end{align}
A distinct feature of this model is the presence of the perturbation term $f\circ \widetilde T^n-f$,  
which depends on the whole sequence $\omega=(g_1,g_2,\ldots) \in \Omega$, particularly on the future coordinates $(g_k)_{k > n}$.
For $y\in \bb P(\bb V^*)$ and $t\in \bb R$, we define the perturbed exit time 
\begin{align} \label{stop time ideal case-001}
\tilde \tau_{y,t}(\omega) 
= \min \left\{ n \geq 1 : t + \tilde S_n(\omega,y) < 0 \right\}. 
\end{align}
The time $\tilde \tau_{y,t}$ is a random variable on $\Omega$,
 but it is not a stopping time with respect to the natural filtration $(\scr A_k)_{k \geq 0}$  
associated with the dual random walk $-\sigma^*( g^{-1}_k\cdots g^{-1}_1, y)$. 
For any  continuous function $\varphi$ on $\bb P(\bb V)$, $t\in \bb R$ and $n\geq 1$, we set
\begin{align} \label{def-U-varphi-001}
U^{\varphi}_{n}(t) 
= \int_{\bb P(\bb V^*) } \int_{\Omega} \bigg[  \Big( t +\tilde S_n(\omega,y)  \Big) \varphi(\xi(\omega) ); 
\tilde \tau_{y,t}(\omega) >n     \bigg]  \bb P(d\omega) \nu^*(dy). 
\end{align}

We will verify that the assumptions of Theorem \ref{Pro-Appendix-Final2-Inequ} 
are satisfied for the action of the group $\bb G$ on the space $\bb X = \bb P(\bb V^*)$, 
 equipped with the cocycle $ -\sigma^*$ (instead of $\sigma$ there) 
and the perturbation sequence $\mathfrak f = (f_n)_{n \geq 0}$,
where the sequence is constant with 
$f_n(\omega, y) = f(\omega, y) = \delta ( \xi (\omega), y )$ for all $n\geq 0$.
By the assumptions of Theorem \ref{Thm-measure-rho}, the Lyapunov exponent $\lambda_{\mu}$ corresponding to the cocycle $\sigma$ is zero. 
Consequently, the cocycle $\sigma^*$ also has a zero Lyapunov exponent (see Theorem 3.28 in \cite{BQ16b}).
From Lemma 10.18 in \cite{BQ16b}, 
it follows that the cocycle $\sigma^*$ can be centered, 
meaning that there exists a continuous function 
$\psi_0$ on $\bb P(\bb V^*)$ such that, for any $y\in \bb P(\bb V^*)$, 
\begin{align*} 
\int_{\bb G } \Big( \sigma^*(g^{-1},y) + \psi_0(g^{-1} y) - \psi_0(y) \Big)  \mu(dg) = 0.
\end{align*}
In other words,  up to a bounded coboundary, the cocycle $\sigma^*$ satisfies the assumption \eqref{centering-001}.
Besides, the effective central limit theorem \eqref{BEmart-001} holds true, 
by the Berry-Esseen theorem in \cite{LePage82}; see also \cite{Boug-Lacr85} (Theorem 5.1, page 122). 
Note that, due to the bound \eqref{Regularity-nu} and Proposition \ref{Prop-projection approx-001}, 
the sequence  $\mathfrak f = (f_n)_{n \geq 0}$  
satisfies both the moment condition \eqref{exp mom for g 002} and the approximation property \eqref{approxim rate for gp-002}. 
The approximation property \eqref{approx property of theta-001} is confirmed by
 \eqref{approx property of theta-in bb Y-001} of Corollary \ref{corollary f_y-001}. 
 Therefore, all conditions of Theorem  \ref{Pro-Appendix-Final2-Inequ} are verified for 
the group $\bb G$, the space $\bb X =\bb P(\bb V^*)$, the cocycle $-\sigma^*$ 
(in place of $\sigma$ in the original formulation)
and the perturbation sequence $\mathfrak f $.
As a consequence, Corollaries \ref{Main th-003} and \ref{Main th-001} hold true. 
This fact will be used below to prove the following important proposition, 
which states the existence of the target measure $\rho$.

\begin{proposition} \label{prop-reversed oper-000}
Assume that  $\Gamma_{\mu}$ is proximal and strongly irreducible and that 
$\mu$ has a finite exponential moment and its Lyapunov exponent is $0$.
Then, there exists a Radon measure $\rho$ on $\bb P(\bb V) \times \bb R$ such that, 
for any continuous function $\varphi$ on $\bb P(\bb V)$ 
and any continuous compactly supported function $\psi$ on $\bb R$, we have,
as $n\to\infty$,
\begin{align} \label{target measure for ideally pertub RW-001}
 \int_{\bb R} \psi(t) U^{\varphi}_{n}(t)  dt \to   \int_{\bb P(\bb V) \times \bb R}  \varphi(x) \psi(t) \rho(dx,dt).  
\end{align}
The marginal of the Radon measure $\rho$ on $\bb R$  is absolutely continuous 
with respect to the Lebesgue measure with non-decreasing density function
$W(t)=\frac{\rho(\bb P(\bb V),dt)}{dt}$, $t\in \bb R$ such that 
\begin{align} \label{bound-U-app}
W(t) \leq c (1+\max \{t,0\}), \quad t\in \bb R.
\end{align}
Moreover, for any continuous compactly supported function $h$ on $\bb P(\bb V) \times \bb R$, we have 
\begin{align}\label{limit-U-t-001}
\lim_{t \to \infty} \frac{1}{t} \int_{ \bb P(\bb V) \times \bb R }  h(x', t' - t) \rho(dx', dt') 
= \int_{  \bb P(\bb V) \times \bb R  }  h(x', t') \nu(dx') dt'. 
\end{align}
\end{proposition}

\begin{proof}
First assume that the function $\varphi$ is H\"older continuous on $\bb P(\bb V)$.  
By Corollary \ref{corollary f_y-001}, the function $\omega \mapsto \varphi(\xi(\omega)) $ 
satisfies the approximation property \eqref{approx property of theta-001}.
Therefore, by Corollary \ref{Main th-003},
for any non-negative continuous compactly supported function $\psi$ on $\bb R$, as $n\to \infty$, 
the quantity $ \int_{\bb R} U^{\varphi}_{n}(t) \psi(t) dt  $
has a limit of the form $ \int_{\bb R} U^{\varphi}(t) \psi(t) dt$ for some non-decreasing function $U^{\varphi}: \bb R \to \bb R_+$.

In the case where $\varphi$ is any continuous function on $\bb P(\bb V)$, 
for any $\ee >0$, we can choose a H\"older continuous function $\varphi'$ on $\bb P(\bb V)$
such that $\sup_{x \in \bb P(\bb V)} |\varphi(x)-\varphi'(x)| \leq \ee$. 
Then,  for any $\psi$ as above and any $n\geq 1$, we have, using Corollary \ref{Main th-001}, 
\begin{align*} 
\left|    \int_{\bb R} U^{\varphi}_{n}(t) \psi(t) dt  - \int_{\bb R} U^{\varphi'}_{n}(t) \psi(t) dt  \right|
&\leq  \ee   \int_{\bb R} U_{n}(t)   \psi(t) dt  \\
&\leq  c \ee  \int_{\bb R}  (1+\max{t,0})    \psi(t) dt.  
\end{align*}
Since, by the first part of the proof, the integral $\int_{\bb R} U^{\varphi'}_{n}(t) \psi(t) dt$ has a limit,
the inequality above shows that the sequence of integrals $ \int_{\bb R} U^{\varphi}_{n}(t) \psi(t) dt,$ $n\geq 1$, 
forms a Cauchy sequence and thus also converges as $n\to\infty$.
In other words, for any continuous compactly supported function $h$ on $\bb P(\bb V) \times \bb R$ which is of the product form
$(x, t) \mapsto \varphi(x) \psi(t)$, the integral 
\begin{align} \label{INT-EXPECT-001}
\int_{\bb P(\bb V^*) \times \bb R} \int_{\Omega} &\bigg[  \bigg( t + \tilde S_n(\omega,y) \bigg) h(\xi(\omega) , t);  
\tilde\tau_{y,t}(\omega) >n \bigg] \bb P(d\omega) 
\nu^*(dy) dt
\end{align}
has a limit as $n\to\infty$.
Since every continuous compactly supported function $h$ on $\bb P(\bb V) \times \bb R$ 
can be uniformly approximated by sums of functions of the product form,
the same argument as above shows that the integral  \eqref{INT-EXPECT-001}
has a limit for any $h$.
As this integral is non-negative when $h$ is non-negative, 
by the Riesz representation theorem, the limit is of the form $\rho (h)$, 
where $\rho$ is a Radon measure on $\bb P(\bb V) \times \bb R$.

The properties of the marginal of the measure $\rho$ on $\bb R$  are direct consequences of Corollaries \ref{Main th-003} and \ref{Main th-001}. 
Finally, we prove \eqref{limit-U-t-001}. 
Assume first that $h$ is a continuous compactly supported function on $\bb P(\bb V) \times \bb R$ which is of the form $(x, t) \mapsto \varphi(x) \psi(t)$,
where $\varphi$ is a H\"older continuous function on $\bb P(\bb V)$ and $\psi$ is a continuous compactly supported function on $\bb R$. 
Then, by Corollaries \ref{corollary f_y-001} and \ref{Main th-003}, we obtain 
\begin{align*}
\lim_{t \to \infty} \frac{1}{t} \int_{\bb R} U^{\varphi}(t') \psi(t' - t) dt'
&=  \int_{\bb R} \psi(t') dt' \int_{\Omega} \varphi(\xi(\omega)) \bb P(d \omega) \\
&= \int_{\bb R} \psi(t') dt' \int_{\bb P(\bb V)} \varphi(x) \nu(dx),
\end{align*}
which establishes \eqref{limit-U-t-001} when $h$ is of the above form. 

The general case of  \eqref{limit-U-t-001} follows by standard approximation and by using the bound \eqref{bound-U-app}. 
Indeed, if $h$ is any continuous compactly supported function on $\bb P(\bb V) \times \bb R$,
then for any $\ee > 0$, we can find an integer $m \geq 1$, H\"older continuous functions $\varphi_1, \ldots, \varphi_m$ on $\bb P(\bb V)$
and continuous compactly supported functions $\psi_1, \ldots, \psi_m$ on $\bb R$, such that 
\begin{align} \label{approxim by finite sum of prod-001}
\sup_{(x, t) \in \bb P(\bb V) \times \bb R} \Big| h(x, t) - \sum_{i = 1}^m \varphi_i(x) \psi_i(t) \Big| \leq \ee.
\end{align}
Thus, for any $t \in \bb R$, we get 
\begin{align*} 
& \left| \int_{ \bb P(\bb V) \times \bb R }  h(x', t' - t) \rho(dx', dt') -  \sum_{i = 1}^m \int_{ \bb P(\bb V) \times \bb R }  \varphi_i(x') \psi_i(t' - t)  \rho(dx', dt') \right| \notag\\
& \leq \ee \rho(\bb P(\bb V) \times [t-C, t + C] ),
\end{align*}
for some constant $C>0$ depending on the support of the above functions. Using the bound \eqref{bound-U-app}, 
we get $\rho(\bb P(\bb V) \times [t-C, t + C] ) \leq c (t + 2C).$
Dividing by $t$ and using  \eqref{limit-U-t-001} which has been already established for functions $h$  of the form $(x, t) \mapsto \varphi(x) \psi(t)$,  
we obtain 
\begin{align*}
\limsup_{t \to \infty} 
\left|  \frac{1}{t} \int_{ \bb P(\bb V) \times \bb R }  h(x', t' - t) \rho(dx', dt') 
-  \sum_{i = 1}^m \int_{ \bb P(\bb V) \times \bb R }  \varphi_i(x') \psi_i(t')  \nu(dx') dt' \right| 
\leq c \ee. 
\end{align*}
The result ensues from the application of \eqref{approxim by finite sum of prod-001} and letting $\ee\to 0$. 
\end{proof}

\subsection{The reversed random walk} \label{sec-reversed random walk-001}
 In the previous subsection, the target measure $\rho$ appears in 
equation \eqref{target measure for ideally pertub RW-001} as the vague
limit of a sequence of Radon measures pertaining to the ideal random walk \eqref{eq-idealRWaaa-001}. 
In order to prove Theorem \ref{Thm-measure-rho}, we will deal 
with the sequence of Radon measures $(\rho_{n,x})_{n\geq 1}$ on $\bb P(\bb V) \times \bb R$ defined as follows: 
for $n\geq 1$, $x\in \bb P(\bb V)$
and any continuous compactly supported function $h$ on $\bb P(\bb V) \times \bb R$,
\begin{align} \label{Expect in main Theorem-001}
\int_{\bb P(\bb V) \times \bb R}   h(x', t) \rho_{n,x}(dx',dt)
= \int_{0}^{\infty}  t \,  \bb E\Big(  h(g_n\cdots g_1 x, t+\sigma(g_n\cdots g_1, x)); \tau_{x,t} > n - 1 \Big) dt. 
\end{align}
Actually, the proof of Theorem \ref{Thm-measure-rho} consists in showing that the sequence of measures $(\rho_{n,x})_{n\geq 1}$ converges vaguely
to the same measure $\rho$. 
Towards this goal, we need the following lemma.

\begin{lemma} \label{lem-good y-001}
Assume that $\mu$ is such that  
$\Gamma_{\mu}$ is proximal and strongly irreducible.
Then, for any $x \in \bb P(\bb V)$,  $\nu^*$-almost surely in $y\in \bb P(\bb V^*)$, for any $n\geq 1$, we have 
\begin{align*} 
\bb P\left( \delta( g_n \cdots g_1 x, y) < \infty \right) = 1.
\end{align*}
\end{lemma}

\begin{proof}
Fix $x\in \bb P(\bb V)$.
By the definition of $\delta$ (cf.\ \eqref{def of delta func-001}) and \eqref{cohomological eq -vers 002}, for any $y \in \bb P(\bb V^*)$ and $g \in \bb G$, 
we have that $\delta(gx, y) =  \infty$ if and only if $ \delta(x, g^{-1} y) = \infty$. 
Therefore, in order to establish the lemma, it is equivalent to prove that, $\nu^*$-almost surely in $y\in \bb P(\bb V^*)$, 
for any $n\geq 1$, 
$\bb P \left(  \delta(x, (g_n \cdots g_1)^{-1} y) < \infty    \right) = 1$. 

As the measure $\nu^*$ assigns zero mass to every proper projective subspace of $\bb P(\bb V^*)$ (see Lemma 4.6 in \cite{BQ16b}), 
we have
\begin{align*} 
\nu^*\big( y\in \bb P(\bb V^*): \delta(x,y) = \infty   \big) =0.
\end{align*}
Since $\nu^*$ is $\mu^{-1}$-stationary, we get that for any $n\geq 1$, 
\begin{align*} 
& \int_{\bb P(\bb V^*)} \bb P \left(  \delta(x, (g_n \cdots g_1)^{-1} y) = \infty    \right) \nu^*(dy) \notag\\
& = \bb E \left[  \nu^*\left( y\in \bb P(\bb V^*): \delta(x, (g_n \cdots g_1)^{-1} y) = \infty   \right)   \right]  \notag\\
& = \nu^*\left( y\in \bb P(\bb V^*): \delta(x,y) = \infty   \right) = 0,
\end{align*}
which ends the proof of the lemma. 
\end{proof}

For $x \in \bb P(\bb V)$, set 
\begin{align*}
\Delta_x: = \big\{ y \in \bb P(\bb V^*):  \bb P( \delta( g_n \cdots g_1 x, y) < \infty ) = 1, \forall n \geq 1  \big\}, 
\end{align*}
so that, by Lemma \ref{lem-good y-001}, we have $\nu^*(\Delta_x)  =1.$

We will relate the integral on the right-hand side of \eqref{Expect in main Theorem-001}  
to an equivalent expression in terms of the following array of reversed random walks:  
for $\omega=(g_1,g_2,\ldots) \in \Omega,$  $x\in\bb P(\bb V)$, $y \in \Delta_x$ and $1 \leq n \leq m$, 
\begin{align} \label{reversed RWfor products-001}
 \tilde S^{x,m}_{n}(\omega, y) =
-\sigma^*( g^{-1}_{n}\cdots g^{-1}_1,y) +   \delta(g_{n+1}\cdots g_m x, g^{-1}_{n} \cdots g^{-1}_{1}y) - \delta(g_1\cdots g_m x,y),
\end{align}
with the convention that, for $k>m$, the empty right product $g_k \cdots g_m$ is identified with the identity matrix. 
In particular, with $m=n$, we have 
\begin{align} \label{reversed RWfor products-001bbb}
 \tilde S^{x,n}_{n}(\omega, y) =
-\sigma^*( g^{-1}_{n}\cdots g^{-1}_1,y) +   \delta(x, g^{-1}_{n} \cdots g^{-1}_{1}y) - \delta(g_1\cdots g_n x,y).
\end{align}
In these definitions, the additional parameter $m$ represents the range of dependence: 
it specifies how far into the future the perturbation of the random walk $-\sigma^*( g^{-1}_{n}\ldots g^{-1}_1,y)$ can extend. 
The ideal perturbation considered in the previous subsection corresponds to the case where $m = \infty$.

The right-hand side of \eqref{Expect in main Theorem-001} is connected to
the array $(\tilde S^{x,n}_{k}(\cdot, y))_{1\leq k\leq n}$ through the following reversal lemma, 
in analogy to \eqref{duality-ident-001} for random walks on $\bb R$.
\begin{lemma} \label{lemma-duality-for-product-001}
Let $x\in \bb P(\bb V)$ and $y\in \Delta_x$. For any $n\geq 1$ and non-negative measurable function $h$ on
$\bb P(\bb V)\times \bb R$, we have 
\begin{align} \label{duality id-001}
& \int_{\bb R_+} t \,  \bb E \Big[ h\Big( g_n\cdots g_1 x, t+ \sigma (g_n\cdots g_1,x) \Big);  \tau_{x,t} > n-1 \Big] dt \notag\\
 &= \int_{\bb R} \bb E\Big[ \Big(t+ \tilde S^{x,n}_{n}(\cdot,y)  \Big)   
 h(g_1\cdots g_n x, t);   t+ \tilde S^{x,n}_{k}(\cdot,y) \geq 0, 1\leq k\leq n  \Big] dt. 
\end{align}
\end{lemma}
\begin{proof}
For brevity, set $S^x_n= \sigma(g_n\cdots g_1, x)$ for $n\geq 1$ and $S^x_0=0$. 
Using the change of variable $t = u+ S^x_n$, the left-hand side of \eqref{duality id-001} can be rewritten as
\begin{align} \label{After a change of variable-001}
J: & =\int_{\bb R_+} u \bb E \Big[ h  \Big( g_n\cdots g_1 x, u+ S^x_n \Big);  \tau_{x,u} > n-1 \Big] du \notag\\
& = \int_{\bb R}  \bb E \Big[ u h  \Big(g_n\cdots g_1 x, u+ S^x_n\Big);  
 u + S^x_k \geq 0, 0\leq k\leq n-1  \Big] du \notag\\
 & = \int_{\bb R} \bb E\Big[ \Big(t - S^x_n \Big)
   h(g_n\cdots g_1 x, t);    t - S^x_n + S^x_k \geq 0, 0\leq k\leq n-1 \Big] dt.
\end{align}
In the last integral, we have, for $0\leq k\leq n-1$,
\begin{align*} 
S^x_n - S^x_k = \sigma(g_{n}\cdots g_{k+1}, g_{k}\cdots g_{1} x ),
\end{align*}
where by convention $g_{k}\cdots g_{1} x=x$ for $k=0$.
We shall make use of the following cohomological equation, which is obtained from \eqref{cohomological eq -vers 002}:
for $g \in \bb G$, $a \in \bb P(\bb V)$ and $b \in \bb P(\bb V^*)$, we have 
\begin{align} \label{ cohomological eq -vers 003}
\sigma(g, a)  =  \sigma^*(g^{-1}, b)  - \delta (a, g^{-1}b) +  \delta(g a, b). 
\end{align}
Then, with $g =g_{n}\cdots g_{k+1}$, $a =g_{k}\cdots g_{1} x$ and $b=y$, we get, for $0\leq k\leq n$,
\begin{align} \label{ cohomological eq -vers 004}
&S^x_n - S^x_k 
= \sigma(g_{n}\cdots g_{k+1}, g_{k}\cdots g_{1} x )  \notag \\
&= \sigma^* \left( (g_{n}\cdots g_{k+1})^{-1}, y \right) 
- \delta \left( g_{k}\cdots g_{1} x, (g_{n}\cdots g_{k+1})^{-1}y \right)  +  \delta(g_{n}\cdots g_{1} x, y).
\end{align}
Using identity \eqref{ cohomological eq -vers 004} we have
\begin{align*}
J&= \int_{\bb R} \bb E\Big[ \Big(t -\sigma^*( g^{-1}_1\cdots g^{-1}_n,y) + \delta(x,g^{-1}_1\cdots g^{-1}_{n}y)-\delta(g_n\cdots g_1x,y) \Big)
   h(g_n\cdots g_1 x, t);  \notag\\
&\qquad\qquad  t -\sigma^*( g^{-1}_{k+1}\cdots g^{-1}_n,y) +  \delta(g_k\cdots g_1 x, g^{-1}_{k+1} \cdots g^{-1}_{n}y)-\delta(g_n\cdots g_1x,y)
\geq 0, \notag\\
 & \qquad\qquad\qquad\qquad 0\leq k\leq n-1  \Big] dt.
\end{align*}
The random elements  $g_1,\ldots,g_n$ are exchangeable, which is justified by the assumption that these elements 
are independent and identically distributed. 
Therefore, we can reverse the order of the elements $g_1,\ldots,g_n$, 
which gives 
\begin{align*} 
J & = \int_{\bb R}  \bb E\Big[ \Big(t-\sigma^*( g^{-1}_n\cdots g^{-1}_1,y) + \delta(x,g^{-1}_n\cdots g^{-1}_{1}y) - \delta(g_1\cdots g_nx,y) \Big) 
   h(g_1\cdots g_n x, t);  \notag \\
& \qquad\qquad  t-\sigma^*( g^{-1}_{k}\cdots g^{-1}_1,y) +   \delta(g_{k+1}\cdots g_n x, g^{-1}_{k} \cdots g^{-1}_{1}y) - \delta(g_1\cdots g_nx,y) \geq 0, 
\notag\\
& \qquad\qquad\qquad\qquad  1\leq k\leq n  \Big] dt \notag\\ 
& = \int_{\bb R} \bb E\Big[ \Big(t + \tilde S^{x,n}_{n}(\cdot, y) \Big) 
   h(g_1\cdots g_n x, t);   t + \tilde S^{x,n}_{k}(\cdot, y) \geq 0, 1\leq k\leq n  \Big] dt. 
\end{align*}
This is exactly the integral on the right-hand side of \eqref{duality id-001}.
\end{proof}

From Lemma \ref{lemma-duality-for-product-001}, by integrating with respect to $\nu^*(dy)$ and applying Fubini's theorem, we obtain
the following equivalent representation of the integral in \eqref{Expect in main Theorem-001}:
\begin{align} \label{duality id-integr-version-002}
&\int_{\bb P(\bb V) \times \bb R}   h(x',t) \rho_{n,x}(dx',dt) \\
 &=  \int_{\bb R} \int_{\bb P(\bb V^*)} \bb E\Big[ \Big(t+ \tilde S^{x,n}_{n}(\cdot,y)  \Big)   
 h(g_1\cdots g_n x, t);   t+ \tilde S^{x,n}_{k}(\cdot,y) \geq 0, 1\leq k\leq n  \Big] \nu^*(dy) dt. \notag
\end{align}
To apply the theory developed in subsequent sections (more precisely Theorem \ref{Pro-Appendix-Final2-Inequ}),
we need to rewrite the perturbation in the reversed random walks \eqref{reversed RWfor products-001} and \eqref{reversed RWfor products-001bbb} 
in a convenient form.
Recall that, from \eqref{perturbed function in bb Y-inte-001}, 
 for any $\omega = (g_1, g_2, \ldots) \in \Omega$, $x \in \bb P(\bb V)$ and $y \in \Delta_x$, we have defined 
\begin{align} \label{more general perturb-001}
f_n^{x, m}(\omega, y) = \delta(g_1 \cdots g_{m-n} x, y),\ 0 \leq n \leq m. 
\end{align}
From \eqref{more general perturb-001} we get that, for any $0 \leq n \leq m$,
\begin{align*} 
f_n^{x, m} \circ \widetilde T^n(\omega, y)  = \delta \left( g_{n+1}\cdots g_m x, g^{-1}_{n} \cdots g^{-1}_{1}y \right) 
\end{align*}
with the convention $f_m^{x, m} \circ \widetilde T^m(\omega, y)  = \delta (x, g^{-1}_{m} \cdots g^{-1}_{1}y )$ 
and $f_0^{x, m} (\omega, y) = \delta (g_{1}\cdots g_m x, y )$. 
With this notation, we can rewrite 
$(\tilde S^{x,m}_{n}(\cdot,y))_{1\leq n \leq m}$ defined by \eqref{reversed RWfor products-001} as 
\begin{align} \label{another form of the perturbRW-001}
\tilde S^{x,m}_{n}(\omega,y) 
= -\sigma^*( g^{-1}_{n}\cdots g^{-1}_1,y) +  f_n^{x, m} \circ \widetilde T^n(\omega, y)  - f_0^{x, m} \left(\omega, y \right),
\end{align}
where $\omega$ stands for the infinite sequence $(g_1, g_2, \ldots) \in \Omega$.
Note that the ideal version of this walk, with the full dependence on the future (corresponding to the case $m=\infty$), 
has been introduced already in Section \ref{sec-ideal preturb-001}, see equation \eqref{eq-idealRWaaa-001}. 

At this point we assume that $h(x, t) = \varphi(x) \psi(t)$ for $x \in \bb P(\bb V)$ and $t \in \bb R$, 
where $\varphi$ is a non-negative H\"older continuous function on $\bb P(\bb V)$
and $\psi$ is a compactly supported non-negative continuous function on $\bb R$. 
Set, for any $x \in \bb P(\bb V)$, $t\in \bb R$ and $1 \leq n \leq m$,
\begin{align} \label{def of U^xm_n with phi-001}
U^{x,m,\varphi}_{n}(t) 
=  \int_{\bb P(\bb V^*)}
\bb E \Big[ \Big( t +\tilde S^{x,m}_{n}(\cdot,y) \Big)  \varphi(g_1\cdots g_m x);  t + \tilde S^{x,m}_{k}(\cdot,y) \geq 0, 1\leq k\leq n  \Big]  \nu^*(dy).
\end{align}
With this notation,  equation \eqref{duality id-integr-version-002} becomes 
\begin{align} \label{identity-rho-001}
\int_{\bb P(\bb V) \times \bb R}   h(x',t) \rho_{n,x}(dx',dt) 
=\int_{\bb R} \psi(t)  U^{x,n, \varphi}_{n}(t) dt. 
\end{align}
For the functions $U^{x, m, \varphi}_{n}$, we establish the following statement. 

\begin{proposition}\label{Prop-increase-sequence-fn}
Assume that  $\Gamma_{\mu}$ is proximal and strongly irreducible and that 
$\mu$ has a finite exponential moment and its Lyapunov exponent is $0$.
Let $\varphi$ be a non-negative H\"older continuous function on $\bb P(\bb V)$. 
Then, there exist constants $\beta, \gamma, c >0$ such that, 
for any $x \in \bb P(\bb V)$, $1\leq n\leq n' \leq m$ and $t \in \bb R$,  we have 
\begin{align*}
U^{x,m,\varphi}_{n}(t) \leq  U^{x,m,\varphi}_{n'}(t + c n^{-\gamma})  + c n^{-\beta} \left( 1 + \max \{t,0\}   \right) 
\end{align*}
and 
\begin{align*}
U^{x,m,\varphi}_{n'}(t) \leq  U^{x,m,\varphi}_{n} \left( t + c n^{-\gamma} \right) + c n^{-\beta} (1+\max \{t,0\}). 
\end{align*}
\end{proposition}

\begin{proof}
This statement is a translation of Theorem \ref{Pro-Appendix-Final2-Inequ} in the setting of products of random matrices. 
Indeed, we can check that the assumptions of this theorem are satisfied. 
First, we extend the sequence of perturbations $(f_{n}^{x,m})_{0 \leq n \leq m}$ as a full sequence 
$(f_{n}^{x,m})_{n \geq 0}$ as follows: 
for $n > m$, $\omega \in \Omega$ and $y \in \bb P(\bb V^*)$, we set, for example, $f_{n}^{x,m}(\omega, y)= f_{m}^{x,m}(\omega, y) = \delta(x, y)$. 
(Actually, the precise values of $f_{n}^{x,m}$ for $n > m$ will not play a role in our computations.)  
We also extend the random walk $(\tilde S^{x,m}_{n}(y))_{1\leq n \leq m}$ as a full random walk 
$(\tilde S^{x,m}_{n}(y))_{n \geq 1}$ by using the same formula as in \eqref{another form of the perturbRW-001}. 
Therefore, for any fixed $x\in \bb P(\bb V)$ and $m\geq 1$, 
the random walk $(\tilde S^{x,m}_n(y))_{n\geq 1}$ has the form required in Theorem \ref{Pro-Appendix-Final2-Inequ}. 
As already mentioned, it follows from Lemma 10.18 in \cite{BQ16b} that the cocycle $\sigma^*$ can be centered, 
meaning that there exists a continuous function 
$\psi_0$ on $\bb P(\bb V^*)$ such that, for any $y\in \bb P(\bb V^*)$, 
\begin{align*}
\int_{\bb G } \left( \sigma^*(g^{-1},y) + \psi_0(g^{-1} y)-\psi_0(y) \right)  \mu(dg) = 0.
\end{align*}
In other words, up to a bounded coboundary, the assumption \eqref{centering-001} is satisfied with the cocycle $\sigma = -\sigma^*.$
Besides, the effective central limit theorem \eqref{BEmart-001} holds true, by
the Berry-Esseen theorem in Bougerol-Lacroix \cite{Boug-Lacr85} (Theorem 5.1, page 122). 
The perturbation sequence $\mathfrak f =(f^{x,m}_n)_{n\geq 1}$ satisfies the 
bound \eqref{exp mom for g 002} by virtue of inequality \eqref{First-important-Property},
and possesses the approximation property \eqref{approxim rate for gp-002} due to Proposition \ref{Prop-projection approx-inte001}. 
Finally, the approximation property \eqref{approx property of theta-001} is satisfied thanks to Corollary \ref{corollary f_y-001}. 
\end{proof}

The further strategy of the proof is to show that the right-hand side in \eqref{identity-rho-001} 
has the same behavior as the analogous quantity based on the ideal perturbation, 
which is the one defined in \eqref{def-U-varphi-001}. To be more precise, we 
will compare \eqref{identity-rho-001} with its ideal version 
\begin{align*}
 \int_{\bb R} \psi(t) U_{n}^{\varphi}(t) dt =
    \int_{\bb R} \psi(t) 
    \int_{\bb P(\bb V^*) } \bb E \bigg[  \Big( t +\tilde S_n(\cdot,y)  \Big)  \varphi(\xi (\omega) ); \tilde \tau_{y,t} >n     \bigg] \nu^*(dy)  dt.
\end{align*}
To this aim, we will need the following lemma. 

\begin{lemma}\label{Lem-ideal-1}
There exist constants $a, b, c>0$ such that, 
for any $x \in \bb P(\bb V)$, $y \in \bb P(\bb V^*)$ and $0 \leq n \leq m$, 
\begin{align*}
\bb P \left( \left| f(\omega, y) - f_n^{x,m}(\omega, y) \right| > e^{-a(m-n)}, f_n^{x,m}(\omega, y) < \infty  \right) \leq ce^{-b(m-n)}. 
\end{align*}
\end{lemma}

\begin{proof}
Recall that for $\bb P$-almost every $\omega = ( g_1, g_2, \ldots) \in \Omega$ and any $n \geq 1$,
we have $\xi(\omega) = g_1 \cdots g_n \xi (T^n \omega)$. 
Therefore, by \cite[Proposition 14.3]{BQ16b} (see equation (14.6)), 
we may find constants $a, b, c>0$ such that for $x \in \bb P(\bb V)$ and $n \geq 1$,
\begin{align*}
\bb P \left(d \left( g_1 \cdots g_{n} x, \xi (\omega) \right) > e^{-an} \right) \leq  ce^{-bn}. 
\end{align*}
Besides, still by \cite[Proposition 14.3]{BQ16b}, we may also assume that, for any $x \in \bb P(\bb V)$, $y \in \bb P(\bb V^*)$ and $n \geq 1$, 
\begin{align*}
\bb P \left( \delta\left( g_1 \cdots g_n x, y \right) > \frac{a}{2} n \right) \leq c e^{-bn}. 
\end{align*}
By the mean value theorem, we obtain, for any $x \in \bb P(\bb V)$, $y \in \bb P(\bb V^*)$ and $n \geq 1$, 
\begin{align*}
\bb P \left( \left| \delta (g_1 \cdots g_n x, y) - \delta (\xi (\omega), y) \right| > e^{-\frac{a}{2} n},  \delta (g_1 \cdots g_n x, y) < \infty \right)
\leq 3c e^{-bn}, 
\end{align*}
which ends the proof of the lemma. 
\end{proof}

From Lemma \ref{Lem-ideal-1}, we get a comparison between $U_n^{x,m,\varphi}$ 
defined by \eqref{def of U^xm_n with phi-001} and the ideal function 
$U_n^{\varphi}$ defined by \eqref{def-U-varphi-001}.
\begin{corollary}\label{Cor-ideal}
There exist constants $a, b, c>0$ such that, 
for any $x \in \bb P(\bb V)$, $t \in \bb R$ and $0 \leq n \leq m$, 
\begin{align*}
U_n^{x,m,\varphi}(t) \leq U_n^{\varphi} \left(t + 2e^{-a(m-n)} \right) + c e^{-b(m-n)} ( \max\{t, 0\} + \sqrt{n}) 
\end{align*}
and 
\begin{align*}
U_n^{x,m,\varphi}(t) \geq U_n^{\varphi} \left(t - 2e^{-a(m-n)} \right) - c e^{-b(m-n)} ( \max\{t, 0\} + \sqrt{n}). 
\end{align*}
\end{corollary}

\begin{proof}
We only provide a proof of the upper bound, as the lower bound can be obtained in a similar manner. 
Let $a, b>0$ be as in Lemma \ref{Lem-ideal-1}. 
For $x \in \bb P(\bb V)$ and $1 \leq n \leq m$, we set 
\begin{align*}
& E_{n}^{x, m} = \Big\{(\omega, y) \in \Omega \times \bb P(\bb V^*):  
d(g_1 \cdots g_m x, \xi(\omega)) \leq e^{-am}, \notag\\
& \quad \forall k \in [0, n],  \left| f\circ \tilde T^k(\omega, y) - f_k^{x,m}\circ \tilde T^k(\omega, y) \right| \leq e^{-a(m-k)}, 
 f_k^{x,m}\circ \tilde T^k(\omega, y) < \infty \Big\}
\end{align*}
and let $(E_{n}^{x, m})^c$ be its complement. 
By Lemmas \ref{Lem-random-measure} and \ref{Lem-ideal-1}, we have 
$$\bb P \otimes \nu^* ((E_{n}^{x, m})^c) \leq c \sum_{k=0}^n e^{-b(m-k)} \leq c' e^{-b(m-n)}.$$
Then using \eqref{def of U^xm_n with phi-001},  we obtain
\begin{align*}
 U^{x,m,\varphi}_{n}(t) 
 &=  \int_{\bb P(\bb V^*)}
\int _{\Omega} \Big[ \Big( t + \tilde S^{x,m}_{n}(\omega,y)\Big) \varphi(g_1(\omega)\cdots g_m(\omega) x);  \\  
& \qquad\qquad\qquad    t + \tilde  S^{x,m}_{k}(\omega,y)  \geq 0, 1\leq k\leq n, (\omega, y) \in E_{n}^{x, m}  \Big] \bb P(d\omega)  \nu^*(dy) \notag\\
& \quad + \int_{\bb P(\bb V^*)}
\int _{\Omega} \Big[ \Big( t + \tilde S^{x,m}_{n}(\omega,y)\Big) \varphi(g_1(\omega)\cdots g_m(\omega) x);   \\
&\qquad\qquad\qquad    t + \tilde  S^{x,m}_{k}(\omega,y)
 \geq 0, 1\leq k\leq n, (\omega, y) \notin E_{n}^{x, m}  \Big] \bb P(d\omega) \nu^*(dy).
\end{align*}
The first integral is dominated by 
\begin{align*}
& U_n^{\varphi}(t + 2 e^{-a(m-n)})  + Ce^{-cm} U_n^{\bf 1}( t + 2 e^{-a(m-n)}) \notag\\
& \leq  U_n^{\varphi}(t + 2 e^{-a(m-n)})  +  Ce^{-cm} (1 + \max\{t, 0\}), 
\end{align*}
where in the last inequality we used Corollary \ref{Main th-001}. 
By the Cauchy-Schwarz inequality, Lemma \ref{Lem-trick} and inequality  \eqref{First-important-Property}, 
the second integral is bounded by 
\begin{align*}
\Big( \bb P \otimes \nu^* ((E_{n}^{x, m})^c) \Big)^{1/2}  \left( \max\{t, 0\} + c \sqrt{n} \right)
\leq  c' e^{- \frac{b}{2}(m-n)} \left( \max\{t, 0\} + c \sqrt{n} \right), 
\end{align*}
which concludes the proof. 
\end{proof}

Now we are equipped to prove Theorem \ref{Thm-measure-rho} and Corollaries \ref{Thm-measure-rho-v002}, 
\ref{Thm-measure-rho-v003}, \ref{Thm-measure-rho-v004} and \ref{Thm-measure-rho-v005}. 

\begin{proof}[Proof of Theorem \ref{Thm-measure-rho}]
Using Proposition \ref{Prop-increase-sequence-fn} and the fact that $\psi$ is non-negative, 
there exist constants $\beta, \gamma, c >0$ such that,
 for any $x \in \bb P(\bb V)$ and $n \geq 1$, 
\begin{align}\label{ideal-003}
\int_{\bb R} \psi(t) U_n^{x,n,\varphi}(t) dt 
 \leq  \int_{\bb R} \psi(t) U_{[n/2]}^{x,n,\varphi}(t + c n^{-\gamma}) dt 
 + c n^{-\beta}  \int_{\bb R} \psi(t) (1+\max \{t,0\}) dt 
\end{align}
and
\begin{align}\label{ideal-003bbb}
\int_{\bb R} \psi(t) U_n^{x,n,\varphi}(t) dt 
\geq  \int_{\bb R} \psi(t) U_{[n/2]}^{x,n,\varphi}(t - c n^{-\gamma}) dt 
 - c n^{-\beta}  \int_{\bb R} \psi(t) (1+\max \{t,0\}) dt. 
\end{align}
By \eqref{ideal-003} and Corollary \ref{Cor-ideal}, 
there exist constants $\beta, \gamma, c >0$ such that, for any $x \in \bb P(\bb V)$ and $n \geq 1$, 
\begin{align}\label{ideal-004}
\int_{\bb R} \psi(t) U_n^{x,n,\varphi}(t) dt 
& \leq    \int_{\bb R} \psi(t)  U_{[n/2]}^{\varphi} \left(t + c n^{-\gamma} \right) dt 
 + c n^{-\beta}  \int_{\bb R} \psi(t) (1+\max \{t,0\}) dt \notag\\
& =   \int_{\bb R} \psi(t - c n^{-\gamma})  U^{\varphi}_{[n/2]} \left(t \right) dt  + c n^{-\beta}  \int_{\bb R} \psi(t) (1+\max \{t,0\}) dt \notag\\
& \leq   \int_{\bb R} \psi(t)  U^{\varphi}_{[n/2]} \left( t \right) dt 
  +   \int_{\bb R} \sup_{|u| \leq c n^{-\gamma}} | \psi(t + u) - \psi(t)|  U^{\varphi}_{[n/2]} \left( t \right) dt  \notag\\
& \quad  + c n^{-\beta}  \int_{\bb R} \psi(t) (1+\max \{t,0\}) dt. 
\end{align}
By Proposition \ref{prop-reversed oper-000} and the continuity of $\psi$, we obtain, uniformly in $x \in \bb P(\bb V)$, 
\begin{align*}
& \limsup_{n \to \infty} \int_{\bb P(\bb V) \times \bb R}   h(x', t)  \rho_{n,x}(dx', dt) 
= \limsup_{n \to \infty} \int_{\bb R} \psi(t) U_n^{x,n,\varphi}(t) dt \notag\\
& \qquad \leq  \int_{\bb P(\bb V) \times \bb R}  \varphi(x') \psi(t)  \rho(dx', dt) 
=  \int_{\bb P(\bb V) \times \bb R}  h(x', t)  \rho(dx', dt). 
\end{align*}
Reasoning in the same way, and using \eqref{ideal-003bbb} instead of \eqref{ideal-003}
we get, uniformly in $x \in \bb P(\bb V)$, 
\begin{align*}
 \liminf_{n \to \infty} \int_{\bb P(\bb V) \times \bb R}   h(x', t)  \rho_{n,x}(dx', dt) 
= \liminf_{n \to \infty} \int_{\bb R} \psi(t) U_n^{x,n,\varphi}(t) dt 
 \geq  \int_{\bb P(\bb V) \times \bb R}  h(x', t)  \rho(dx', dt). 
\end{align*}
Therefore, the first assertion \eqref{exist of measure rho-001} follows
for functions of the form $h(x, t) = \varphi(x) \psi(t)$,
where $\varphi$ is a non-negative H\"older continuous on $\bb P(\bb V)$, 
and $\psi$ is a compactly supported non-negative continuous on $\bb R$.
 The extension to any continuous compactly supported function $h$ on $\bb P(\bb V) \times \bb R$
follows by standard approximation arguments. 
The shows \eqref{exist of measure rho-001}. 
The proof of \eqref{exist of measure rho-002} can be done in the same way. 
\end{proof}

\begin{proof}[Proof of Corollary \ref{Thm-measure-rho-v002}]
This is a consequence of Proposition \ref{prop-reversed oper-000}.
\end{proof}

\begin{proof}[Proof of Corollary \ref{Thm-measure-rho-v003}]
Define a Borel measure $\rho'$ on $\bb P(\bb V) \times \bb R$ by setting, for any non-negative Borel measurable function 
$h$ on $\bb P(\bb V) \times \bb R$, 
\begin{align*}
\int_{\bb P(\bb V) \times \bb R} h(x,t) \rho'(dx,dt) = 
\int_{\bb P(\bb V) \times [0,\infty) } \bb E h \Big(g_1 x, t + \sigma(g_1, x) \Big) \rho(dx,dt). 
\end{align*}
We claim that $\rho'$ is a Radon measure on $\bb P(\bb V) \times \bb R$, meaning that it is finite on compact sets. 
Indeed, since $\mu$ has a finite exponential moment, 
by the bound \eqref{bound-U-app} of Proposition \ref{prop-reversed oper-000}, 
there exists a constant $\alpha>0$ such that, for any $a\geq 0$, 
 \begin{align*} 
\rho'(\bb P(\bb V) \times (-\infty,a]) 
& = \int_{\bb P(\bb V) \times [0,\infty) } \bb P \Big( t + \sigma(g_1, x)  \leq a \Big) \rho(dx,dt) \\
& \leq c \int_{0}^{\infty} e^{-\alpha (t-a) } (1+t) dt <\infty. 
\end{align*}

Now we show that the Radon measure $\rho'$ coincides with the Radon measure $\rho$ 
on continuous compactly supported functions. 
Indeed, fix $x\in \bb P(\bb V)$. 
For any continuous compactly supported function $h$ on $\bb P(\bb V) \times \bb R$,
we have, by Theorem \ref{Thm-measure-rho} and the Markov property,
\begin{align} \label{intRPhhh000aa}
& \int_{\bb P(\bb V) \times [0,\infty) }  h (x',t')   \rho(dx',dt') \notag \\
& = \lim_{n\to\infty} 
\int_{\bb P(\bb V) \times [0,\infty) } t \,  \bb E \Big( h(g_n\cdots g_1 x, t + \sigma(g_n\cdots g_1, x)) ;  \tau_{x,t} >n-1 \Big) dt \notag \\
& = \lim_{n\to\infty} 
\int_{\bb P(\bb V) \times [0,\infty) } t \,  \bb E  \Big( P h (g_{n-1}\cdots g_{1} x, t + \sigma(g_{n-1}\cdots g_1, x)) ;  \tau_{x,t} >n-1 \Big) dt,
\end{align}
where we used the notation
\begin{align*} 
Ph(x',t') = \bb E h \Big( g_1 x',t'+\sigma(g_1,x') \Big).
\end{align*}
Note that $Ph$ is continuous, but we cannot apply directly Theorem \ref{Thm-measure-rho}, since the function $(x',t') \mapsto Ph(x',t') \mathds 1_{\{t'\geq 0 \}}$ 
may not be compactly supported. 
Nevertheless, for any $a\geq 0$, by Theorem \ref{Thm-measure-rho} and Corollary \ref{Thm-measure-rho-v002},
we get
\begin{align} \label{intRPhhh000} 
\int_{\bb P(\bb V) \times [0,a] }  P h (x',t') \rho(dx',dt')  
& = \lim_{n\to\infty} 
\int_{\bb R_+} t \bb E  \Big(   P h (g_{n-1}\cdots g_{1} x, t + \sigma(g_{n-1}\cdots g_1, x)) ; \notag \\
& \qquad\qquad  \tau_{x, t} >n-1, t+\sigma(g_{n-1}\cdots g_1, x))\leq a  \Big) dt. 
\end{align}
Now we notice that, since $h$ has compact support and $\mu$ has finite exponential moments, 
we may find constants $\alpha, c >0$ such that for any $(x',t') \in \bb P(\bb V) \times \bb R$, 
\begin{align*} 
| Ph(x',t') |   \leq c e^{-\alpha t'}   
\end{align*}
so that, using the bound \eqref{bound-U-app}, we arrive at
\begin{align} \label{intRPhhh001}
\int_{\bb P(\bb V) \times [a,\infty) }  P h (x',t') \rho(dx',dt') \leq \int_{\bb P(\bb V) \times [a,\infty) } c (1+t') e^{-\alpha t'} dt'  
\leq c e^{-\alpha a}.
\end{align}
Besides, by Lemma \ref{lemma-duality-for-product-001} and \eqref{def of U^xm_n with phi-001}, we deduce that, for $n \geq 1$, 
\begin{align} \label{intRPhhh002}
& \int_{\bb R_+}  t \bb E  \left(   P h (g_{n-1}\cdots g_{1} x, t + \sigma(g_{n-1}\cdots g_1, x)) ;  \tau_{x,t} >n-1,
 t + \sigma(g_{n-1}\cdots g_1, x)) > a  \right) dt \notag \\
& \leq c \int_{\bb R_+}  t \bb E  \left(   e^{-\alpha (t + \sigma(g_{n-1}\cdots g_1, x) ) }  ;  \tau_{x,t} >n-1,
t+\sigma(g_{n-1}\cdots g_1, x)) > a  \right) dt \notag \\
& = c \int_a^\infty   e^{-\alpha t}    U_{n-1}^{x,n-1,\bf 1} (t)  dt \notag \\
& \leq c \int_a^\infty   e^{-\alpha t}    (1+t)  dt \leq c e^{-\alpha a},
\end{align}
where for the last line we used Corollary \ref{Main th-001}.
Combining \eqref{intRPhhh000}, \eqref{intRPhhh001} and \eqref{intRPhhh002} yields 
\begin{align*} 
& \int_{\bb P(\bb V) \times \bb R_+ }  P h (x',t') \rho(dx',dt') \notag \\
& = \lim_{n\to\infty} 
\int_{\bb R_+} t \bb E  \Big(   P h (g_{n-1}\cdots g_{1} x, t + \sigma(g_{n-1}\cdots g_1, x)); \tau_{x,t} >n-1  \Big) dt, 
\end{align*}
which together with  \eqref{intRPhhh000aa} implies $\rho(h)=\rho'(h)$.
As the Radon measures $\rho$ and $\rho'$ are equal on continuous compactly supported functions, they are also equal 
on non-negative Borel measurable functions, which is the assertion of Corollary \ref{Thm-measure-rho-v003}. 
\end{proof}

\begin{proof}[Proof of Corollary \ref{Thm-measure-rho-v004}]
The assertions \eqref{limit of rho-001} and \eqref{limit of rho-002} follow from Proposition \ref{prop-reversed oper-000}. 
To show \eqref{non-degeneracy of W-001}, we choose a non-negative continuous function $\psi$ on $\bb R$, supported in $[0,1]$, with 
$\int_{0}^{1} \psi(t) dt =1$. Then, as $W$ is non-decreasing, for any $t > 0$, we deduce that 
\begin{align*} 
\frac{W(t)}{t} 
\leq \frac{1}{t} \int_{\bb R} \psi(t'-t) W(t')dt' 
= \frac{1}{t} \int_{\bb P(\bb V) \times \bb R} \psi(t'-t) \rho(dx',dt')
\leq \frac{W(t+1)}{t}.  
\end{align*}
Since, by \eqref{limit of rho-001}, we have
\begin{align*} 
\lim_{t\to\infty}  \frac{1}{t} \int_{\bb P(\bb V) \times \bb R} \psi(t'-t) \rho(dx',dt') =\int_{0}^{1} \psi(t) dt =1,
\end{align*} 
the assertion \eqref{non-degeneracy of W-001} follows. 
\end{proof}

\begin{proof}[Proof of Corollary \ref{Thm-measure-rho-v005}]
We shall first prove \eqref{property-W-negative-t}.
We claim that there exists $c>0$ such that for any $s \geq 0$, one has
\begin{align}\label{rho-integrability-001}
\rho(\bb P(\bb V) \times (-\infty, -s]) \leq c e^{-\alpha s}, 
\end{align}
where $\alpha>0$ is the exponent from \eqref{Exponential-moment}. 
Indeed, for $0 \leq s < s'$, consider the function 
\begin{align*}
h_{s, s'}(x, t) = \bb P \left( t + \sigma(g, x) \in [-s', -s] \right) \mathds 1_{\{t \geq 0\}}. 
\end{align*}
Then, by \eqref{exist of measure rho-001}, we get that for any $x \in \bb P(\bb V)$, 
\begin{align*}
\rho(\bb P(\bb V) \times [-s', -s]) 
& =  \lim_{n \to \infty} \int_{0}^{\infty} t \bb P \left( t + \sigma(g_n \cdots g_1, x) \in [-s', -s], \tau_{x, t} > n-1   \right) dt \notag\\
& = \lim_{n \to \infty} \int_{0}^{\infty} t \bb E \Big[ h_{s, s'} (g_{n-1} \cdots g_1 x, t + \sigma(g_{n-1} \cdots g_1, x) ); \tau_{x, t} > n-2  \Big] dt, 
\end{align*}
where in the last equality we used the Markov property. 
By our moment assumption \eqref{Exponential-moment} and Markov's inequality,  
there exists $c>0$ such that, for any $x \in \bb P(\bb V)$ and $t \geq 0$, 
\begin{align*}
h_{s, s'}(x, t) \leq c e^{-\alpha (s+t)}. 
\end{align*}
Thus we get
\begin{align*}
\rho(\bb P(\bb V) \times [-s', -s]) 
& \leq  c e^{-\alpha s} \limsup_{n \to \infty}  \int_{0}^{\infty} t  \bb E \Big[ e^{-\alpha ( t + \sigma(g_{n-1} \cdots g_1, x) )}; \tau_{x, t} > n-2  \Big] dt \notag\\
& =  c e^{-\alpha s}  \limsup_{n \to \infty}  \int_{0}^{\infty} e^{-\alpha t} U_{n-1}^{x, n-1, \bf{1}}(t) dt, 
\end{align*}
where the last equality holds due to Lemma \ref{lemma-duality-for-product-001}. 
By Corollary \ref{Main th-001} (see the proof of Proposition \ref{Prop-increase-sequence-fn} for the reason why this result can be applied here), we get 
\begin{align*}
\rho(\bb P(\bb V) \times [-s', -s])  
\leq  c e^{-\alpha s}  \int_{0}^{\infty} e^{-\alpha t} (1 + t) dt
\leq c' e^{-\alpha s}. 
\end{align*}
By letting $s' \to \infty$, we obtain \eqref{rho-integrability-001}. 
Since the function $W$ is non-decreasing, for any $s\geq 0$, we have 
\begin{align*}
W(-s-1) \leq  \rho(\bb P(\bb V) \times [-s-1, -s])  \leq c' e^{-\alpha s} 
\end{align*} 
and \eqref{property-W-negative-t} follows.

To conclude the proof of the corollary, it remains to show that $\rho( \bb P(\bb V)  \times (-\infty,0)  ) >0$.
For $x\in \bb P(\bb V)$ and $t\in \bb R$, set $h_0(x,t) = \mathds 1_{\{t < 0\}}$ and, for $n\geq 1$,
define $h_n(x,t) = \bb P \left(\tau_{x,t} = n  \right)$.
By Corollary \ref{Thm-measure-rho-v003}, for any $n\geq 1$, we have 
\begin{align*} 
\int_{\bb P(\bb V)  \times \bb R }  h_{n-1} (x,t)  \rho(dx,dt) 
= \int_{\bb P(\bb V)  \times [0, \infty) }  h_{n} (x,t)  \rho(dx,dt). 
\end{align*}
This gives 
\begin{align*} 
\rho( \bb P(\bb V)  \times (-\infty,0)  ) = \int_{\bb P(\bb V)  \times [0, \infty) } \bb P \left(\tau_{x,t} = n  \right) \rho(dx,dt).
\end{align*}
 Since, for any $x\in \bb P(\bb V)$ and $t\in \bb R$, the stopping time $\tau_{x,t}$ is $\bb P$-almost surely finite, 
 it follows that 
 \begin{align*}
\bb P(\bb V)  \times [0,\infty) = \bigcup_{n\geq 1} \{ (x,t) \in   \bb P(\bb V)  \times [0,\infty) :  \bb P \left(\tau_{x,t} = n  \right) >0  \}.
\end{align*}
By \eqref{non-degeneracy of W-001} of Corollary \ref{Thm-measure-rho-v004}, we know that $\rho$ is not zero on $\bb P(\bb V)  \times [0,\infty)$.
Therefore, there exists $n\geq 1$ such that 
\begin{align*}
\rho\left( \{ (x, t) \in   \bb P(\bb V)  \times [0,\infty) :  \bb P \left(\tau_{x, t} = n  \right) >0  \} \right) > 0, 
\end{align*}
which implies that 
\begin{align*}
\int_{\bb P(\bb V)  \times [0,\infty) } \bb P \left(\tau_{x, t} = n  \right) \rho(dx, dt) > 0.
\end{align*}
The conclusion follows.
\end{proof}


\section{Random walks with perturbation depending on the future}\label{sec invar func}

\subsection{Setting and assumptions}\label{subsec set assump}
In the previous section, we deduced Theorem \ref{Thm-measure-rho} from Theorem \ref{Pro-Appendix-Final2-Inequ}. 
The remainder of the paper is devoted to stating and proving Theorem \ref{Pro-Appendix-Final2-Inequ}. 
This theorem will be presented within an abstract framework, which we introduce below. 
The strategy of the proof relies on extending the approach of Denisov and Wachtel \cite{DW15}.

In the sequel, let $\bb G$ denote a general second countable locally compact group, and $\mu$ 
a Borel probability measure on $\bb G$. 
Consider the measurable space $\Omega= \bb G^{\bb N^*}$ endowed 
with the product probability measure $\bb P = \mu^{\otimes \bb N^*}$. 
For any $k\geq 1$, denote by $g_k$ the coordinate map $\omega\mapsto g_k(\omega)$ on $\Omega$. 
Correspondingly, $g_1,g_2,\ldots$ will form a  sequence of independent and identically distributed elements of $\bb G$ with law  $\mu$.
The expectation with respect to $\bb P$ is denoted by $\bb E$.

We fix a second countable locally compact space $\bb X$ equipped with a continuous action of the group $\bb G$.
Assume that the space $\bb X$ is equipped with a $\mu$-stationary probability measure $\nu$. 
We also fix a continuous cocycle $\sigma: \bb G \times \bb X \to \bb R$, meaning that  
for any $g_1, g_2 \in \bb G$ and $x\in \bb X$,
\begin{align} \label{def-cocycle}
\sigma(g_2 g_1, x) = \sigma(g_2,  g_1 x) + \sigma(g_1, x). 
\end{align}
Assume that $\sigma$ has an exponential moment with respect to $\mu$: there exists a constant $\alpha_0>0$ such that 
\begin{align} \label{exp mom for f 001}
\int_{\bb G} e^ {\alpha_0 \sup_{x\in \bb X} | \sigma (g,x) |} \mu (dg) < \infty. 
\end{align}
Assume also that $\sigma$ is centered in the following strong sense: for any $x\in \bb X$,
\begin{align} \label{centering-001}
\int_{\bb G} \sigma (g,x) \mu (dg) = 0. 
\end{align}
On the probability space $(\Omega, \bb P)$, for any $x\in \bb X$, consider the random walk $(S^x_n)_{n\geq 1}$:
for $\omega=(g_1, g_2, \ldots)  \in \Omega$,
\begin{align} \label{random walk S^x_n-001}
S^x_n (\omega): = \sigma(g_n  \cdots  g_1, x ) = \sum_{i=1}^n \sigma\left(g_i, g_{i-1}\cdots g_1 x\right), \quad n\geq 1.
\end{align}
We equip the space $\Omega\times \bb X$ with the shift map $T$ which is defined as follows:  
for $x\in \bb X$ and $\omega=(g_1, g_2, \ldots)  \in \Omega$,
\begin{align}\label{def-T-Omega-X}
T(\omega, x): = ((g_2, g_3, \ldots), g_1 x). 
\end{align}

In what follows, we will study the properties of the random walk $(\widetilde S_n^x)_{n\geq 1}$ 
with perturbations depending on the future. 
More precisely,  given a sequence $\mathfrak f = (f_n)_{n \geq 0}$ of 
measurable functions $f_n: \Omega\times \bb X \to \bb R$, the walk
$(\widetilde  S_n^x)_{n\geq 1}$ is defined as follows: 
for any $\omega\in \Omega$ and $x\in \bb X$,  
\begin{align} \label{the perturbed random walk 001}
\widetilde  S^x_n(\omega) 
= S^x_n(\omega) + f_n \circ T^n(\omega, x) - f_0(\omega, x), \quad n\geq 1.  
\end{align}
The primary challenge in analyzing the walk \eqref{the perturbed random walk 001} 
stems from the dependence of the functions $f_n$ on the entire sequence $(g_{i})_{i \geq 1}$, 
including the future coordinates $(g_k)_{k > n}$.  
To address this intricate dependence, we introduce an abstract approximation property below. 
This property asserts that perturbation functions with an infinite number of coordinates 
can be effectively approximated by functions dependent on a finite subset of coordinates. 
This concept stands as one of the key ideas in this paper.

First we assume that the sequence $\mathfrak f = (f_n)_{n \geq 0}$ satisfies a uniform exponential moment assumption. 
More precisely, 
we assume that there exists a constant $\alpha > 0$ such that 
\begin{align} \label{exp mom for g 002}
C_{\alpha}(\mathfrak f) 
= \sup_{n \in \bb N} \int_{\bb X} \int_{\Omega} e^{ \alpha |f_n(\omega,x)| } \bb P(d\omega)  \nu(dx)
< \infty.
\end{align}
For any $p\geq 1$, let $\mathscr A_p$ be the $\sigma$-algebra on $\Omega$ 
spanned by the coordinate functions $\omega \in \Omega  \mapsto  g_{i} \in \bb G$ for $1\leq i \leq p$.

For any $p\geq 1$, we consider a finite-size approximation sequence $\mathfrak f_p = (f_{n, p})_{n \geq 0}$, 
where the functions $f_{n, p}$ depend on a finite subset of coordinates in $\Omega$. 
Specifically, given $p \geq 1$,  
we define the approximation $f_{n, p}$ of $f_n$ by setting, for $\omega=(g_1,g_2,\ldots) \in \Omega$
and $x \in \bb X$, 
\begin{align} \label{def-approxi-fnp}
f_{n,p}(\omega,x) 
& = \int_{\bb G^{\bb N^*}} f_n(g_1, \ldots, g_p, g_{p+1}', g_{p+2}', \ldots, x) \mu(dg_{p+1}') \mu(dg_{p+2}') \ldots  \notag\\
& = \bb E( f_{n}(\cdot, x) | \scr A_p)(\omega),  
\end{align}
where the integral makes sense due to condition \eqref{exp mom for g 002}. 
With this definition, the finite-size approximation property of the sequence $\mathfrak f = (f_n)_{n \geq 0}$ is stated as follows: 
 there are constants $\alpha, \beta >0$ such that 
\begin{align} \label{approxim rate for gp-002} 
D_{\alpha,\beta}(\mathfrak f) 
= \sup_{n \in \bb N} \sup_{p \in \bb N^*}   e^{\beta p} 
\left( \int_{\bb X}  \int_{\Omega}  
 e^{\alpha |f_n(\omega, x) - f_{n,p}(\omega, x) |} \bb P(d\omega) \nu(dx)  - 1\right)     <\infty.  
\end{align}
It is evident that if $D_{\alpha,\beta}(\mathfrak f) < \infty$ for some $\alpha, \beta>0$, then for any $\beta' \in (0, \beta]$,
we also have $D_{\alpha,\beta'}(\mathfrak f) < \infty$.
Additionally, it is straightforward to see that the constant $\alpha$ in conditions \eqref{exp mom for g 002} and \eqref{approxim rate for gp-002} can be assumed to be identical. The latter assumption will be consistently applied throughout the paper.
 
In Proposition \ref{Prop-projection approx-inte001}, 
we have already established that property \eqref{approxim rate for gp-002} holds for the random walks with perturbations considered in 
Sections \ref{sec-ideal preturb-001} and \ref{sec-reversed random walk-001}.

In the following we will need a normal approximation result for the sum $S^x_{n}$: 
there exist constants $c, \bf{v}, \epsilon > 0$ such that
for any $a_1 < a_2$, $x\in \bb X$ and $n\geq 1$, 
\begin{align} \label{BEmart-001}
\left\vert  \mathbb{P} \left( \frac{S^x_{n}}{\sqrt{n}}\in [a_1, a_2] \right) 
- \int_{a_1}^{a_2} \phi_{\bf{v}^2} (u) du 
\right\vert \leq \frac{c}{n^{\epsilon}},  
\end{align}
where $\phi_{\bf{v}^2}$ is the normal density of mean $0$ and variance $\bf{v}^2$.

\subsection{Quasi-monotonicity bounds}\label{subsec invar func fixed}
We shall analyze the behavior of the perturbed random walk, 
focusing on the first time the process exits the non-negative half-line $\mathbb{R}_{+}= \left[ 0,\infty \right)$. 
Formally, for any $x\in \bb X$ and $t\in \bb R$, consider the first time
when the perturbed random walk $( t + \widetilde S^x_{k} ) _{k\geq 1}$ (see \eqref{the perturbed random walk 001}) exits $\mathbb{R}_{+}$: 
\begin{align} \label{def-stop time with preturb-001}
\tau_{x,t}^{\mathfrak f} = \min \left\{ k\geq 1: t+ \widetilde S^x_{k} < 0\right\}.
\end{align}
We then introduce the following function:
\begin{align}\label{def-U-f-n-t-001}
U^{\mathfrak f}_n(t) 
= \int_{\bb X} \bb E \left( t +  \widetilde S^{x}_n;  \tau^{\mathfrak f}_{x,t} > n \right) \nu(dx). 
\end{align}

Upon initial consideration, our objective is to establish asymptotic bounds for $U^{\mathfrak f}_n(t)$.
Nevertheless, by closely examining identity \eqref{identity-rho-001} and equation \eqref{def of U^xm_n with phi-001}, 
which precisely delineates the quantity subject to the application of Theorem \ref{Pro-Appendix-Final2-Inequ}, 
it becomes clear that, in addition to monitoring $U^{\mathfrak f}_n(t)$, 
we must also control its modified version, which we will define in what follows.

Let $\theta$ be a bounded measurable non-negative function on $\Omega$, which in the sequel will be called twist function. 
In analogy with \eqref{def-U-varphi-001}, 
for $t \in \bb R$, $n \geq 1$ and $\theta \in L^{\infty}(\Omega, \bb P)$, we set
\begin{align}\label{def-U-f-n-t-theta-001}
U^{\mathfrak f, \theta}_n(t)
=  \int_{\bb X} \bb E \left((t +  \tilde S^{x}_n )  \theta ;  \tau^{\mathfrak f}_{x,t} > n \right) \nu(dx). 
\end{align}
For any $\omega=(g_1,g_2,\ldots) \in \Omega$ and $p\geq 1$, let 
\begin{align}\label{def-theta-p}
\theta_p(\omega)= \bb E(\theta | \scr A_p)(\omega)
=\int_{\bb G^{\bb N^*}} \theta(g_1, \ldots, g_p, g_{p+1}', g_{p+2}', \ldots) \mu(dg_{p+1}') \mu(dg_{p+2}') \ldots.
\end{align}
We shall assume that the function $\theta$ itself satisfies a finite-size 
approximation condition as stated below: 
there exists a constant $\beta >0$ such that 
\begin{align} \label{approx property of theta-001}
N_{\beta}(\theta) : = \sup_{p \geq 1} e^{\beta p} \bb E | \theta - \theta_p | < \infty. 
\end{align}
The constant $\beta$ in both conditions \eqref{approx property of theta-001} and \eqref{approxim rate for gp-002} can be assumed to be same.
As usual, we denote $\|\theta\|_{\infty} = {\rm esssup} |\theta|$. 

The principal outcome in this 
part of the article is contained in the following theorem, which can be regarded as a quantified version of the statement 
that the sequence of Radon measures $U^{\mathfrak f, \theta}_n(t) dt$ converges vaguely on $\bb R$. 
This result asserts that the sequence of functions $U^{\mathfrak f, \theta}_n$ converges 
in the space of distributions on $\bb R$. 
Moreover, this convergence is effective, with constants depending uniformly on $\mathfrak f$ and $\theta$ under certain bounds.

\begin{theorem}\label{Pro-Appendix-Final2-Inequ}
Suppose that the cocycle $\sigma$ admits finite exponential moments \eqref{exp mom for f 001} 
and is centered \eqref{centering-001}. 
We also suppose that the effective central limit theorem \eqref{BEmart-001} holds. 
For any $\beta > 0$ and $B \geq 1$, there exist $A, b, \gamma >0$ with the following property. 
Assume that $\mathfrak f = (f_n)_{n \geq 0}$ is a sequence of  measurable functions on $\Omega \times \bb X$ 
satisfying the moment condition \eqref{exp mom for g 002} and the approximation property \eqref{approxim rate for gp-002} with 
$C_{\alpha}(\mathfrak f) \leq B$ and $D_{\alpha,\beta}(\mathfrak f) \leq B$. 
Assume that $\theta \in L^{\infty}(\Omega, \bb P)$ satisfies the 
approximation property \eqref{approx property of theta-001},
with $\|\theta\|_{\infty} \leq B$ and $N_{\beta}(\theta) \leq B$. 
Then, for any $1 \leq n\leq m$ and $t \in \bb R$, we have
\begin{align} \label{bound with m for U-105-01-2}
U^{\mathfrak f, \theta}_n(t) \leq  U^{\mathfrak f, \theta}_{m}(t + A n^{-\gamma}) 
 + A n^{-b} \left( 1 + \max \{t,0\}  \right)
\end{align}
and 
\begin{align}\label{bound with m for U-105-01-3}
U^{\mathfrak f, \theta}_{m} (t) \leq 
U^{\mathfrak f, \theta}_{n} \left( t + A n^{-\gamma} \right) + A n^{-b} (1+\max \{t,0\}). 
\end{align}
\end{theorem}

The above inequalities are inspired by the analogous ones from \cite{DW15, GLP17, GLL18Ann}.
The main difference is the presence of the perturbation term $A n^{-\gamma}$ affecting the variable $t$ 
on the right-hand sides of \eqref{bound with m for U-105-01-2} and \eqref{bound with m for U-105-01-3}. 
Because of the presence of this perturbation, the convergence of  $U^{\mathfrak f, \theta}_{n} (t)$ holds in the distributional sense. 
This is similar to  the case of hyperbolic dynamical systems discussed in \cite{GQX23}.

As a consequence of the previous theorem, we can deduce that the sequence of Radon measures $U_n^{\mathfrak f, \theta}(t)  dt$ on $\bb R$ 
converges vaguely to a limiting Radon measure $U^{\mathfrak f, \theta}(t)  dt$. 
\begin{corollary} \label{Main th-003}
Suppose that the cocycle $\sigma$ admits finite exponential moments \eqref{exp mom for f 001}
and is centered \eqref{centering-001}.
We also suppose that the effective central limit theorem \eqref{BEmart-001} holds.
Assume that $\mathfrak f = (f_n)_{n \geq 0}$ is a sequence of  measurable functions on $\Omega \times \bb X$ 
satisfying the moment condition \eqref{exp mom for g 002} and the approximation property \eqref{approxim rate for gp-002}. 
Let $\theta \in L^{\infty}(\Omega, \bb P)$ be such that \eqref{approx property of theta-001} holds.
Then, there exists a measurable function $U^{\mathfrak f, \theta}: \bb R \to \bb R_+$ such that 
for any continuous compactly supported test function $\varphi$ on $\bb R$, we have 
\begin{align} \label{MAIN_GOAL-theta-003}
\lim_{n\to\infty } \int_{\bb R}  \varphi(t) U_n^{\mathfrak f, \theta}(t)  dt
=  \int_{\bb R} \varphi(t)  U^{\mathfrak f, \theta}(t)  dt.
\end{align}
Moreover, the following holds: 
\begin{align}\label{MAIN_GOAL-theta-004}
\lim_{t \to \infty} \frac{1}{t} U^{\mathfrak f, \theta}(t) = \bb E \theta. 
\end{align}
\end{corollary}

In the proof of this corollary we will make use of the following upper bound of $U^{\mathfrak f}_n(t)$. 

\begin{corollary} \label{Main th-001}
Suppose that the cocycle $\sigma$ admits finite exponential moments \eqref{exp mom for f 001}
and is centered \eqref{centering-001}.
We also suppose that the effective central limit theorem \eqref{BEmart-001} holds.
Assume that $\mathfrak f = (f_n)_{n \geq 0}$ is a sequence of  measurable functions on $\Omega \times \bb X$ 
satisfying the moment condition \eqref{exp mom for g 002} and the approximation property \eqref{approxim rate for gp-002}. 
Then, there is a constant $c>0$ such that for any $t\in \bb R$ and $n \geq 1$,
\begin{align} \label{MAIN_GOAL-001}
U^{\mathfrak f}_n(t) \leq c \left( 1 + \max \{t, 0 \}  \right).
\end{align}
\end{corollary}

\begin{proof}[Proof of Corollary \ref{Main th-001}]
Applying Theorem \ref{Pro-Appendix-Final2-Inequ} with $\theta = 1$, we get that for any $n \geq 1$ and $t \in \bb R$, 
\begin{align*}
U^{\mathfrak f}_n(t) \leq  U^{\mathfrak f}_{1}(t + A) 
 + A \left( 1 + \max \{t,0\}  \right). 
\end{align*}
From our moment assumptions, we have $U^{\mathfrak f}_{1}(t + A) \leq c \left( 1 + \max \{t,0\}  \right)$ and the conclusion follows. 
\end{proof}

We will make use of the following elementary fact from the  theory of distributions, which we reproduce from \cite[Lemma 3.7]{GQX23}:
\begin{lemma}\label{Lem_Conver_Func}
Let $(U_n)_{n \geq 1}$ be a sequence of non-decreasing functions on $\bb R$.
Assume that for every continuous compactly supported function $\vphi$ on $\bb R$, as $n \to \infty$, 
the sequence $\int_{\bb R} U_n(t) \vphi (t) dt$ admits a finite limit. 
Then, there exists a unique right continuous and non-decreasing function $U$ on $\bb R$ such that 
for any continuous compactly supported function $\vphi$, we have
\begin{align*}
\lim_{n \to \infty} \int_{\bb R} U_n(t) \vphi (t) dt = \int_{\bb R} U(t) \vphi (t) dt. 
\end{align*}
\end{lemma}

\begin{proof}[Proof of Corollary \ref{Main th-003}]
We fix a non-negative continuous compactly supported function $\varphi$ on $\bb R$. 
Let $b, \gamma, A > 0$ be as in Theorem \ref{Pro-Appendix-Final2-Inequ}.
Then, for any $1\leq n\leq m$, we have
\begin{align*} 
& \int_{\bb R} U_n^{\mathfrak f, \theta} (t) \varphi(t + A n^{-\gamma}) dt \notag\\
&\leq \int_{\bb R} U_m^{\mathfrak f, \theta} \left(t+ A n^{-\gamma}\right) \varphi(t + A n^{-\gamma}) dt 
+  A n^{-b}  \int_{\bb R} \left( 1 + \max \{t,0\}  \right) \varphi(t + A n^{-\gamma}) dt \notag\\
&= \int_{\bb R} U_m^{\mathfrak f, \theta} (t) \varphi(t) dt 
+  A  n^{-b}  \int_{\bb R} \left( 1 + \max \{t,0\}   \right) \varphi(t + A n^{-\gamma}) dt. 
\end{align*}
Taking the limit as $m \to\infty$, we obtain that, for any $n\geq 1$,
\begin{align} \label{eq sup-inf-001}
& \int_{\bb R} U_n^{\mathfrak f, \theta} (t) \varphi(t + A n^{-\gamma}) dt  \notag\\
& \leq \liminf_{m\to\infty}  \int_{\bb R} U_m^{\mathfrak f, \theta} (t) \varphi(t) dt 
+  A  n^{-b}  \int_{\bb R} \left( 1 + \max \{t,0\}   \right) \varphi(t + A n^{-\gamma}) dt. 
\end{align}
By the continuity of $\varphi$ and the uniform bound of Corollary \ref{Main th-001}, we have that, as $n\to\infty$,
\begin{align*}
\int_{\bb R} U_n^{\mathfrak f, \theta} (t) \varphi(t + A n^{-\gamma}) dt  - \int_{\bb R} U_n^{\mathfrak f, \theta} (t) \varphi(t) dt \to 0.
\end{align*}
Hence, taking the limit as $n\to\infty$ in \eqref{eq sup-inf-001}, we get 
\begin{align} \label{eq sup-inf-002}
\limsup_{n\to\infty}  \int_{\bb R} U_n^{\mathfrak f, \theta} (t) \varphi(t) dt
\leq \liminf_{m\to\infty} \int_{\bb R} U_m^{\mathfrak f, \theta} (t) \varphi(t) dt.
\end{align}
Therefore, the integral $\int_{\bb R} U_n^{\mathfrak f, \theta} (t) \varphi(t) dt$ admits a limit as $n\to\infty$, 
and this limit is finite by Corollary \ref{Main th-001}. 
Hence, the assertion \eqref{MAIN_GOAL-theta-003} follows from Lemma \ref{Lem_Conver_Func}.

Now we prove the assertion \eqref{MAIN_GOAL-theta-004}. 
First, note that, for any $n \geq 1$, we have 
\begin{align}\label{convergence-U-n-t-001}
\lim_{t \to \infty} \frac{1}{t} U_n^{\mathfrak f, \theta}(t) = \bb E \theta. 
\end{align}
Indeed, by \eqref{def-U-f-n-t-theta-001}, we write, for $t \geq 0$, 
\begin{align*}
\big| U_n^{\mathfrak f, \theta}(t) - t \bb E \theta  \big| 
 \leq t  \|\theta\|_{\infty} \int_{\bb X} \bb P \big(   \tau_{x,t}^{\mathfrak f} \leq n \big)  \nu(dx)
 + \|\theta\|_{\infty} \int_{\bb X} \bb E  \big| \tilde S^{x}_n  \big|  \nu(dx). 
\end{align*}
By \eqref{def-stop time with preturb-001}, we have 
\begin{align*}
\int_{\bb X} \bb P \big(   \tau_{x,t}^{\mathfrak f} \leq n \big)  \nu(dx)
 \leq \int_{\bb X} \bb P \Big(  \max_{1\leq k\leq n}|\tilde S^{x}_k| >t \Big)  \nu(dx) 
\leq \sum_{k=1}^n \int_{\bb X} \bb P \big(   \big| \tilde S^{x}_k  \big| > t  \big)  \nu(dx), 
\end{align*}
which implies that 
\begin{align*}
\lim_{t \to \infty}  \int_{\bb X} \bb P \big(   \tau_{x,t}^{\mathfrak f} \leq n \big)  \nu(dx) = 0. 
\end{align*}
Hence \eqref{convergence-U-n-t-001} holds.
To prove \eqref{MAIN_GOAL-theta-004}, 
we fix a continuous compactly supported function $\varphi$ on $\bb R$ which is non-negative
with support in $[0, 1]$. We also assume that $\int_0^1 \varphi(t) dt = 1$. 
In particular, for any $t \in \bb R$, we have
\begin{align*}
U^{\mathfrak f, \theta}(t) = \int_{\bb R} \varphi(t'-t) U^{\mathfrak f, \theta}(t) dt' 
\leq \int_{\bb R} \varphi(t'-t) U^{\mathfrak f, \theta}(t') dt'. 
\end{align*}
By Theorem \ref{Pro-Appendix-Final2-Inequ}, for any $n \geq 1$ and $t \geq 0$, 
we derive that 
\begin{align*}
\int_{\bb R} \varphi(t'-t) U^{\mathfrak f, \theta}(t') dt'
& = \lim_{m \to \infty} \int_{\bb R} \varphi(t'-t) U_m^{\mathfrak f, \theta}(t') dt' \notag\\
& \leq  \int_{\bb R} \varphi(t'-t) U_{n}^{\mathfrak f, \theta}(t' + A n^{-\gamma}) dt' + A n^{-b} \int_{\bb R} \varphi(t'-t) (1 + t') dt'  \notag\\
& \leq  \sup_{t \leq t' \leq t + 1 + A n^{-\gamma}} U_{n}^{\mathfrak f, \theta}(t') + A (t +2) n^{-b}. 
\end{align*}
From \eqref{convergence-U-n-t-001}, we get 
\begin{align*}
\limsup_{t \to \infty} \frac{1}{t}  U^{\mathfrak f, \theta}(t) \leq \bb E \theta + A n^{-b}. 
\end{align*}
As this is true for any $n \geq 1$, we get $\limsup_{t \to \infty} \frac{1}{t}  U^{\mathfrak f, \theta}(t) \leq \bb E \theta$. 
Reasoning in the same way, we can prove that $\liminf_{t \to \infty} \frac{1}{t}  U^{\mathfrak f, \theta}(t) \geq \bb E \theta$ 
and so \eqref{MAIN_GOAL-theta-004} holds. 
\end{proof}


\section{Bound for the stopping time in the absence of perturbation}\label{sec-results without perturb}

In this section, we assume the framework of Section \ref{sec invar func} but focus on 
the random walk $(S^x_n)_{n\geq 1}$ defined by \eqref{random walk S^x_n-001}, which is 
not subject to perturbations.

For any $x\in \bb X$ and $t \in \bb R$, define the stopping time $\vartheta_{x,t} = \tau_{x, t}^0$ on $\Omega$ by
\begin{align} \label{stopping time theta 001}
\vartheta_{x,t} = \inf \{ k\geq 1:  t + S_k^x  < 0 \}.
\end{align}
Our objective is to establish the following bound, which will be employed in the proof of Theorem \ref{Pro-Appendix-Final2-Inequ},  
and more specifically, in the proof of Lemma \ref{Lem-tau-prior}. 

\begin{proposition}\label{Prop-vartheta}
Suppose that the cocycle $\sigma$ admits finite exponential moments \eqref{exp mom for f 001}
and is centered \eqref{centering-001}. 
We also suppose that the effective central limit theorem \eqref{BEmart-001} is satisfied.
Then there exist constants $c, \beta>0$ such that for any $n \geq 1$, $x \in \bb X$ and $t \in \bb R$,
\begin{align} \label{lem-moment of tau-001}
\bb P (\vartheta_{x,t}  = n) \leq \bb P (\vartheta_{x,t}  \geq n) \leq c \frac{1+\max\{t, 0\}}{n^{\beta}}.
\end{align}
\end{proposition}

It turns out that this proof is not straightforward. 
While a proof for the case of random walks on the general linear group can be found in \cite{GLP17}, 
the more general setting considered here is not covered by that result.  
We decided to include this proof here, 
since it requires techniques 
which will later be instructive for studying the exit time for random walks with perturbations depending on the future.

\subsection{Harmonic function for walks without perturbations}

The proof of Proposition \ref{Prop-vartheta} will rely on the properties of the following function: 
for $n \geq 1$, $x \in \bb X$ and $t \in \bb R$, 
\begin{align}\label{def-V-n-x-t}
V_n(x, t) = \bb E \left( t + S_n^x; \vartheta_{x,t} > n \right). 
\end{align}
Note that $(x,t)\mapsto V_n(x, t)$ is non-negative, and non-decreasing with respect to $t$. 
Clearly, the function $V_n$ coincides with the integrand appearing in the definition of
$U_n^{\mathfrak f}$ (see \eqref{def-U-f-n-t-001}), 
when the perturbation sequence $\mathfrak f = (f_n)_{n \geq 0}$ is identically zero.

In the sequel, it will be useful to note that, if the cocycle $\sigma$ admits finite exponential moments \eqref{exp mom for f 001}
and is centered \eqref{centering-001}, then the sequence $(S^x_n)_{n \geq 0}$ (see \eqref{random walk S^x_n-001}) 
with $S^x_0=0$ is a zero-mean martingale with respect to the natural filtration $(\mathscr A_n)_{n\geq 0}$.

\begin{lemma}\label{Lem-Vn-increasing}
Suppose that the cocycle $\sigma$ admits finite exponential moments \eqref{exp mom for f 001}
and is centered \eqref{centering-001}. 
Then the sequence $(V_n)_{n \geq 1}$ is non-decreasing. 
More precisely, for any $1 \leq n \leq m$, $x \in \bb X$ and $t \in \bb R$, we have 
\begin{align*}
\max\{t, 0\} \leq V_n(x, t) \leq V_m(x, t). 
\end{align*}
\end{lemma}

\begin{proof}
Since the sequence $(S^x_n)_{n \geq 0}$ is a zero-mean martingale,
we can apply the optional stopping theorem to obtain the following expression for $V_n(x, t)$: 
\begin{align}\label{optional-stp-iden}
V_n(x, t) = \bb E \left( t + S_n^x; \vartheta_{x,t} > n \right) 
&= t - \bb E \left( t + S_n^x; \vartheta_{x,t} \leq n \right)  \notag \\
&= t - \bb E \left( t + S_{ \vartheta_{x,t} }^x; \vartheta_{x,t} \leq n \right). 
\end{align}
The lower bound follows as the random variable $t + S_{ \vartheta_{x,t} }^x$ is negative. 
The upper bound also follows from \eqref{optional-stp-iden} since $- \bb E ( t + S_{ \vartheta_{x,t} }^x; \vartheta_{x,t} \leq n )$
is increasing in $n$. 
\end{proof}

The key property of the sequence $(V_n)_{n \geq 1}$ is that   
a kind of converse to Lemma \ref{Lem-Vn-increasing} also holds true. 
More precisely, we will  show that the sequence $(V_n)_{n \geq 1}$ is also quasi-decreasing in the sense stated
in the following proposition: 

\begin{proposition} \label{Prop-Vn-m-001}
Suppose that the cocycle $\sigma$ admits finite exponential moments \eqref{exp mom for f 001}
and is centered \eqref{centering-001}. 
Suppose also that the effective central limit theorem \eqref{BEmart-001} is satisfied.
Then, there exist constants $b, A >0$ such that for any $1 \leq n \leq m$, $x\in \bb X$ 
and $t \in \bb R,$ 
\begin{align*}
V_{m} (x, t)  \leq  V_{n} (x, t)  + A n^{- b} (1+ \max\{t, 0\}).
\end{align*}
\end{proposition}

From Proposition \ref{Prop-Vn-m-001}, we derive the following two corollaries concerning the asymptotic behavior of the function $V_n$. 

\begin{corollary} \label{Corollary-Vn-m-001}
Suppose that the cocycle $\sigma$ admits finite exponential moments \eqref{exp mom for f 001}
and is centered \eqref{centering-001}. 
We also suppose that the effective central limit theorem \eqref{BEmart-001} is satisfied.
Then, there exists a constant $c > 0$ such that for any $n \geq 1$, $x\in \bb X$ and $t \in \bb R,$ 
\begin{align*}
V_{n} (x, t)  \leq  c(1+ \max\{t, 0\}).
\end{align*}
\end{corollary}

\begin{proof}
By Proposition \ref{Prop-Vn-m-001}, for any $n \geq 1$, $x\in \bb X$ and $t \in \bb R,$ 
we get
\begin{align*}
V_{n} (x, t) 
 \leq  V_{1} (x, t)  + A (1+ \max\{t, 0\}). 
\end{align*}
Recall that $V_1(x, t) = \bb E \left( t + \sigma(g_1, x); t + \sigma(g_1, x) \geq 0 \right).$
If $t \leq 0$, we have $V_{1} (x, t) \leq \bb E |\sigma(g_1, x)| \leq c$.  
If $t> 0$, we have $V_{1} (x, t) \leq t + \bb E |\sigma(g_1, x)| \leq t + c$.  
The conclusion follows. 
\end{proof}

\begin{corollary} \label{Corollary-Vn-m-002}
Suppose that the cocycle $\sigma$ admits finite exponential moments \eqref{exp mom for f 001}
and is centered \eqref{centering-001}. 
We also suppose that the effective central limit theorem \eqref{BEmart-001} is satisfied.
Then, there exists $c > 0$ such that for any $n \geq 1$, $x\in \bb X$ and $t \leq 0,$ 
\begin{align*}
V_{n} (x, t)  \leq  c e^{-\alpha_0 |t|}, 
\end{align*}
where $\alpha_0 >0$ is the exponent from the moment assumption \eqref{exp mom for f 001}. 
\end{corollary}

\begin{proof}
Using the Markov property and the definition of the function $V_n$ in \eqref{def-V-n-x-t}, we get 
\begin{align*}
V_n(x, t) = \bb E \Big( V_{n-1} (g_1 x, t + \sigma(g_1, x));  t + \sigma(g_1, x) \geq 0 \Big), 
\end{align*}
where for $n = 1$ we have written $V_0(x, t) = \max\{t, 0\}$. 
By Corollary \ref{Corollary-Vn-m-001}, 
there exists a constant $c > 0$ such that for any $n \geq 1$, $x\in \bb X$ and $t \leq 0,$ 
\begin{align*}
V_n(x, t) \leq  c \bb E \Big( t + \sigma(g_1, x);  t + \sigma(g_1, x) \geq 0 \Big)
& = c \int_{0}^{\infty} \bb P \left( t + \sigma(g_1, x) \geq s \right) ds \notag\\
& = c \int_{-t}^{\infty} \bb P \left( \sigma(g_1, x) \geq s \right) ds. 
\end{align*}
The conclusion now follows from the assumption \eqref{exp mom for f 001} and Markov's inequality. 
\end{proof}

Proposition \ref{Prop-Vn-m-001}, in conjunction with Lemma \ref{Lem-Vn-increasing}, allows us to establish Theorem \ref{Pro-Appendix-Final2-Inequ} 
for the case where the perturbation sequence $\mathfrak f = (f_n)_{n \geq 0}$ is identically zero. 

The key point in the proof of Proposition \ref{Prop-Vn-m-001} is the following 
quasi-decreasing behaviour of $V_n$, 
which is inspired by the results in \cite{DW15, GLP17, GLL18Ann}.

\begin{proposition} \label{Prop-Vn-001}
Suppose that the cocycle $\sigma$ admits finite exponential moments \eqref{exp mom for f 001}
and is centered \eqref{centering-001}. 
We also suppose that the effective central limit theorem \eqref{BEmart-001} is satisfied.
Then, there exist constants $\ee, c>0$ such that for any $n \geq 1$, $x\in \bb X$ 
and $t \in \bb R,$ 
\begin{align}\label{bound-Vn-001-0}
V_{n} (x, t) 
 \leq \left( 1+\frac{c }{n^{\ee }}\right)  
   V_{[n^{1-\ee }] } (x, t)  + c e^{- c n^{\ee}} \left( 1 + \max\{t, 0\} \right).
\end{align}
\end{proposition}
Propositions \ref{Prop-Vn-001} and \ref{Prop-Vn-m-001}  will be proven in 
Subsections \ref{sec-proof of proposition upper bound harm fun-001}  and  \ref{subsec: Quasi-monotonicity bounds for sequences}
below, respectively.


\subsection{Properties of the random walk near the level $\sqrt{n}$} \label{sec-control-high-level-00}

Let $\ee >0$, $x\in \bb X$ and $t\in \bb R.$ 
Consider the first time $\nu_{n,x,t}$
when $\left\vert t+ S_{k}^x  \right\vert $ exceeds $2 n^{1/2-\ee }:$ 
\begin{align}\label{nu n}
\nu_{n,x,t} = \min \left\{ k\geq 1: \left|t + S_{k}^x  
\right|  \geq 2 n^{1/2-\ee }\right\} .  
\end{align}

\begin{lemma} \label{Lemma 2-0}
There exists a constant $\beta >0$ 
such that for any $\ee \in (0,\frac{1}{2})$, $n\geq 1$,  $\ell \geq 1$, $x \in \bb X$ and $t\in \bb R$, 
\begin{align*}
\mathbb{P} \left( \nu_{n,x,t} > \ell \right) 
\leq  2 \exp \left(- \frac{\beta \ell }{n^{1-2\ee}} \right).
\end{align*}
\end{lemma}

In the proof of this lemma we need the following assertion. 

\begin{lemma} \label{lemma-sup  S_k -001} 
There exists a constant $\beta > 0$ such that for any $M \geq 1$ and $n\geq 1$, 
\begin{align*} 
\sup_{x\in \bb X} \sup_{t \in \bb R}
\bb P\left( \sup_{1\leq k \leq n}  | t + S^x_k | \leq M   \right) \leq 2 e^{- \beta \frac{n}{M^2}}.
\end{align*}
 \end{lemma}

\begin{proof}
Let $m = [\gamma^{-2} M^2]$ and
$K= \left[ n/ m \right]$, where $\gamma > 0$ will be chosen later.
It is easy to see that, for any $x\in \bb X$ and $t \in \bb R$,
\begin{align}
\mathbb{P}\left( \max_{1\leq k\leq n}\left\vert t + S^x_{k}\right\vert \leq  M \right)   
\leq \mathbb{P}\left( \max_{1\leq k\leq K}\left\vert t+S^x_{km}\right\vert \leq M \right) .  \label{nu000}
\end{align}
Using the cocycle property \eqref{def-cocycle}, it follows that 
\begin{align*}
\mathbb{P}\left( \max_{1\leq k\leq K}\left\vert t +S^x_{km}\right\vert \leq M \right)  
\leq \mathbb{P}\left( \max_{1\leq k\leq K-1}\left\vert t +S^x_{km}\right\vert \leq M \right) 
\sup_{x'\in \bb X} \sup_{t'\in \mathbb{R}} \mathbb{P} \left( \left\vert t'+S^{x'}_{m}\right\vert \leq M \right) ,
\end{align*}
from which iterating, we get
\begin{align} \label{Piterations-001}
\mathbb{P} \left( \max_{1\leq k\leq K}\left\vert t +S^x_{km}\right\vert \leq  M \right) 
 \leq \left( \sup_{x'\in \bb X} \sup_{t'\in \mathbb{R}} \mathbb{P} \left( \left\vert t'+S^{x'}_{m}\right\vert \leq M \right) \right) ^{K}.
\end{align}
By the effective central limit theorem \eqref{BEmart-001}, 
there exist constants $c, \bf{v}, \epsilon > 0$ such that for any $x' \in \bb X$, $t' \in \bb R$ and $m \geq 1$, 
\begin{align*}
\mathbb{P}\left( \left|  t' +S^{x'}_{m} \right| \leq M \right) 
 = \mathbb{P} \left( \frac{S^{x'}_{m}}{\sqrt{m}} \in \left[ \frac{-M - t' }{\sqrt{m}},  \frac{M- t'}{\sqrt{m}} \right] \right)  
 \leq  \int_{  \frac{-M-t'}{\sqrt{m}} }^{ \frac{M-t'}{\sqrt{m} } }  \phi_{\bf{v}^2} (u) du  + \frac{c}{ m^{\epsilon} }
\leq  \frac{1}{2}, 
\end{align*}
where in the last inequality we take $\gamma >0$ sufficiently small so that 
$\frac{2M}{\sqrt{m}} + \frac{c}{ m^{\epsilon} } \leq 4 \gamma + c (2 \gamma^2)^{\epsilon}  \leq \frac{1}{2}$. 
The assertion of the lemma follows from \eqref{Piterations-001}. 
\end{proof}

\begin{proof}[Proof of Lemma \ref{Lemma 2-0}]
By Lemma \ref{lemma-sup  S_k -001}, there exists a constant $\beta>0$ such that for any $n \geq 1$, $\ell \geq 1$, $x \in \bb X$ and $t \in \bb R$,
\begin{align*} 
\mathbb{P} \left( \nu_{n,x,t}> \ell \right) 
= \mathbb{P} \left( \sup_{ 1\leq k \leq \ell } |t + S_k^{x} |  \leq  2 n^{1/2-\ee  }  \right) 
\leq 2 \exp \left(- \frac{\beta \ell}{4 n^{1-2\ee}}\right), 
\end{align*}
completing the proof of the lemma, by replacing $\beta/4$ with $\beta$. 
\end{proof}

\begin{lemma}\label{Lem-nu-less-n-0} 
\label{Lemma-nu-lower}For any $\ee \in (0,\frac{1}{4})$, there exist constants $c, c_{\ee} >0$ 
such that for any $n\geq 1$, $x\in \bb X$ and  $t\geq - n^{\ee}$, 
\begin{align*}
\bb{P} \left( \nu_{n,x,t} \leq n^{1/2-\ee } - t \right) 
\leq c \exp \left( -c_{\ee }n^{\ee/2}   \right). 
\end{align*}
\end{lemma}

\begin{proof}
Set $A_{n,x,t} = \{ \max_{ k \in [1, n^{1/2-\ee } - t] } |\sigma(g_{k},g_{k-1} \cdots g_{1} x)| \leq n^{\ee/2}  \}$. 
By \eqref{exp mom for f 001}, there exist constants $c, c_{\ee}, \alpha > 0$ 
such that for any $n\geq 1$, $x\in \bb X$ and  $t\geq - n^{\ee}$, 
\begin{align*}
\bb P \big( A_{n,x,t}^c \big) \leq c  \max\{ n^{1/2-\ee } - t, 0 \} e^{- \alpha n^{\ee/2}} \leq c e^{- c_{\ee} n^{\ee/2}}. 
\end{align*}
On the event $A_{n,x,t}$, it holds that for any $k \in [1, n^{1/2-\ee } - t]$, 
\begin{align*}
|t + S_k^x| \leq |t| +  k n^{\ee/2} \leq n^{1/2 - \ee/2}. 
\end{align*}
Hence, by the definition of $\nu_{n,x,t}$, it follows that
$\{ \nu_{n,x,t} \leq n^{1/2-\ee } - t \}$ is included in $A_{n,x,t}^c$. 
The lemma follows. 
\end{proof}


\subsection{Upper bound for the harmonic function: Denisov and Wachtel's method} \label{sec-proof of proposition upper bound harm fun-001}
 In this subsection, we prove Proposition \ref{Prop-Vn-001} by
 using the method originally introduced by Denisov and Wachtel \cite{DW15}. 
To do so, we first state the following basic inequality. 
 
\begin{lemma}\label{Lem-trick}
For any $t \in \bb R$ and any random variable $Z$,
\begin{align*} 
\bb E^{1/2}\left( \left(t + Z\right)^2; t+Z\geq 0 \right) \leq \max\{t,0 \} + \bb E^{1/2} (Z^2).
\end{align*}
\end{lemma}

The next lemma provides an upper bound on the value of the process $(t+S_n^x)_{n\geq 1}$ when it crosses the boundary. 
Recall that $\vartheta_{x, t}$ is defined by \eqref{stopping time theta 001}. 
 
 \begin{lemma}\label{Lemma 1-0} 
There exist constants $c>0$ and  $\ee _{0}>0$ such that for any $\ee \in (0,\ee _{0})$, $n\geq 1$,  $x \in \bb X$ and $t \geq  n^{1/2-\ee }$,  
\begin{align*}
\bb{E} \left( \left\vert t + S_{ \vartheta_{x, t} }^x \right\vert; \vartheta_{x, t} \leq n \right) 
\leq c \frac{t}{n^{\ee }}. 
\end{align*}
\end{lemma}

\begin{proof}
Consider the event 
\begin{align*}
A_{n,x} = \left\{ \max_{1 \leq k \leq n}\left\vert \sigma(g_k, g_{k-1} \cdots g_1 x) \right\vert \leq n^{1/2-2\ee } \right\}. 
\end{align*}
We have
\begin{align}
\bb{E} \left( \left\vert t + S_{ \vartheta_{x, t} }^x \right\vert; \vartheta_{x, t} \leq n \right)
= \bb{E} \left( \left\vert t + S_{ \vartheta_{x, t} }^x \right\vert; \vartheta_{x, t} \leq n,  A_{n,x} \right)  
  +  \bb{E} \left( \left\vert t + S_{ \vartheta_{x, t} }^x \right\vert; \vartheta_{x, t} \leq n,   A_{n,x}^c  \right).  \label{eq-lemma1-000-0}
\end{align}
For the first term, since $\vartheta_{x, t}$ is the first integer $k \geq 1$
when $t + \sum_{i=1}^{k} \sigma(g_k, g_{k-1} \cdots g_1 x)$ 
becomes negative and the size of the jump is bounded by $|\sigma(g_{\vartheta_{x, t}}, g_{ \vartheta_{x, t} -1} \cdots g_1 x)|$  
which does not exceed $n^{1/2-2\ee }$ on
the event $A_{n,x}$, we get that for $t \geq  n^{1/2-\ee }$,  
\begin{align}\label{eq-lemma1-R1-0}
\bb{E} \left( \left\vert t + S_{ \vartheta_{x, t} }^x \right\vert; \vartheta_{x, t} \leq n, A_{n,x} \right) 
\leq  n^{1/2-2\ee }\mathbb{P} \left( A_{n,x}\right)  \leq  n^{1/2-2\ee } \leq {\frac{t}{n^{\ee}}}. 
\end{align}
To handle the second term in \eqref{eq-lemma1-000-0}, by Markov's inequality and \eqref{exp mom for f 001}, we get
\begin{align*}
\mathbb{P} \left( A_{n,x}^c \right) 
&=\bb P \left( \max_{1 \leq k\leq n} |\sigma(g_k, g_{k-1} \cdots g_1 x)| >  n^{1/2 - 2\ee }   \right)  \notag \\
&\leq n \sup_{x\in \bb X}  \mathbb{P} \left( |\sigma(g_{1}, x)| >  n^{1/2-2\ee }   \right) 
\leq   c n e^{- \alpha n^{1/2-2\ee } } \leq c e^{- \alpha n^{1/3}}.   
\end{align*}
Therefore, by H\"older's inequality, 
\begin{align}\label{eq-lemma1-R2-0}
 \bb{E} \left( \left\vert t + S_{ \vartheta_{x, t} }^x \right\vert; \vartheta_{x, t} \leq n,   A_{n,x}^c \right)
& =  \sum_{k =1}^n \bb{E} \left( \left\vert t + S_{k}^x \right\vert;   \vartheta_{x, t} = k,  A_{n,x}^c \right)  \notag\\
& \leq  \sum_{k =1}^n \bb{E} \left( \left\vert t + S_{k}^x \right\vert;  A_{n,x}^c \right)  \notag\\
& \leq     \sum_{k = 1}^n  \bb{E}^{1/2} \left( \left\vert t + S_{k}^x \right\vert^2 \right)  
             \bb P^{1/2}(A_{n,x}^c)  \notag\\
& \leq  c n (t + n^{1/2})  e^{- \frac{\alpha}{2} n^{1/3 } },
\end{align}
where in the last inequality we used 
\begin{align*}
\bb{E}^{1/2} \left( \left\vert t + S_{k}^x \right\vert^2 \right)
\leq t+ \sup_{x\in \bb X} \sqrt{\mathbb{E}  (S_{k}^{x} )^{ 2 }    }
\leq t+c n^{1/2}, 
\end{align*}
by Minkowski's inequality. 
Combining \eqref{eq-lemma1-R1-0} and \eqref{eq-lemma1-R2-0} completes the proof. 
\end{proof}

\begin{proof}[Proof of Proposition \ref{Prop-Vn-001}]
 The main idea of the proof is to stop the process $(t+S_n^x)_{n\geq 1}$ at the exit time $\nu_{n,x,t}$ and to apply the Markov property. 

We shall first show that there exist constants $c, \ee _{0}>0$ such that for any 
$\ee \in (0, \ee_0)$, $n\geq 1$,  $x \in \bb X$ and $t\geq n^{1/2-\ee }$, 
\begin{align}\label{Un-xy-Bound-001-0}
V_n(x, t) \leq \left( 1+ \frac{c}{n^{\ee}} \right) t. 
\end{align}
Let $x \in \bb X$ and $t\geq n^{1/2-\ee }$. 
Using \eqref{optional-stp-iden}, we get
\begin{align*}
V_n(x, t)  
= t - \bb{E} \left( t + S_{ \vartheta_{x, t} }^x; \  \vartheta_{x, t} \leq n\right) 
 \leq t + \bb{E} \left( \left\vert t + S_{ \vartheta_{x, t} }^x \right\vert; \vartheta_{x, t} \leq n \right). 
\end{align*}
This inequality, combined with Lemma \ref{Lemma 1-0}, yields the desired bound \eqref{Un-xy-Bound-001-0}.

Now we prove \eqref{bound-Vn-001-0} by using the bound \eqref{Un-xy-Bound-001-0} and the Markov property. 
Note that for any $\ee >0$, $n \geq 1$, $x \in \bb X$ and $t \in \bb R$, 
\begin{align}
V_{n} (x, t) = \bb{E} \left( t + S_{n}^x; \vartheta_{x, t} > n, \nu_{n,x, t} > n^{1-\ee } \right) 
 + \bb{E} \left( t + S_{n}^x; \vartheta_{x, t} > n, \nu_{n,x, t} \leq n^{1-\ee } \right).  \label{bound J1+J2-0}
\end{align}
By the Cauchy-Schwarz inequality and Lemmas \ref{Lemma 2-0} and \ref{Lem-trick}, we get
\begin{align}\label{first-term-V-n-001}
\bb{E} \left( t + S_{n}^x; \vartheta_{x, t} > n, \nu_{n,x, t} > n^{1-\ee } \right)
\leq c' e^{- c n^{\ee}} \left( 1 + \max\{t, 0\} \right). 
\end{align}
For the second term in \eqref{bound J1+J2-0}, we decompose it as 
\begin{align} \label{bound J1 002-0}
\bb{E} \left( t + S_{n}^x; \vartheta_{x, t} > n, \nu_{n,x, t} \leq n^{1-\ee } \right)  
= \sum_{k=1}^{ [n^{1-\ee }] }
\bb{E} \left( t + S_{n}^x; \vartheta_{x, t} > n, \nu_{n,x, t} = k \right).
\end{align}
By the Markov property,  we get 
\begin{align} \label{bound E-0}
\bb{E} \left( t + S_{n}^x; \vartheta_{x, t} > n, \nu_{n,x, t} = k \right) 
= \bb E \left[ V_{n-k} \left( g_k \cdots g_1 x, t + S_k^x \right); 
      \vartheta_{x, t}  >k, \nu_{n,x, t} = k  \right].
\end{align}
On the event $\{ \vartheta_{x, t}  >k, \nu_{n,x, t} = k \}$, we have $t + S_k^x \geq n^{1/2 - \ee}$.
Thus, by \eqref{Un-xy-Bound-001-0}, we get that for any $k \in [1, [n^{1-\ee }]]$, 
\begin{align*}
 \bb{E} \left( t + S_{n}^x; \vartheta_{x, t} > n, \nu_{n,x, t} = k \right)
\leq  \left( 1+ \frac{c}{n^{\ee}} \right) \bb E \left(  t + S_k^x; 
      \vartheta_{x, t}  >k, \nu_{n,x, t} = k  \right). 
\end{align*}
Inserting this into \eqref{bound J1 002-0}, we obtain 
\begin{align}\label{equ-n-k-S}
\bb{E} \left( t + S_{n}^x; \vartheta_{x, t} > n, \nu_{n,x, t} \leq n^{1-\ee } \right)
\leq  \left( 1+ \frac{c}{n^{\ee}} \right) 
\sum_{k=1}^{ [n^{1-\ee }] }
 \bb E \left(  t + S_k^x;  \vartheta_{x, t}  >k, \nu_{n,x, t} = k  \right). 
\end{align}

We claim that the sequence $\big( (t + S_k^x) \mathds 1_{\{ \vartheta_{x, t}  \leq k \}} \big)_{k \geq 1}$
is a supermartingale with respect to the natural filtration $(\mathscr A_k)_{k \geq 1}$.
Indeed, for any $k \geq 1$ and $A \in \mathscr A_k$, we have 
\begin{align*}
 \bb E \left(  t + S_{k+1}^x;  \vartheta_{x, t}  \leq  k+1,  A  \right) 
& =  \bb E \left(  t + S_{k+1}^x;  \vartheta_{x, t}  \leq  k,  A  \right)
 +   \bb E \left(  t + S_{k+1}^x;  \vartheta_{x, t}  = k+1,  A  \right)  \notag\\
& \leq  \bb E \left(  t + S_{k}^x;  \vartheta_{x, t}  \leq  k,  A  \right), 
\end{align*}
where we have used \eqref{centering-001} and the fact that, by the definition of $\vartheta_{x, t}$, one has $t + S_{k+1}^x < 0$ on the set $\{ \vartheta_{x, t}  = k+1 \}$. 
Still by \eqref{centering-001}, this implies that the sequence $\big( (t + S_k^x) \mathds 1_{\{ \vartheta_{x, t}  > k \}} \big)_{k \geq 1}$
is a submartingale. 
From this submartingale property, we have that, for any $k\in [1, [n^{1-\ee }] ]$, 
\begin{align}\label{Equ-desired-0}
\bb E \left(  t + S_k^x;  \vartheta_{x, t}  >k, \nu_{n,x, t} = k  \right)  
\leq   \bb E \left(  t + S_{[n^{1-\ee}]}^x;  \vartheta_{x, t}  > [n^{1-\ee}], \nu_{n,x, t} = k  \right). 
 \end{align}
Therefore, using \eqref{equ-n-k-S} and \eqref{Equ-desired-0}, we get
\begin{align*}
\bb{E} \left( t + S_{n}^x; \vartheta_{x, t} > n, \nu_{n,x, t} \leq n^{1-\ee } \right)
& \leq  \left( 1+ \frac{c}{n^{\ee}} \right)  
 \bb E \left(  t + S_{[n^{1-\ee}]}^x;  \vartheta_{x, t}  > [n^{1-\ee }], \nu_{n,x, t} \leq n^{1-\ee }   \right) \notag\\
 & \leq \left( 1+ \frac{c}{n^{\ee}} \right) V_{[n^{1-\ee}]}(x, t). 
\end{align*}
Combining this  with \eqref{bound J1+J2-0} and  \eqref{first-term-V-n-001} 
completes the proof of Proposition \ref{Prop-Vn-001}.
\end{proof}

\subsection{Quasi-monotonicity for sequences of functions} \label{subsec: Quasi-monotonicity bounds for sequences}

We now deduce Proposition \ref{Prop-Vn-m-001} from Proposition \ref{Prop-Vn-001}. 
To this end, we utilize the following more general result 
presented here in a form that will also be useful later. 
The distinct feature of this result is the presence of an additional shift in the argument of the function $V_n$.
This will play a key role in Sections \ref{sec quasi-increasingness} and  \ref{sec-quasi-decreasing property}, 
where we consider the case of random walks with perturbations depending on the future.

In this section, however, we will only need the result of the particular case where the shift is absent. 
Despite this, the techniques developed here will be useful for the more complex case involving perturbations. 
By analyzing this shifted setting, we obtain quasi-monotonicity properties for sequences of functions like 
$V_n$, which will help us understand better the subsequent sections dealing with future-dependent perturbations.

\begin{lemma} \label{lem-iter-proc-001}
Let $(V_n)_{n\geq 1}$ be a sequence of non-negative and non-decreasing functions on $\bb R$.
 Assume that the sequence $(V_n)_{n\geq 1}$ is quasi-increasing, in the  sense that there exist real numbers
$\ee, \beta >0$ and $a, b \geq 0$ such that, for any $n \leq m$ and $t \in \bb R$, 
\begin{align}\label{V-n-m-second}
V_{n} (t)  \leq  V_{m}(t + a n^{-\ee}) 
 + b n^{-\beta} \left( 1 + \max \{t,0\}   \right). 
\end{align}
If, in addition, the sequence $(V_n)_{n\geq 1}$ satisfies the property that for any $t\in \bb R$: 
\begin{align} \label{eq-bound for V_1-001}
V_1(t) \leq b (1 + \max\{t, 0\}) 
\end{align}
and, for any $n\geq 1$, 
\begin{align} \label{V-n-m-first}
V_{n} (t)  \leq \left( 1 + b n^{-\ee} \right)  
   V_{[n^{1-\ee }] } (t + a n^{-\ee})  + b n^{-\beta} (1+\max \{t,0\}), 
\end{align}
then, there exist constants $A, B \geq 0$ depending only on $\ee, \beta, a, b$ such that, for any $1 \leq n \leq m$ and $t \in \bb R$, 
\begin{align*} 
V_{m} (t) \leq  
V_{n} \left( t + A n^{-\ee (1-\ee)} \right) + B n^{- \min \{ \beta (1-\ee), \ee\} } (1+\max \{t,0\}). 
\end{align*}
Moreover, if $a = 0$, we can take $A=0$. 
\end{lemma}

\begin{proof}
Define the sequence $(m_j)_{j\geq 0}$ as follows: with $m_0=m$, and for $j \geq 0$,
\begin{align}\label{def-mj-sequence}
m_{j+1} = [m_j^{1-\ee}]. 
\end{align}
This sequence is non-increasing and converges to $1$. 
Fix $2 \leq n < m$. 
There exists a least integer $\ell \in \bb N$ such that $m_{\ell} \leq n$ and we have $\ell \leq  \alpha \log \frac{\log m}{\log n}$ for some constant $\alpha >0$. 
By our assumption \eqref{V-n-m-first}, it holds that for any $j \in \bb N$ and $t \in \bb R$,
\begin{align*}
V_{m_j} (t) 
 \leq \left( 1 + b m_j^{-\ee} \right)  
   V_{m_{j+1} } \left( t + a m_j^{-\ee} \right)  +  b m_j^{-\beta} (1 + \max \{t, 0 \}).
\end{align*}
By iterating over $j$, we get
\begin{align*} 
V_{m_j} (t) 
& \leq \left( 1 + b m_{j+1}^{-\ee} \right) \left( 1 + b m_j^{- \ee} \right)   
   V_{m_{j+2} } \left( t  + a m_{j}^{-\ee}  + a m_{j+1}^{-\ee} \right)  \notag\\ 
& \quad +  b m_{j+1}^{- \beta}  \left( 1 + b m_j^{-\ee} \right)  \left(1+ \max \{t,0\} + a m_{j}^{-\ee} \right)   
 + b m_j^{-\beta} (1+ \max \{t,0\}  ).
   \end{align*}
After $\ell - 1$ iterations, we obtain
\begin{align} \label{V-A-B-j}
V_{m} (t) 
 \leq A_{\ell} V_{m_{\ell} } \left(t + a B_{\ell}  \right) 
   +  b \sum_{j=0}^{\ell -1}  m_j^{- \beta}  A_{j}  \left(1+\max \{t,0\} + a B_j \right), 
\end{align}
where, for $0 \leq j \leq \ell$, 
\begin{align*}
A_j =  \prod_{k=0}^{j-1} \left( 1 + b m_{k}^{-\ee} \right), 
\qquad 
B_j  =  \sum_{k=0}^{j-1} m_k^{-\ee} 
\end{align*}
with the convention $\prod_{k=0}^{-1} = 1$ and $\sum_{k=0}^{-1} = 0$. 
To give an upper bound for $A_{\ell}$, 
we first use the inequality $1 + x \leq e^x$ for $x \geq 0$ to get
\begin{align} \label{majAm}
A_{\ell} \leq \prod_{k=0}^{\ell -1} e^{b m_k^{-\ee} } = e^{b B_{\ell}}. 
\end{align}
Since $\ell \in \bb N$ is the least integer such that $m_{\ell} \leq n$, we have $m_{\ell - 1} \geq n$, so that
\begin{align} \label{bound-B-ell-sum}
B_{\ell}  =   \sum_{k=0}^{\ell-1} m_k^{-\ee} 
& = \frac{1}{m^{\ee}} + \frac{1}{m_{1}^{\ee}} + \ldots + \frac{1}{m_{\ell-1}^{\ee}}   \notag\\ 
& \leq  \frac{c}{n^{\ee}}  m^{\ee (1-\ee)^{\ell -1} }  \left( \frac{1}{m^{\ee}} + \frac{1}{m^{\ee (1-\ee) }}  
   + \ldots  + \frac{1}{m^{\ee (1-\ee)^{\ell -1} }} \right)  \notag\\ 
& \leq  \frac{c}{n^{\ee}} \left( m^{\ee ((1-\ee)^{\ell -1} -1 )} + m^{\ee (1-\ee) ((1-\ee)^{\ell -2} -1 )}  
   + \ldots  + 1 \right) \notag\\
& =  \frac{c}{n^{\ee}} \sum_{k=0}^{\ell -1} m^{\ee (1-\ee)^{k} ((1-\ee)^{\ell -1- k} -1 ) } \notag\\
& =  \frac{c}{n^{\ee}} \sum_{k=0}^{\ell -1} m^{\ee (1-\ee)^{\ell - 1-k} ((1-\ee)^{k} -1 ) }. 
\end{align}
As $\min_{ 0 \leq k \leq \ell -2 } m_k \geq n$, it holds that $\min_{ 0 \leq k \leq \ell -2 } m^{(1-\ee)^{\ell - 2 -k}} \geq n$
and for any $0 \leq k \leq \ell -2$, 
\begin{align*}
\frac{m^{\ee (1-\ee)^{\ell - 2 -k} ((1-\ee)^{k+1} -1 ) } }{ m^{\ee (1-\ee)^{\ell - 1-k} ((1-\ee)^{k} -1 ) } }
= m^{ - \ee^2 (1 - \ee)^{\ell -2 -k} }
\leq  n^{- \ee^2 }. 
\end{align*}
Hence, we obtain 
\begin{align*}
\sum_{k=0}^{\ell -1} m^{\ee (1-\ee)^{\ell - 1-k} ((1-\ee)^{k} -1 ) }
\leq  \sum_{k=0}^{\ell -1}   n^{- \ee^2 k}
\leq \frac{1}{1 - n^{-\ee^2}} 
\leq \frac{1}{1 - 2^{-\ee^2}}. 
\end{align*}
Substituting this into \eqref{bound-B-ell-sum} and \eqref{majAm}, we get
\begin{align} \label{bound-A-B-ell-24}
B_{\ell}  \leq  \frac{c}{n^{\ee}},  \quad
A_{\ell} \leq  e^{b B_{\ell}}  \leq 1 + \frac{c'}{n^{\ee}}, 
\end{align}
where $c' > 0$ depends on $b$.  
Reasoning as above shows that $\sum_{j=0}^{\ell -1}  m_j^{-\beta} \leq c n^{-\beta}$. 
Therefore, from \eqref{V-A-B-j} and \eqref{bound-A-B-ell-24}, we obtain that there exists a constant $c > 0$ depending on $a$ and $b$ such that 
\begin{align}\label{Bound-Vmt-aaa}
V_{m} (t) 
& \leq  \left( 1 + c n^{-\ee}  \right) V_{m_{\ell} } \left( t + c n^{-\ee}  \right) 
   +  c \left(1+\max \{t,0\}  \right) \sum_{j=0}^{\ell -1}   m_j^{-\beta}  \notag\\
& \leq  \left( 1 + c n^{-\ee}  \right) V_{m_{\ell} } \left( t + c n^{-\ee}  \right) 
   + c n^{- \beta} \left(1+\max \{t,0\} \right). 
\end{align}
As $m_{\ell} \leq n$, by \eqref{V-n-m-second}, we also have 
\begin{align}\label{Bound-Vmt-bbb}
V_{m_{\ell} } \left( t + c n^{-\ee}  \right) 
& \leq V_{n} \left( t + c n^{-\ee} + a m_{\ell}^{-\ee} \right)  
 + c m_{\ell}^{-\beta} \left( 1 + \max \{t, 0\}   \right)  \notag\\
& \leq  V_{n} \left( t +  c n^{-\ee(1-\ee)} \right) 
 + c n^{-\beta(1-\ee)} \left( 1 + \max \{t,0\}   \right), 
\end{align}
where in the last inequality we used the fact that $n \leq m_{\ell -1}$, so that 
$n^{1-\ee} \leq m_{\ell -1}^{1 - \ee} \leq m_{\ell} + 1$ by \eqref{def-mj-sequence}. 
Substituting \eqref{Bound-Vmt-bbb} into \eqref{Bound-Vmt-aaa} gives that for any $1 \leq n \leq m$ and $t \in \bb R$, 
\begin{align}\label{inequ-recursive-001}
V_{m} (t) \leq  \left( 1 + c n^{-\ee} \right)  
V_{n} \left( t + c n^{-\ee (1-\ee)} \right) + c n^{- \beta (1-\ee) } (1+\max \{t,0\}). 
\end{align}
In particular, using the fact that $V_1(t) \leq b (1 + \max\{t, 0\})$,  we get 
\begin{align*}
V_{n} (t) \leq \left( 1 + c \right)   
V_{1} \left( t + c \right) + c (1+\max \{t,0\})
\leq c' (1+\max \{t,0\}). 
\end{align*}
Finally, inserting this into \eqref{inequ-recursive-001} yields the conclusion of the lemma. 
The proof of the fact that $A=0$ when $a=0$ is left to the reader. 
\end{proof}

\begin{proof}[Proof of Proposition \ref{Prop-Vn-m-001}]
It follows from Lemma \ref{Lem-Vn-increasing} and Proposition \ref{Prop-Vn-001}
that the assumptions of Lemma \ref{lem-iter-proc-001} are satisfied. 
Hence the conclusion follows. 
\end{proof}


\subsection{Estimate for the stopping time $\vartheta_{x,t}$}
In this subsection, we give a proof of Proposition \ref{Prop-vartheta}. 

\begin{proof}[Proof of Proposition \ref{Prop-vartheta}]
Since $\bb P \left( \vartheta_{x,t} > n \right) $ is non-decreasing in $t \in \bb R$, 
it suffices to prove the result for $t \geq 0$. 
By Corollary \ref{Corollary-Vn-m-001}, we can find a constant $c>0$ such that for any $n \geq 1$, $x \in \bb X$ and $t \in \bb R$, 
\begin{align*}
V_n(x,t) = \bb E \left( t + S_n^x; \vartheta_{x,t} > n \right) \leq c (1 + \max\{t, 0\}). 
\end{align*}
We fix $\beta >0$, whose value will be determined later.
Note that we can assume $t \leq n^{\beta}$. 
We write 
\begin{align*}
\bb P \left( \vartheta_{x,t} > n \right) 
 = \bb P \left( \vartheta_{x,t} > n, |S_n^x| \leq n^{2\beta} \right) + \bb P \left( \vartheta_{x,t} > n, |S_n^x| > n^{2\beta} \right).  
\end{align*}
Recall that $0 \leq t \leq n^{\beta}$. Then, on the set $\{ \vartheta_{x,t} > n, |S_n^x| > n^{2\beta} \}$
we have $t + S_n^x \geq n^{2\beta}$. From the above, we get that, by Chebyshev's inequality,  
\begin{align*}
\bb P \left( \vartheta_{x,t} > n \right) 
\leq \bb P \left( |S_n^x| \leq n^{2\beta} \right) + n^{-2 \beta} V_n(x, t)
\leq \bb P \left( |S_n^x| \leq n^{2\beta} \right) + c n^{- 2\beta} (1 + \max\{t, 0\}). 
\end{align*}
Now we apply the effective central limit theorem \eqref{BEmart-001} to obtain 
\begin{align*}
\bb P \left( |S_n^x| \leq n^{2\beta} \right) \leq \int_{-n^{2\beta - \frac{1}{2}} }^{ n^{2\beta - \frac{1}{2}} } \phi_{\bf{v}^{2}}(u) du 
+ \frac{c}{n^{\epsilon}}
\leq c \left( n^{2\beta - \frac{1}{2}} + n^{-\epsilon} \right),
\end{align*}
which leads to the desired result as soon as $\beta < \frac{1}{4}$. 
\end{proof}


\section{Quasi-increasing behaviour of the sequence $(U_n^{\mathfrak f})_{n\geq 1}$}\label{sec quasi-increasingness}

The goal of this section is to establish the bound \eqref{bound with m for U-105-01-2} of Theorem \ref{Pro-Appendix-Final2-Inequ},
which can be seen as a generalization of Lemma \ref{Lem-Vn-increasing}. 
The key strategy is based on  finite-size 
approximation of perturbation sequences, which consists in
replacing the sequence of functions $\mathfrak f = (f_n)_{n \geq 0}$  by a sequence of functions depending 
on a finite but increasing number of coordinates in $\Omega$. 
The finite-size approximation is effective because the analysis for such sequences can be tackled using similar techniques to those used in the proof of Lemma \ref{Lem-Vn-increasing}, but within the framework of a suitably defined Markov chain. This approach simplifies the problem by allowing us to focus on a tractable approximation, gradually incorporating more complexity as we increase the number of coordinates considered in the sequence.
The next subsection will focus on the construction of this Markov chain. 
This chain will play a pivotal role not only in establishing the bound \eqref{bound with m for U-105-01-2}
 but also in proving the converse inequality \eqref{bound with m for U-105-01-3} later in Section \ref{sec-quasi-decreasing property}.


\subsection{Construction of the Markov chain}\label{Subsec-Markov-chian}
In this subsection, we consider that the sequence of perturbations depends on finitely many coordinates and 
 we will provide a construction of an appropriate Markov chain which will be used to prove both bounds 
\eqref{bound with m for U-105-01-2} and \eqref{bound with m for U-105-01-3} of Theorem \ref{Pro-Appendix-Final2-Inequ}. 
Besides, this construction will also allow us to derive an associated martingale that plays an important role in the analysis. 

In this subsection we fix $p \geq 1$. 
Assume that $\mathfrak f = \mathfrak f_p=(f_{n})_{n\geq0}$ is a sequence of measurable functions, 
where each $f_{n}: \bb G^{\{1,\ldots,p \}}\times \bb X \to \bb R$ depends only on the first $p$ coordinates. 
Specifically, for any $n\geq 0$ and  $x\in \bb X$,  
the function $f_{n}(\cdot, x)$ 
depends only on the first $p$ coordinates $(g'_{1},\ldots, g'_{p})$, meaning that $f_{n}(\cdot, x)$ is 
$\mathscr A_p$ measurable, where 
 $\mathscr A_p$ is the $\sigma$-algebra on $\Omega$ generated by $g_{1}, \ldots, g_{p}$. 
By abuse of notation,  we shall sometimes use the same notation $f_{n}$ to denote 
the function on $\Omega \times \bb X$
defined by $(\omega,x) \mapsto f_{n}(g_1(\omega),\ldots, g_p(\omega),x)$.
The same remark concerns the function on $\bb G^{\{0,\ldots,p \}} \times \bb X$
defined by $(g'_0,\ldots, g'_p,x) \mapsto f_{n}(g'_1,\ldots, g'_p,x)$,
which does not depend on the coordinate $g'_0$.
This ghost coordinate will be useful below, see \eqref{trans-prob-xi}. 

Recall the definition of $U^{\mathfrak f}_n(t)$ from \eqref{bound with m for U-105-01-2} and \eqref{bound with m for U-105-01-3} 
of Theorem \ref{Pro-Appendix-Final2-Inequ} (see \eqref{def-U-f-n-t-001}): for any 
$t\in \bb R$ and $n\geq 1$, 
\begin{align*} 
U^{\mathfrak f}_n(t)
 = \int_{\bb X} \bb E \left( t + S^x_n + f_{n}\circ T^n(\omega,x) - f_{0}(\omega,x);  \tau^{\mathfrak f}_{x, t} > n \right) \nu(dx), 
\end{align*}
where $S^x_n =\sigma (g_n \cdots g_1,x)$ is defined by \eqref{random walk S^x_n-001}. 
We start by decomposing $U^{\mathfrak f}_n(t)$ into two parts for a clearer analysis: 
\begin{align} \label{MAIN_GOAL-002}
U^{\mathfrak f}_n(t) = I_1 + I_2, 
\end{align}
where
\begin{align} 
I_1 & = \int_{\bb X} \bb E \left( t + S_{n+p}^{x}; \tau^{\mathfrak f}_{x, t} > n \right) \nu(dx) \notag \\  
I_2 &= \int_{\bb X} 
 \bb E \left(  S_{n}^{x} - S_{n+p}^{x} + f_n\circ T^n(\omega,x) - f_0(\omega,x); \tau^{\mathfrak f}_{x, t} > n \right) \nu(dx). \label{MAIN_GOAL-002bbb}
\end{align}
The term $I_1$ will give the main contribution while the term $I_2$ will be negligible.

To handle the main term $I_1$, we introduce below a Markov chain which will play an important role in the subsequent proofs. 
This chain will allow us to effectively manage the sequence of perturbations depending on finitely many coordinates 
and facilitate the analysis of the random walk. 
For any $p \geq 1$, set $\bb A_{p} =\bb G^{\{0,\ldots,p \}} \times \bb X \times \bb N$. 
Consider the family of transition probabilities $\{ P_a: a \in \bb A_p \}$ defined in the following way: for any $a=(g_0,\ldots,g_{p},x,q)\in \bb A_p$
and any nonnegative measurable function $\varphi: \bb A_p \to \bb R$, 
\begin{align} \label{trans-prob-xi}
P_a (\varphi) = \int_{\bb G} \varphi \left( g_{1},\ldots, g_{p}, g, g_1 x, q+1 \right) \mu(dg).  
\end{align}
For any $a\in \bb A_p$, we write $\bb P_a$ for the probability measure 
on $\Omega'_p = \bb A_p^{\bb N}$ associated to the transition probability $P_a$.
The expectation with respect to $\bb P_a$ is denoted by $\bb E_a$.
 For any 
 $\omega' \in \bb A_p^{\bb N}$ and $i \geq 0$, 
let $\xi_i(\omega')$ be the $i$-th coordinate map of $\omega'$ in $\bb A_p^{\bb N}$.
It is straightforward to verify that, under the measure $\bb P_a$,
the sequence $(\xi_i)_{i \geq 0}$ forms an $\bb A_p$-valued Markov chain with transition probabilities $\{P_a, a \in \bb A_p \}$ 
and with initial value $\xi_0 = a$, $\bb P_a$-almost surely. 
Introduce the $\sigma$-algebras $\mathscr G_n= \sigma \{ \xi_0,\ddd,\xi_{n} \}$, $n\geq 0$.
Additionally, we define the function $\sigma_p:  \bb A_p \to  \bb R$ by setting, for any $a=(g_0,\ldots,g_{p},x,q) \in \bb A_p$, 
\begin{align*}
\sigma_p (g_0, \ldots, g_{p}, x,q) = \sigma(g_0,g_0^{-1}x)=-\sigma(g_0^{-1},x) . 
\end{align*}

With the definitions introduced above, we can express $S^x_n:=\sigma (g_n \cdots g_1,x)$ as an additive functional
of the Markov chain $(\xi_i)_{i \geq 0}$ as follows:  
for any $a=(g_0,\ldots,g_{p},x,q) \in \bb A_p$ and $n\geq 1$,
for $\bb P_a$-almost every $\xi=(\xi_0,\xi_1,\ldots)\in \Omega'_p = \bb A_p^{\bb N}$, we have
\begin{align} \label{new representation for S^x_n 001}
S^x_n := \sigma (g_n \cdots g_1,x) =  \sum_{i=1}^{n} \sigma_p(\xi_i),  
\end{align}
where, for  $j>p$, the element $g_j$ is the unique element of $\bb G$ such that 
\begin{align*}
\xi_{j-p} = (g_{j-p},\ldots, g_j, g_{j-p} \cdots g_1 x, q+ j-p).
\end{align*}

Now we define a measurable perturbation function $\tilde f$ on $\bb A_p$ by setting, for any $a=(g_0,\ldots,g_{p},x,q) \in \bb A_p$, 
\begin{align} \label{function f tilde001}
\tilde f(a) = f_q(g_1,\ldots, g_p,x).
\end{align}
For any $k \in \bb N$ and $a \in \bb A_p$, we denote
\begin{align}\label{def-Fka}
\mathcal F_k(a) = \bb E_a e^{\alpha |\tilde f(\xi_k) |},
\end{align}
where $\alpha>0$ is a constant given in \eqref{exp mom for g 002}. 
Note that, by \eqref{exp mom for g 002}, for any $q \in \bb N$, 
we have $\mathcal F_k(g_0, \ldots, g_p, x, q) < \infty$ for $\bb P \otimes \nu$-almost all $g_0, \ldots, g_p, x$.

For any $t \in \bb R$, define the function 
$\tilde \tau^{\mathfrak f}_{t}: \bb A_p^{\bb N} \to \bb N \cup \{\infty \} $ by setting 
\begin{align} \label{def-tau-f-y}
 \tilde\tau^{\mathfrak f}_{t} = \min \left\{ k\geq 1:  t + \sum_{i=1}^{k} \sigma_p(\xi_i) + \tilde f(\xi_k)-\tilde f(\xi_0) < 0 \right\}.
\end{align}
It is straightforward to verify that $\tilde\tau^{\mathfrak f}_{t}$ is a $(\mathscr G_n)_{n\geq 0}$-stopping time on the space $\bb A_p^{\bb N}$. 

With this notation,  for any $t \in \bb R$ and $n\geq 1$, we have
\begin{align} \label{Expect-E_x001} 
I_1
& = \int_{\bb X}  \bb E \left( t +S_{n+p}^{x} ; \tau^{\mathfrak f}_{x, t}>n \right) \nu(dx)  \notag\\
& =  \int_{\bb X}  \int _{\bb G^{\{0,\ldots, p\}} } \bb E_{(g_0, \ldots, g_{p}, x, 0)} 
   \bigg( t +  \sum_{i=1}^{n+p} \sigma_p(\xi_i); \tilde\tau^{\mathfrak f}_{t} >n \bigg)   
     \mu(dg_{0}) \ldots \mu(dg_{p})  \nu(dx)   \notag\\
& = \int_{\bb X}  \int _{\bb G^{\{0,\ldots, p\}} }  W^{\mathfrak f}_{n} ((g_{0},\ldots,g_{p}, x, 0), t)  \mu(dg_{0}) \ldots \mu(dg_{p})
\nu(dx),
\end{align}
where, for $a  \in \bb A_p$, $t \in \bb R$ and $n \geq 1$, we denote 
\begin{align} \label{EXPECT-E_x-001}
W^{\mathfrak f}_{n}(a, t)
 = \bb E_a \bigg( t + \sum_{i=1}^{n+p} \sigma_p(\xi_i); \tilde \tau^{\mathfrak f}_{t} >n \bigg). 
\end{align}

We will make use of the following lemma, which is a consequence of the 
Markov property of the sequence $\xi_0, \xi_1, \ldots$.

\begin{lemma} \label{Lemma-Markov}
For any $p\geq 1$,  $a\in \bb A_p$,  $n > k \geq 1$, $t \in \bb R$, 
and any $\mathscr G_k$-measurable event $B_k$,  
we have 
\begin{align*} 
&\bb E_a \left( t +  \sum_{i=1}^{n+p} \sigma_p(\xi_i); \tilde \tau^{\mathfrak f}_{t} > n, B_k \right) \\
&\quad = \bb E_a \left[ W^{\mathfrak f}_{n-k}\left( \xi_{k}, t + \sum_{i=1}^{k} \sigma_p(\xi_i)+\tilde f(\xi_k)-\tilde f(\xi_0) \right);  \tilde \tau^{\mathfrak f}_{t} >k, B_k  \right].
\end{align*}
\end{lemma}

\begin{proof} 
For brevity, we denote for $n > k \geq 1$ and $t \in \bb R$,  
\begin{align*} 
A_{k,n}(t) = \bigg\{ t + \sum_{i=1}^{j} \sigma_p(\xi_{i}) + \tilde f(\xi_j)-\tilde f(\xi_0)\geq 0, \  k+1 \leq j\leq n \bigg\}.
\end{align*}
By \eqref{def-tau-f-y}, it holds that $\{\tilde \tau^{\mathfrak f}_{t}>n\} =\{ \tilde \tau^{\mathfrak f}_{t} >k\} \bigcap A_{k,n}(t)$. 
Hence, by taking the conditional expectation with respect to $\mathscr G_k$, 
we get that for any $a\in \bb A_p$,  $n > k \geq 1$ and $t \in \bb R$, 
\begin{align} \label{PPPPR001}
\bb E_a \bigg( t +  \sum_{i=1}^{n+p} \sigma_p(\xi_i); \tilde \tau^{\mathfrak f}_{t} > n \bigg | 
\mathscr G_k \bigg) 
 = \mathds 1_{\{\tilde \tau^{\mathfrak f}_{t} > k\} } \bb E_a \bigg( t +  \sum_{i=1}^{n+p} \sigma_p(\xi_i); A_{k,n}(t) \bigg | 
\mathscr G_k \bigg).
\end{align}
We shall use the  following conditioning formula: 
for any bounded Borel measurable function $h: \bb A_p ^{n+p+1} \to \bb R$ and $\eta_0,\ddd,\eta_{k} \in \bb A_p$,
 denote by $h_{\eta_0,\ddd,\eta_{k}} $ the bounded Borel measurable function 
on ${\bb A}_p^{n+p-k+1}$ defined by 
$h_{\eta_0,\ddd,\eta_{k}} (\xi_0,\ddd,\xi_{n+p-k})= h(\eta_0,\ddd,\eta_{k},\xi_1,\ddd,\xi_{n+p-k}),$
then it holds
\begin{align} \label{PPPPR002}
\bb E_a (h(\xi_0,\ddd,\xi_{n+p} )\big | \xi_0=\eta_0,\ddd,\xi_{k}=\eta_{k} ) 
= \bb E_{\eta_{k}} h_{\eta_0,\ddd,\eta_{k}}(\xi_{0},\ddd, \xi_{n+p-k}). 
\end{align}
Choose $\eta_0,\ddd,\eta_{k} \in \bb A_p$ and set 
$t' = t +  \sum_{i=1}^{k} \sigma_p(\eta_i) +f(\eta_k)-f(\eta_0)$. 
Then, by \eqref{PPPPR002}, we get 
 \begin{align} \label{PPPPR003}
&\bb E_a \bigg( t +  \sum_{i=1}^{n+p} \sigma_p(\xi_i); A_{k,n}(t) \bigg | 
\xi_0=\eta_0,\ddd,\xi_{k}=\eta_{k} \bigg) \notag \\ 
&= \bb E_a \bigg( t +  \sum_{i=1}^{k} \sigma_p(\xi_i) +   \sum_{i=k+1}^{n+p} \sigma_p(\xi_i); A_{k,n}(t) \bigg |  
\xi_0=\eta_0,\ddd,\xi_{k}=\eta_{k} \bigg)  \notag \\ 
&= \bb E_{\eta_{k}} \bigg( t'  +  \sum_{i=1}^{n+p-k} \sigma_p(\xi_{i}); \tilde \tau^{\mathfrak f}_{t'} > n-k \bigg) 
  = W^{\mathfrak f}_{n-k} (\eta_{k}, t').
\end{align}
Since $\{ \tilde \tau^{\mathfrak f}_{t} >k\}$ and $B_k$ are $\mathscr G_k$-measurable, 
the assertion of the lemma now follows from \eqref{PPPPR001} and \eqref{PPPPR003}. 
\end{proof}

We now justify a martingale representation of the expectation $W^{\mathfrak f}_n(a,t)$. 
We define the function $h_p:  \bb A_p \to  \bb R$ by setting, for any $a=(g_{0}, \ldots, g_{p}, x,q) \in \bb A_p$, 
\begin{align} \label{def of func h-001}
h_p (a)=h_p (g_{0}, \ldots, g_{p}, x,q) = \sigma(g_{p}, g_{p-1} \cdots  g_1 x). 
\end{align}
In view of the definition of the 
probability measure $\bb P_a$ defined on $\bb A_p^{\bb N}$, 
for any $i \geq 0$ and $a\in \bb A_p$, we have $h_p(\xi_{i}) = \sigma_p(\xi_{i+p})$, $\bb P_a$-almost surely. Consequently, 
the quantity $W^{\mathfrak f}_n(a,t)$ defined by \eqref{EXPECT-E_x-001} can be rewritten as follows: 
for any $a = (g_{0},\ldots,g_{p},x,q) \in \bb A_p$,  $t \in \bb R$ and $n \geq 1$,
\begin{align} \label{EXPECT-E_x-002aa}
W^{\mathfrak f}_n(a,t)
= \bb E_a \bigg( t + \sigma(g_{p}\cdots g_{1}, x) +  \sum_{i=1}^{n} h_p(\xi_i); \tilde \tau^{\mathfrak f}_{t} >n \bigg). 
\end{align}
Let $a \in \bb A_p$. 
Set $M_0=0$ and, for $n\geq 1$, $p \geq 1$, $\bb P_a$-almost surely, 
\begin{align} \label{def-Mn-martingel-001}
M_n = \sum_{i=p+1}^{n+p} \sigma_p(\xi_{i})=\sum_{i=1}^{n} \sigma_p(\xi_{i+p}) = \sum_{i=1}^{n} h_p(\xi_i).
\end{align}
With this notation, the expectation $W^{\mathfrak f}_n(a,t)$ becomes: 
\begin{align} \label{EXPECT-E_x-002}
W^{\mathfrak f}_n(a,t)
= \bb E_a \left( t + \sigma(g_{p}\cdots g_{1}, x) + M_n; \tilde \tau^{\mathfrak f}_{t} >n \right).
\end{align}
The useful property related to this representation is that the sequence $(M_n)_{n\geq 0}$ is a 
$(\mathscr G_n)_{n\geq 0}$-martingale. 
This property is formally stated in the following lemma. 

\begin{lemma} \label{Lemma-Martingale001}
For any $p \geq 1$ and  $a\in \bb A_p,$ the sequence $(M_n)_{n\geq 0}$ is a $(\mathscr G_n)_{n\geq 0}$-martingale under $\bb P_a$.
\end{lemma}

\subsection{A priori control on the exit times} \label{sec-control-exit times}

In this subsection, we assume the same setting as in the previous subsection, 
where the functions $f_n$  depend only on the first $p$ coordinates in $\Omega$, for some fixed $p \geq 1$. 
We will provide a bound for $I_2$ in \eqref{MAIN_GOAL-002} using a control on the probability $\mathbb{P}_a ( \tilde \tau^{\mathfrak f}_{t} > n )$. 
 
The following lemma gives an upper bound for $\mathbb{P}_a ( \tilde \tau^{\mathfrak f}_{t} > n )$. 
In particular, it shows that, for any $a \in \bb A_p$, the stopping time $\tilde \tau^{\mathfrak f} _{t}$ is $\mathbb{P}_a$-almost surely finite.
Recall that $\mathcal F_k(a)$ is defined by \eqref{def-Fka}.

\begin{lemma}\label{Lem-tau-prior}
There exist constants $c, \beta > 0$ such that for any $a= (g_0,\ldots,g_{p},x,q) \in \bb A_p$, $t \in \bb R$, $n\geq 1$ and $1\leq p < n $,  
\begin{align*}
\mathbb{P}_a \left( \tilde \tau^{\mathfrak f}_{t} = n \right) 
&\leq \mathbb{P}_a \left( \tilde \tau^{\mathfrak f}_{t} \geq n \right)  \\
& \leq  c \frac{\max \{t,0\} +  |\sigma(g_p\cdots g_1,x)|   +|\tilde{f}(a)|+ \log n}{(n-p)^{\beta}}
 + \frac{c }{n^8} \sum_{k = p+1}^n \mathcal F_k(a).   
\end{align*}
\end{lemma}

\begin{proof}
We fix $a=(g_0,\ldots,g_{p},x,q) \in \bb A_p$ and $t\in \bb R$.
Set $x'=g_p\cdots g_1x$ and $t' = t + \sigma(g_p\cdots g_1, x) - \tilde f(a) $.
Consider the event 
\begin{align*}
B_{n}=\left\{  \max_{p < k\leq n}\left\vert  \tilde f(\xi_k) \right\vert \leq c_1 \log n \right\}, 
\end{align*}
where the constant $c_1>0$ will be chosen later.
Then 
\begin{align}\label{decom-probab-tauft}
\mathbb{P}_a \left( \tilde \tau^{\mathfrak f}_{t} \geq n \right) 
\leq  \mathbb{P}_a \left( \tilde \tau^{\mathfrak f}_{t} \geq n,  B_n \right)  
  +  \mathbb{P}_a \left(  B_n^c \right). 
\end{align}
By Proposition \ref{Prop-vartheta}, there exist constants $c, \beta>0$ such that for any $n \geq 1$, $x \in \bb X$ and $t \in \bb R$, 
\begin{align*} 
\bb P (\vartheta_{x,t}  \geq n) \leq c \frac{1+\max \{t,0\}}{n^{\beta}},
\end{align*}
where  $\vartheta_{x,t}$ is a stopping time defined by \eqref{stopping time theta 001}. 
Using this bound, for the first term in \eqref{decom-probab-tauft},  we get
\begin{align}\label{probab-tauft-Bn}
\mathbb{P}_a \left( \tilde \tau^{\mathfrak f}_{t} \geq n,  B_n \right)
&\leq  \mathbb{P} \left( \vartheta_{x' , t' + c_1 \log n} \geq n - p  \right)   \notag\\
&\leq  c \frac{\max \{t,0\} + | \sigma(g_p\cdots g_1, x) | +|\tilde{f}(a)| + \log n}{(n -p)^{\beta}}, 
\end{align}
with $c>0$ not depending on $\tilde{f}$. 
For the second term in \eqref{decom-probab-tauft}, by Markov's inequality and \eqref{def-Fka}, we obtain 
\begin{align} \label{PAbar bound}
\mathbb{P}_a\left( B_n^c\right) \leq
     \sum_{k = p+1}^n
		\mathbb{P}_a\left( \left\vert \tilde{f}(\xi_k)\right\vert  >  c_1 \log n \right) 
\leq   \frac{c}{n^{\alpha c_1}} \sum_{k = p+1}^n \mathcal F_k(a)  \leq   \frac{c }{n^8} \sum_{k = p+1}^n \mathcal F_k(a),
\end{align}
where in the last inequality we take $c_1>0$ to be sufficiently large.  
Combining the bounds \eqref{decom-probab-tauft}, \eqref{probab-tauft-Bn} and \eqref{PAbar bound}
completes the proof of the lemma. 
\end{proof}

As a first consequence of Lemma \ref{Lem-tau-prior}, we obtain an extension of Proposition \ref{Prop-vartheta}, 
which will be used in the proof of Corollary \ref{Lem-bound-tau-p-000aa}.

\begin{corollary}\label{Lem-bound-tau-p-000}
There exist constants $\ee, \beta, c > 0$ 
such that for any $ n \geq 1$, $p\leq n^{\ee}$, $x\in \bb X$ and $t\in \bb R$ 
and any sequence $\mathfrak f = (f_n)_{n \geq 0}$ of $\mathscr A_p$-measurable functions, 
\begin{align*}
\int_{\bb X} \mathbb{P} \left( \tau^{\mathfrak f}_{x, t}  > n \right) \nu(dx)
\leq c \frac{  \max \{t,0\} +1 }{n^{\beta}} C_{\alpha}(\mathfrak f). 
\end{align*}
\end{corollary}

\begin{proof}
By construction, we have 
\begin{align*}
\int_{\bb X} \mathbb{P} \left( \tau^{\mathfrak f}_{x, t}  > n \right) \nu(dx)
 =  \int_{\bb X} \int _{\bb G^{\{0,\ldots, p\}} }  \mathbb{P}_{(g_0,\ldots, g_p, x, 0)} 
\left(   \tilde \tau^{\mathfrak f}_{t}  > n \right)   \mu(dg_0) \ldots \mu(dg_p) \nu(dx). 
\end{align*}
The conclusion follows from Lemma \ref{Lem-tau-prior} 
since, for every $k \in \bb N$, we have 
\begin{align}\label{stationary-nu-001}
& \int_{\bb X}  \int_{\bb G^{p+1} }  \mathcal F_k(g_0,\ldots, g_p, x, 0)  \mu(dg_0) \ldots \mu(dg_p) \nu(dx) \notag\\
& =  \int_{\bb X}  \int_{\bb G^k} \int_{\Omega} e^{\alpha |f_k(\omega, g_k \cdots g_1 x)|}  \bb P(d\omega) 
 \mu(dg_1) \ldots \mu(dg_k)  \nu(dx) \notag\\
& =   \int_{\bb X}  \int_{\Omega}  e^{\alpha |f_k(\omega, x)|} \bb P(d\omega)   \nu(dx) 
 \leq C_{\alpha}(\mathfrak f),
\end{align}
where we have used the fact that the measure $\nu$ is $\mu$-stationary 
and $C_{\alpha}(\mathfrak f)$ is the finite constant appearing in \eqref{exp mom for g 002}. 
\end{proof}

As a second consequence of Lemma \ref{Lem-tau-prior}, 
we get the following bound which will be used below to bound the term $I_2$ defined in \eqref{MAIN_GOAL-002bbb}. 

\begin{lemma}\label{Lem-bound-tau_p-new}
There exist constants $\beta, c > 0$ such that for any $ n \geq 1$, $p\leq n^{\beta}$, $a=(g_0,\ldots,g_{p},x, q) \in \bb A_p$ and $t\in \bb R$, 
\begin{align*}
J & = \mathbb{E}_a \bigg( \bigg| - \sum_{i=n+1}^{n+p} \sigma_p(\xi_i) + \tilde f(\xi_n) - \tilde f(\xi_0) \bigg|;  \tilde \tau^{\mathfrak f}_{t}  > n \bigg) \notag\\
&\leq c \frac{  \max \{t,0\} + \Big|\sigma(g_p\cdots g_1,x) \Big| + |\tilde f(a)|}{n^{\beta}} 
\left(  \mathcal F_n(a)  + |\tilde f(a)| \right)
 +  |\tilde f(a)| \frac{ \sum_{k = p+1}^n \mathcal F_k(a) }{n^3}.
\end{align*}
\end{lemma}

\begin{proof}
By H\"older's inequality and Lemma \ref{Lem-tau-prior}, we get
\begin{align*}
J & \leq  \bigg[ \bb E_a^{1/2} \bigg( \bigg|\sum_{i=n+1}^{n+p} \sigma_p(\xi_i) \bigg|^2  \bigg)  
                +  \mathbb{E}_a^{1/2}\left(  |\tilde f (\xi_n)|^2 \right) +|\tilde f(\xi_0)|  \bigg]  
\bb P_a^{1/2} \left(\tilde \tau^{\mathfrak f}_{t}  > n\right)  \notag\\
& \leq c\left( \sqrt{p}  +   \mathcal F_n(a)^{1/2} + |\tilde f(a)|  \right) \notag\\
& \qquad \times \left[\frac{\max \{t,0\} +  |\sigma(g_p\cdots g_1,x)|   +|\tilde f(a)|+ \log n}{(n-p)^{3\beta}}
+   \frac{ \sum_{k = p+1}^n \mathcal F_k(a) }{n^8} \right]^{1/2},
\end{align*}
from which the assertion follows immediately.
\end{proof}

Now, using Lemma \ref{Lem-bound-tau_p-new} we give an estimate for the term $I_2$ in \eqref{MAIN_GOAL-002bbb}.

\begin{corollary}\label{Lem-bound-tau-p}
There exist constants $\ee, \beta, c > 0$ 
such that for any $ n \geq 1$, $p\leq n^{\ee}$, $x\in \bb X$ and $t\in \bb R$ 
and any sequence $\mathfrak f = (f_n)_{n \geq 0}$ of $\mathscr A_p$-measurable functions, 
\begin{align*}
 |I_2| = \left| \int_{\bb X} \mathbb{E}\left( S_{n}^x - S_{n+p}^x  + f_n\circ T^n - f_0; \ \tau^{\mathfrak f}_{x, t}  > n \right) \nu(dx) \right|
\leq c \frac{  \max \{t,0\} +1 }{n^{\beta}} C_{\alpha}(\mathfrak f). 
\end{align*}
\end{corollary}

\begin{proof}
Similarly to \eqref{Expect-E_x001}, by construction we have 
\begin{align*}
I_2 & = \int_{\bb X} \int _{\bb G^{\{0,\ldots, p\}} }  \mathbb{E}_{(g_0,\ldots, g_p, x, 0)} 
\bigg( -\sum_{i=n+1}^{n+p} \sigma_p(\xi_i) + \tilde f(\xi_n) - \tilde f(\xi_0);  \tilde \tau^{\mathfrak f}_{t}  > n  \bigg)  \notag\\
& \qquad\qquad \mu(dg_0) \ldots \mu(dg_p) \nu(dx). 
\end{align*}
The expectation inside the above integral is bounded 
by using Lemma \ref{Lem-bound-tau_p-new} with $\alpha/2$ instead of $\alpha$ in $\mathcal F_k$. 
The conclusion now follows by using \eqref{stationary-nu-001}. 
\end{proof}

As a further consequence, we get the following bound which will be utilized in the next section,
specifically in the proof of Lemma \ref{Lemma 4}. 

\begin{lemma}\label{Lemma 1} 
There exist constants $c, \ee_0 > 0$ such that, for any $\ee \in (0, \ee_0)$,   
$1\leq p < n$, $t \geq  n^{1/2-\ee }$ and $a=(g_0,\ldots,g_{p},x, q) \in \bb A_p$, 
\begin{align*}
& \mathbb{E}_a \bigg( \bigg\vert t + \sigma(g_p\cdots g_1,x) + \sum_{i=1}^{\tilde \tau^{\mathfrak f}_{t}} h_p(\xi_i) \bigg\vert; 
 \ p < \tilde \tau^{\mathfrak f}_{t} \leq n\bigg) \notag\\ 
& \leq \frac{t}{n^{\ee }} + c n  e^{- \frac{\alpha}{2(p+3)} n^{1/2-2\ee} } 
 \Big(\Big|t + \sigma(g_p\cdots g_1,x) \Big|+ n^{1/2}\Big)  \sum_{k = p+1}^n \mathcal F_k(a). 
\end{align*}
\end{lemma}

\begin{proof}
Consider the event 
\begin{align*}
A_{n}=\left\{ \max_{p < k \leq n}\left\vert h_p(\xi_k) \right\vert \leq \frac{1}{p + 3} n^{1/2-2\ee },  
          \max_{p < k \leq n}\left\vert \tilde f(\xi_k) \right\vert \leq  \frac{1}{p + 3} n^{1/2-2\ee } \right\}. 
\end{align*}
We denote $t'  = t + \sigma(g_p\cdots g_1,x)$ and write 
\begin{align}\label{eq-lemma1-000}
 \bb{E}_a \bigg( \bigg\vert t' + \sum_{i=1}^{\tilde \tau^{\mathfrak f}_{t}} h_p(\xi_i) \bigg\vert; \  
 p < \tilde \tau^{\mathfrak f}_{t} \leq n \bigg)  
& = \bb{E}_a \bigg( \bigg\vert t' + \sum_{i=1}^{\tilde \tau^{\mathfrak f}_{t}} h_p(\xi_i) \bigg\vert; \  
 p < \tilde \tau^{\mathfrak f}_{t} \leq n, A_{n} \bigg)  \notag \\
& \quad  
+  \bb{E}_a \bigg( \bigg\vert t' + \sum_{i=1}^{\tilde \tau^{\mathfrak f}_{t}} h_p(\xi_i)  \bigg\vert; \  
 p < \tilde \tau^{\mathfrak f}_{t} \leq n,   A_n^c \bigg). 
\end{align}
For the first term, on the event $\{p < \tilde \tau^{\mathfrak f}_{t}\}$, since $\tilde \tau^{\mathfrak f}_{t}$ is defined as the first integer $k > p$
when $t' + \sum_{i=1}^{k -p} h_p(\xi_i) + \tilde f(\xi_{k}) - \tilde f(\xi_0)$ 
becomes negative (see \eqref{def-tau-f-y}), 
and the size of the jump is bounded by 
$|h_p(\xi_{\tilde \tau^{\mathfrak f}_{t} -p})| + |\tilde f(\xi_{\tilde \tau^{\mathfrak f}_{t}})| + |\tilde f(\xi_{\tilde \tau^{\mathfrak f}_{t}-1})|$
which in turn does not exceed $\frac{3}{p +3}n^{1/2-2\ee }$ on the event $A_{n}$, it follows that
\begin{align}\label{eq-lemma1-R1}
\bb{E}_a \bigg( \bigg\vert t' + \sum_{i=1}^{\tilde \tau^{\mathfrak f}_{t}} h_p(\xi_i) \bigg\vert; \  
 p < \tilde \tau^{\mathfrak f}_{t} \leq n, A_{n} \bigg)
& \leq  \bb{E}_a \bigg( \bigg\vert t' + \sum_{i=1}^{\tilde \tau^{\mathfrak f}_{t} -p} h_p(\xi_i) \bigg\vert; \  p < \tilde \tau^{\mathfrak f}_{t} \leq n, A_{n} \bigg)   \notag\\
& \quad   + \bb{E}_a \bigg( \bigg\vert  \sum_{i=\tilde \tau^{\mathfrak f}_{t} -p + 1}^{\tilde \tau^{\mathfrak f}_{t}} h_p(\xi_i) \bigg\vert;\  p < \tilde \tau^{\mathfrak f}_{t} \leq n, A_{n} \bigg)  \notag\\
& \leq \frac{3}{p+3} n^{1/2-2\ee } + \frac{p}{p +3} n^{1/2-2\ee }  \notag\\
& \leq  n^{1/2-2\ee } \leq {t\over n^{\ee }},
\end{align}
where the last inequality holds since $t\geq n^{1/2-\ee }$.

We proceed to give an upper bound for the second term on the right-hand side of \eqref{eq-lemma1-000}.
By Markov's inequality, 
the exponential moment assumption \eqref{exp mom for f 001} and the fact that $\mathcal F_k(a) \geq 1$ (see \eqref{def-Fka}), we get
\begin{align*}
\mathbb{P}_a (A_n^c) 
& \leq 
 \mathbb{P}_a\left(  \max_{p < k\leq n} |h_p(\xi_k)| > \frac{1}{p+3} n^{1/2-2\ee }   \right)
   +  \mathbb{P}_a\left(   \max_{p < k\leq n} |\tilde f(\xi_{k})| >  \frac{1}{p+3} n^{1/2-2\ee }     \right)  \notag \\
& \leq 
(n-p) \sup_{x\in \bb X}  \mathbb{P}\left( |\sigma(g_{1}, x)| > \frac{1}{p+3} n^{1/2-2\ee }   \right)
   + \sum_{k = p+1}^n   \mathbb{P}_a \left( |\tilde f(\xi_k)| >  \frac{1}{p+3} n^{1/2-2\ee }     \right)  \notag \\
& \leq   c (n-p) e^{- \frac{\alpha}{p+3} n^{1/2-2\ee } } +  e^{- \frac{\alpha}{p+3} n^{ 1/2 -2\ee }} 
 \sum_{k = p+1}^n \mathcal F_k(a) \notag \\  
& \leq  c  e^{- \frac{\alpha}{p+3} n^{ 1/2 - 2\ee}} \sum_{k = p+1}^n \mathcal F_k(a).  
\end{align*}
Therefore, by H\"older's inequality, 
\begin{align}\label{eq-lemma1-R2}
\bb{E}_a \bigg( \bigg\vert t' + \sum_{i=1}^{\tilde \tau^{\mathfrak f}_{t}} h_p(\xi_i) \bigg\vert; \  
 p < \tilde \tau^{\mathfrak f}_{t} \leq n,   A_n^c  \bigg)
& =  \sum_{k =p+ 1}^n \bb{E}_a \bigg( \bigg\vert t' + \sum_{i=1}^{\tilde \tau^{\mathfrak f}_{t}} h_p(\xi_i) \bigg\vert;
\  \tilde \tau^{\mathfrak f}_{t} = k, A_n^c  \bigg)  \notag\\
& \leq  \sum_{k =p+ 1}^n \bb{E}_a \bigg( \bigg\vert t' + \sum_{i=1}^{k} h_p(\xi_i) \bigg\vert; \  A_n^c \bigg)   \notag\\
& \leq     \sum_{k = p+1}^n  \bb{E}_a^{1/2} \bigg( \bigg\vert t' + \sum_{i=1}^{k} h_p(\xi_i) \bigg\vert^2 \bigg)  
             \bb P_a^{1/2}(A_n^c)  \notag\\
& \leq 
c n (|t'|+ n^{1/2})  e^{- \frac{\alpha}{2(p+3)} n^{1/2-2\ee } } \sum_{k = p+1}^n \mathcal F_k(a),
\end{align}
where in the last inequality we used the bound
\begin{align}\label{Inequ-Minkowski-01}
\bb{E}_a^{1/2} \bigg( \bigg\vert t' + \sum_{i=1}^{k} h_p(\xi_i)  \bigg\vert^2  \bigg)
\leq |t'|+ \sup_{x\in \bb X} \sqrt{\mathbb{E}  (S_{k}^{x} )^{ 2 }    }
\leq |t'|+c n^{1/2}, 
\end{align}
which holds by Minkowski's inequality. 
Combining \eqref{eq-lemma1-000}, \eqref{eq-lemma1-R1} and \eqref{eq-lemma1-R2} completes the proof of the lemma. 
\end{proof}

\subsection{Approximation through finite-size perturbations} \label{sec: approximation by finite range perturb}
Let $\mathfrak f = (f_n)_{n \geq 0}$ be  a sequence of perturbations, where
the measurable functions $f_n: \Omega\times \bb X \to \bb R$ depends on the entire sequence $\omega=(g_{i})_{i \geq 1} \in \Omega$.
Assume that the finite-size approximation property \eqref{approxim rate for gp-002} holds, where, for any $p \geq 1$, 
$\mathfrak f_p = (f_{n, p})_{n \geq 0}$ is the sequence 
of perturbation functions $f_{n,p}$ depending only on a finite number of coordinates, as defined by \eqref{def-approxi-fnp}.
For any $p \geq 1$, corresponding to the sequence $\mathfrak f_p$, we can define the stopping time $\tau_{x,t}^{\mathfrak f_p}$ and the expectation $U^{\mathfrak f_p}_n(t)$ by 
 \eqref{def-stop time with preturb-001} and \eqref{def-U-f-n-t-001} with $\mathfrak f$ replaced by $\mathfrak f_p$. 
Specifically, for any $p \geq 1$, $x\in \bb X$ and $t\in \bb R$, 
\begin{align}\label{def-stop time with preturb-001-p}
\tau_{x,t}^{\mathfrak f_p} = \min \left\{ k \geq 1: t+S^x_{k} + f_{k,p}\circ T^k(\omega,x) - f_{0,p}(\omega,x) < 0\right\} 
\end{align}
and 
\begin{align}\label{def-U-f-n-t-001-p}
U^{\mathfrak f_p}_n(t) 
= \int_{\bb X} \bb E \left( t + S^{x}_n + f_{n,p}\circ T^n(\omega,x) - f_{0,p}(\omega,x);  
\tau^{\mathfrak f_p}_{x,t} > n \right) \nu(dx). 
\end{align} 
Hereafter, $f_0(x)$ stands for the random variable $\omega \mapsto f(\omega,x)$ 
and $f_{n,p}\circ T^n(x)$ stands for the random variable $\omega \mapsto f_{n,p}\circ T^n(\omega,x)$.
Our goal is to establish an approximation of $U_n^{\mathfrak f}$ 
in terms of $U_n^{\mathfrak f_p}$. This is provided by the following lemma.

\begin{proposition} \label{Prop g approx 002}
Assume that $\mathfrak f = (f_n)_{n \geq 0}$ is a sequence of  measurable functions on $\Omega \times \bb X$ 
satisfying the moment condition \eqref{exp mom for g 002} and the approximation property \eqref{approxim rate for gp-002}.
Then, for any $\delta,\gamma > 0$, 
there exists a constant $c>0$ such that for any $n \geq 1$, $p \geq n^{\delta}$ and $t \in \bb R,$ 
\begin{align}\label{AA-bound 002-001}
& U^{\mathfrak f_p}_n (t- n^{-\gamma}) - c   \left( \max \{t,0\} + C_{\alpha}(\mathfrak f) \right) e^{ -\beta n^{\delta}} 
D_{\alpha,\beta}(\mathfrak f)
  \notag \\
&\qquad\qquad\quad
\leq  U^{\mathfrak f}_n (t)   \leq U^{\mathfrak f_p}_n (t+ n^{-\gamma}) 
+ c \left( \max \{t,0\} + C_{\alpha}(\mathfrak f) \right) e^{ -\beta n^{\delta}} D_{\alpha,\beta}(\mathfrak f),
\end{align}
where $\alpha$ and $\beta$ are positive constants from conditions \eqref{exp mom for g 002} and \eqref{approxim rate for gp-002}. 
\end{proposition}

\begin{proof}
For $x\in \bb X$ and $n \geq 1$, consider the event 
\begin{align*}
B_{n,x}=\left\{ \omega :  \max_{0 \leq k\leq n}\left\vert  f_k\circ T^k (\omega,x) - f_{k,p}\circ T^k(\omega,x) \right\vert \leq \frac{1}{2}n^{-\gamma} \right\}. 
\end{align*}
Since
\begin{align*}
& \bb E \left( t + S_n^x + f_n \circ T^n (\omega,x)-f_0(\omega,x); \tau^{\mathfrak f}_{x, t} > n, B_{n,x} \right)  \notag\\
& \leq  \bb E \left( t + n^{-\gamma} + S_n^x + f_{n,p} \circ T^n(\omega,x) -  f_{0, p}(\omega,x);  \tau^{\mathfrak f}_{x, t} > n,  B_{n,x} \right)  \notag\\ 
& \leq \bb E \left( t + n^{-\gamma} + S_n^x + f_{n,p} \circ T^n(\omega,x) - f_{0, p}(\omega,x);  \tau^{\mathfrak f_p}_{x, t+n^{-\gamma}} > n  \right), 
\end{align*}
where $\tau^{\mathfrak f_p}_{x, t+n^{-\gamma}}$ is defined by \eqref{def-stop time with preturb-001-p}, 
we get 
\begin{align*}
&\bb E \left( t + S_n^x + f_n \circ T^n (\omega,x)-f_0(\omega,x); \tau^{\mathfrak f}_{x, t} > n \right)  \notag\\
&= \bb E \left( t + S_n^x + f_n \circ T^n (\omega,x)-f_0(\omega,x); \tau^{\mathfrak f}_{x, t} > n, B_{n,x} \right)  \notag\\
&\qquad\qquad\qquad +  \bb E \left( t + S_n^x + f_n \circ T^n (\omega,x)-f_0(\omega,x); \tau^{\mathfrak f}_{x, t} > n, B_{n,x}^c \right) \notag\\  
&\leq  \bb E \left( t + n^{-\gamma} + S_n^x + f_{n,p} \circ T^n(\omega,x)  -  f_{0, p}(\omega,x);  \tau^{\mathfrak f_p}_{x, t+n^{-\gamma}} > n  \right)  \notag\\
&\qquad\qquad\qquad +  \bb E \left( t + S_n^x + f_n \circ T^n (\omega,x)-f_0(\omega,x); \tau^{\mathfrak f}_{x, t} > n, B_{n,x}^c \right).
\end{align*}
By integrating over $x\in\bb X$, 
and using \eqref{def-U-f-n-t-001} and \eqref{def-U-f-n-t-001-p}, we obtain
\begin{align*}
U_n^{\mathfrak f}(t) \leq U_n^{\mathfrak f_p}(t+n^{-\gamma}) 
+ \int_{\bb X} \bb E \left( t + S_n^x + f_n \circ T^n (\omega,x)-f_0(\omega,x); \tau^{\mathfrak f}_{x, t} > n, B_{n,x}^c \right)    \nu(dx).
\end{align*}
By the Cauchy-Schwarz inequality, we see that 
\begin{align*} 
&\int_{\bb X} \bb E \left( t + S_n^x + f_n \circ T^n (\omega,x)-f_0(\omega,x); \tau^{\mathfrak f}_{x, t} > n, 
B_{n,x}^c \right) \nu(dx) \notag\\
&= t \int_{\bb X} \bb P \left( \tau^{\mathfrak f}_{x, t}>n, B_{n,x}^c \right) \nu(dx)   \notag\\
& \quad + \int_{\bb X} \bb E \left(  S_n^x + f_n \circ T^n (\omega,x)-f_0(\omega,x); \tau^{\mathfrak f}_{x, t} > n, 
B_{n,x}^c \right) \nu(dx) \notag\\
&\leq  \max \{t,0\} \int_{\bb X} \bb P ( B_{n,x}^c ) \nu(dx)  \notag\\
& \quad + \left(\int_{\bb X} \bb E \left( \big| S_n^x + f_n \circ T^n (\omega,x)-f_0(\omega,x) \big|^2 \right) \nu(dx) \right)^{1/2} 
\left(\int_{\bb X} \bb P (B_{n,x}^c) \nu(dx)\right)^{1/2}.
\end{align*}
By the definition of $T$ on the space $\Omega\times \bb X$ (cf.\ \eqref{def-T-Omega-X}), 
the $\mu$-stationarity of the measure $\nu$, Chebyshev's inequality and the approximation property \eqref{approxim rate for gp-002}, we obtain
\begin{align} \label{proba-Bnx-001}
\int_{\bb X} \bb P(B_{n,x}^c) \nu(dx)
&\leq  \sum_{k=0}^n  \int_{\bb X} \bb P\left( e^{\alpha \left| f_k \circ T^k(\omega, x) - f_{k,p} \circ T^k(\omega, x)  \right|    } - 1 \geq 
  e^{\frac{\alpha}{2} n^{-\gamma}}-1\right) \nu(dx)  \notag\\
 &=  \sum_{k=0}^n  \int_{\bb X} 
 \bb P\left( e^{\alpha \left\vert   f_k(\omega,  x) - f_{k,p}(\omega, x)  \right\vert    } - 1 \geq e^{\frac{\alpha}{2} n^{-\gamma}}-1\right)  \nu(dx) \notag\\ 
&\leq  \frac{1}{e^{\frac{\alpha}{2} n^{-\gamma}}-1} \sum_{k=0}^n \int_{\bb X}   
 \bb E\left( e^{\alpha \left\vert   f_k(\omega,  x) - \bb E( f_{k}(\cdot,  x) | \scr A_p)(\omega) \right\vert    } - 1 \right) \nu(dx)  \notag\\
&\leq c  n^{1+\gamma} e^{ -\beta p}  D_{\alpha,\beta}(\mathfrak f) 
\leq 
c  n^{1+\gamma} e^{ -\beta n^{\delta}}  D_{\alpha,\beta}(\mathfrak f). 
\end{align}
Using Minkowski's inequality and the assumptions \eqref{exp mom for f 001}, \eqref{centering-001} and \eqref{exp mom for g 002}, we get
\begin{align} \label{bound-Snx-fn-001}
& \left(\int_{\bb X} \bb E\left( \big| S_n^x + f_n \circ T^n (\omega,x)-f_0(\omega,x) \big|^2 \right) \nu(dx) \right)^{1/2} \notag\\
&\leq   \left(\int_{\bb X} \bb E \big( |S_n^x|^2 \big) \nu(dx) \right)^{1/2}  + 
 \left(\int_{\bb X} \bb E\left( \big| f_n \circ T^n (\omega,x)-f_0(\omega,x) \big|^2 \right) \nu(dx) \right)^{1/2}  \notag\\
&\leq c \left( n + C_{\alpha}(\mathfrak f) \right).
\end{align}
Combining these bounds, we derive the upper bound in \eqref{AA-bound 002-001}. The lower bound can be established similarly. 
\end{proof}

By applying the same technique, we can further extend Proposition \ref{Prop-vartheta} to the case where the functions $(f_n)_{n \geq 0}$
may depend on infinitely many coordinates.

\begin{corollary}\label{Lem-bound-tau-p-000aa}
For any $\alpha,\beta>0$ and $B \geq 1$, there exist constants $\ee, c>0$ with the following property. 
Assume that $\mathfrak f = (f_n)_{n \geq 0}$ is a sequence of  measurable functions on $\Omega \times \bb X$ 
satisfying the moment condition \eqref{exp mom for g 002} and the approximation property \eqref{approxim rate for gp-002}
with $C_{\alpha}(\mathfrak f) \leq B$ and $D_{\alpha,\beta}(\mathfrak f) \leq B$. 
Then, for any $n \geq 1$ and $t\in \bb R$, we have
\begin{align*}
\int_{\bb X} \mathbb{P} \left( \tau^{\mathfrak f}_{x, t}  > n \right) \nu(dx)
\leq c \frac{  \max \{t,0\} +1 }{n^{\ee}}. 
\end{align*}
\end{corollary}

\begin{proof}
We use the notation from the proof of Proposition \ref{Prop g approx 002} and we write 
\begin{align*}
\int_{\bb X} \mathbb{P} \left( \tau^{\mathfrak f}_{x, t}  > n \right) \nu(dx)
& = \int_{\bb X} \mathbb{P} \left( \tau^{\mathfrak f}_{x, t}  > n, B_{n,x}  \right) \nu(dx) 
 + \int_{\bb X} \mathbb{P} \left( \tau^{\mathfrak f}_{x, t}  > n, B_{n,x}^c \right) \nu(dx)   \notag\\
 & \leq \int_{\bb X} \mathbb{P} \left( \tau^{\mathfrak f_p}_{x, t + n^{-\gamma}}  > n   \right) \nu(dx)  
  + \int_{\bb X} \mathbb{P} \left(  B_{n,x}^c \right) \nu(dx). 
\end{align*}
By Corollary \ref{Lem-bound-tau-p-000}, the first term is bounded by $c \frac{  \max \{t,0\} +1 }{n^{\ee}} C_{\alpha}(\mathfrak f)$.
By \eqref{proba-Bnx-001}, the second term is dominated by $c  n^{1+\gamma} e^{ -\beta n^{\delta}} D_{\alpha,\beta}(\mathfrak f)$. 
Combining these bounds gives the desired result. 
\end{proof}

To obtain Theorem \ref{Pro-Appendix-Final2-Inequ}, 
we actually have to work with a modified version of Proposition \ref{Prop g approx 002}. 
Specifically, we need to account for the twisted expectation involving $\theta_p$, 
where $\theta_p = \bb E (\theta | \scr A_p)$ for $\theta \in L^{\infty}(\Omega, \bb P)$, as defined by \eqref{def-theta-p}. 
According to \eqref{def-U-f-n-t-theta-001}, we have 
\begin{align*}
U^{\mathfrak f_p, \theta_p}_n(t) 
= \int_{\bb X} \bb E \left( \left( t + S^{x}_n + f_{n,p}\circ T^n(\omega,x) - f_{0,p}(\omega,x) \right) \theta_p(\omega);  \tau^{\mathfrak f_p}_{x,t} > n \right) \nu(dx). 
\end{align*}


\begin{proposition} \label{Prop UfA g approx 002-app}
For any $\gamma, \delta > 0$ and $B \geq 1$, there exist $b, c>0$ with the following property. 
Assume that $\mathfrak f = (f_n)_{n \geq 0}$ is a sequence of  measurable functions on $\Omega \times \bb X$ 
satisfying the moment condition \eqref{exp mom for g 002} and the approximation property \eqref{approxim rate for gp-002}
with $C_{\alpha}(\mathfrak f) \leq B$ and $D_{\alpha,\beta}(\mathfrak f) \leq B$. 
Assume that $\theta \in L^{\infty}(\Omega, \bb P)$ is non-negative
 with $\|\theta\|_{\infty} \leq B$ and $N_{\beta}(\theta) \leq B$. 
Then, for any $n \geq 1$, $p \geq n^{\delta}$ and $t \in \bb R,$ we have 
\begin{align}\label{bound UfA 002-001-app}
& U^{\mathfrak f_p, \theta_p}_n (t- n^{-\gamma}) - c \left(1+ \max \{t,0\}\right) e^{ -b n^{\delta}}  \notag \\
& \qquad\qquad \leq  U^{\mathfrak f,\theta}_n (t)  \leq U^{\mathfrak f_p,\theta_p}_n (t+ n^{-\gamma}) 
+ c \left(1+ \max \{t,0\}\right) e^{ -b n^{\delta}}.  
\end{align}
\end{proposition}

\begin{proof}
Using the Cauchy-Schwarz inequality, Lemma \ref{Lem-trick}, the bounds \eqref{bound-Snx-fn-001} and \eqref{approx property of theta-001}, 
we get that for any $n\geq 1$ and $p\geq n^\delta$,
\begin{align*} 
&\left| U^{\mathfrak f,\theta}_n (t)-U^{\mathfrak f, \theta_p}_n (t) \right| \\ 
&\leq \int_{\bb X} \bb E \left(\left| t + S^{x}_n + f_n \circ T^n (\omega,x)- f_0(\omega,x) \right| \left| \theta - \theta_p \right|;  \tau^{\mathfrak f}_{x, t} > n \right) \nu(dx) \\
&\leq \left(  \int_{\bb X} \bb E \left( \left| t + S^{x}_n + f_n \circ T^n (\omega,x)- f_0(\omega,x) \right|^2  ;  \tau^{\mathfrak f}_{x, t} > n  \right)  \nu(dx)  \right)^{1/2}
  \bb E^{1/2} \left(\theta - \theta_p \right)^2 \\
&\leq c \left( \max \{t, 0\} + n + B^{1/2} \right) \| \theta - \theta_p \|_{\infty}^{1/2} \; 
\bb E^{1/2} |\theta - \theta_p |  \\
&\leq c  \left( \max \{t, 0\} + n + B \right) B  e^{-\beta n^{\delta}/2}  \\   
&\leq  c' e^{-\beta n^{\delta}/4} \left(1+ \max \{t, 0\}  \right). 
\end{align*}
By applying the same techniques used in the proof of Proposition \ref{Prop g approx 002}, we obtain \eqref{bound UfA 002-001-app}.  
\end{proof}

\subsection{Proof of the quasi-increasing behaviour without twist function} 

In this subsection, we give a proof of inequality \eqref{bound with m for U-105-01-2} 
for the case where the twist function $\theta$ in the definition \eqref{def-U-f-n-t-theta-001}  equals $1$. 
The case with an arbitrary twist function will be addressed in Subsection \ref{subsec:quasi-increasing with twist-001}. 
We begin with the following technical lemma, where we recall that 
$\xi_k$ is an element of the set $ \bb A_p  = \bb G^{\{0,\ldots, p\}}\times \bb X \times \bb N$. 
\begin{lemma}\label{Exponential-Y-tau}
There exist constants $c, \beta >0$ such that for any $1 \leq p \leq  n$, $t \in \bb R$, 
$a=(g_{0},\ldots,g_{p},x, q) \in \bb A_p$  
and any sequence of $\mathscr A_p$-measurable functions $\mathfrak f = (f_n)_{n \geq 0}$
satisfying the moment condition \eqref{exp mom for g 002}, we have 
\begin{align*} 
&\bb{E}_a \Bigg( |\tilde f(\xi_{ \tilde \tau^{\mathfrak f}_{t} })| + \sum_{j=\tilde \tau^{\mathfrak f}_{t} -p+1}^{\tilde \tau^{\mathfrak f}_{t}} | h_p(\xi_j) |;
\  \tilde \tau^{\mathfrak f}_{t} >n \Bigg)  \notag\\
&\quad \leq  
c p \left( \max \{t,0\}  + |\sigma(g_p\cdots g_1,x)| + |\tilde f(a)|    \right) (n-p)^{-\beta} 
+  c p  n^{-2} \sum_{k=p+1}^{\infty} \frac{1}{k^2} \mathcal F_k(a).
\end{align*}
\end{lemma}

\begin{proof}
Let $\beta$ be as in Lemma \ref{Lem-tau-prior} and choose $\gamma > 8/\alpha$. 
Note that, by \eqref{def-Fka}, for any $k > p$ and $a \in \bb A_p$, we have
\begin{align}\label{BOUND XI-002-001}
 \bb E_a |\tilde f(\xi_k)| \leq c \mathcal F_k(a). 
\end{align}
Now we handle the term involving $\tilde f$.
Note that 
\begin{align} \label{HHH-0001-001-OO2}
\bb E_a \left(  |\tilde f(\xi_{\tilde \tau^{\mathfrak f}_{t}})| ;  \tilde \tau^{\mathfrak f}_{t} > n \right)  
& = \sum_{k=n+1}^{\infty}  
\bb E_a \left( |\tilde f(\xi_{\tilde \tau^{\mathfrak f}_{t}})|;   \tilde \tau^{\mathfrak f}_{t} = k \right)  \notag\\
& = \sum_{k=n+1}^{\infty}  \bb E_a \left(|\tilde f(\xi_{k})|;  |\tilde f(\xi_{k})| > \gamma \log k,  \tilde \tau^{\mathfrak f}_{t} = k \right) \notag\\
& \quad + \sum_{k=n+1}^{\infty}   \bb E_a \left( |\tilde f(\xi_{k})|;  |\tilde f(\xi_{k})| \leq \gamma \log k ,  \tilde \tau^{\mathfrak f}_{t} = k \right).
\end{align}
For the first term in \eqref{HHH-0001-001-OO2},  
by Chebyshev's inequality, \eqref{BOUND XI-002-001} and \eqref{def-Fka}, we get
\begin{align*} 
\bb E_a \left(|\tilde f(\xi_{k})|;  |\tilde f(\xi_{k})| > \gamma \log k,  \tilde \tau^{\mathfrak f}_{t} = k \right)
&\leq  \bb E_a \left(|\tilde f(\xi_{k})|;  |\tilde f(\xi_{k})| > \gamma \log k \right) \\
& \leq k^{-\alpha \gamma/2}   \bb E_a \left( |\tilde f(\xi_{k})|  e^{ \frac{\alpha}{2}  |\tilde f(\xi_{k})| } \right) \\ 
& \leq c k^{-\alpha \gamma/2} \bb E_a  e^{ \alpha  |\tilde f(\xi_{k})| }  \\ 
& = c k^{-\alpha \gamma/2} \mathcal F_k(a). 
\end{align*}
As $\gamma > 8/\alpha$, summing over $k$ gives the following bound for the first term on the right-hand side of \eqref{HHH-0001-001-OO2}: 
\begin{align} \label{BOUND XI-002-002}
 \sum_{k=n+1}^{\infty}  
\bb E_a \left(|\tilde f(\xi_{k})|;  |\tilde f(\xi_{k})| > \gamma \log k,  \tilde \tau^{\mathfrak f}_{t} = k \right)
& \leq c n^{2-\alpha \gamma/2} \sum_{k=n+1}^{\infty} \frac{1}{k^2} \mathcal F_k(a) \notag\\
&  \leq c n^{-2} \sum_{k=n+1}^{\infty} \frac{1}{k^2} \mathcal F_k(a).     
\end{align}
For the second term on the right-hand side of \eqref{HHH-0001-001-OO2}, we have 
\begin{align} \label{HHH-0001-002-O2}
\sum_{k=n+1}^{\infty}  \bb E_a \left( |\tilde f(\xi_{k})|;  |\tilde f(\xi_{k})| \leq \gamma \log k ,  \tilde \tau^{\mathfrak f}_{t} = k \right)
& \leq \gamma \sum_{k=n+1}^{\infty}  \log k \ \bb P_a \big( \tilde \tau^{\mathfrak f}_{t} = k \big) \notag\\
&= \gamma \bb E_a \left( \log  \tilde \tau^{\mathfrak f}_{t};  \tilde \tau^{\mathfrak f}_{t} >n   \right).
\end{align}
Recall that, for a random variable $Y$ without atoms and a smooth and integrable function $\phi$ on $\bb R$,
and any $t \in \bb R$,  we have 
\begin{align*}
\bb E \left( \phi (Y) \mathds 1_{ \{ Y > t \} } \right) = \phi(t) \bb P (Y > t) + \int_t^{\infty} \phi'(u) \bb P (X > u) du. 
\end{align*}
Applying this identity gives 
\begin{align*}
\bb E_a \left( \log  \tilde \tau^{\mathfrak f}_{t};  \tilde \tau^{\mathfrak f}_{t} >n   \right) 
= (\log n)  \bb P_a \left( \tilde \tau^{\mathfrak f}_{t}  > n \right) 
 +  \int_{n}^{\infty}  u^{-1}  \bb P_a \left( \tilde \tau^{\mathfrak f}_{t}  > u \right) du. 
\end{align*}
Hence, using Lemma \ref{Lem-tau-prior},  we get
\begin{align}\label{BOUND XI-002-003}
&\bb E_a \left( \log  \tilde \tau^{\mathfrak f}_{t};  \tilde \tau^{\mathfrak f}_{t} >n   \right) \notag\\ 
&\leq 
  (\log n) \bigg(  c \frac{\max \{t,0\} +  |\sigma(g_p\cdots g_1,x)|   +|\tilde{f}(a)|+ \log n}{(n-p)^{\beta}}
 + \frac{c }{n^8} \sum_{k = p+1}^{n} \mathcal F_k(a)   \bigg)  \notag\\ 
& \quad + \int_{n+1}^{\infty}  u^{-1} \bigg(  c \frac{\max \{t,0\} +  |\sigma(g_p\cdots g_1,x)|   +|\tilde{f}(a)|+ \log u}{(u-p)^{\beta}}
 + \frac{c }{u^8} \sum_{k = p+1}^{[u]} \mathcal F_k(a)   \bigg)  du  \notag\\
& \leq  c \left( \max \{t,0\}  + |\sigma(g_p\cdots g_1,x)| + |\tilde f(a)|    \right) (n-p)^{-\beta} +  c  n^{-2} \sum_{k=p+1}^{\infty} \frac{1}{k^2} \mathcal F_k(a).
\end{align}
Combining \eqref{HHH-0001-001-OO2}, \eqref{BOUND XI-002-002}, \eqref{HHH-0001-002-O2} and \eqref{BOUND XI-002-003} gives
\begin{align} \label{HHH-0001-005-OO2}
& \bb E_a \left(  |\tilde f(\xi_{\tilde \tau^{\mathfrak f}_{t}})| ;  \tilde \tau^{\mathfrak f}_{t} > n \right)  \notag \\  
&\leq  c \left( \max \{t,0\}  + |\sigma(g_p\cdots g_1,x)| + |\tilde f(a)|    \right) (n-p)^{-\beta} 
+  c  n^{-2} \sum_{k=p+1}^{\infty} \frac{1}{k^2} \mathcal F_k(a).
\end{align}
The term involving $h_p$ is dominated in a similar manner.
\end{proof}

Recall that, by  \eqref{EXPECT-E_x-002} and \eqref{def-Mn-martingel-001}, 
for $a=(g_{0},\ldots,g_{p},x, q) \in \bb A_p$ and $t \in \bb R$, we have
\begin{align*} 
W^{\mathfrak f}_n(a, t) = \bb E_a \left( t + \sigma(g_{p}\cdots g_{1}, x) + M_n; \tilde \tau^{\mathfrak f}_{t} >n \right), 
\end{align*}
where $M_n=  \sum_{i=p+1}^{n+p} \sigma_p(\xi_i)$.
We apply Lemma \ref{Exponential-Y-tau} to show that the sequence $(W^{\mathfrak f}_n(a,t))_{n \geq 1}$
is quasi-increasing.

\begin{lemma} \label{lem-U is increasing-000}
There exist constants $c, \beta>0$ such that for any  $1 \leq p \leq  n \leq m$, 
 $a=(g_{0},\ldots,g_{p},x,q) \in \bb A_p$, $t \in \bb R$ and
  any sequence of $\mathscr A_p$-measurable functions $\mathfrak f = (f_n)_{n \geq 0}$
satisfying the moment condition \eqref{exp mom for g 002}, 
\begin{align} \label{bound-Wn-at-increase}
W^{\mathfrak f}_n(a, t) 
& \leq  W^{\mathfrak f}_{m}(a,t) 
 + c p (1+|\tilde f(a)|) \left( \max \{t,0\}  + |\sigma(g_p\cdots g_1,x)| + |\tilde f(a)|    \right) (n-p)^{-\beta}  \notag\\ 
&\quad +  c p (1+|\tilde f(a)|)  n^{-2} \sum_{k=p+1}^{\infty} \frac{1}{k^2} \mathcal F_k(a). 
\end{align}
\end{lemma}

\begin{proof}
Let $a=(g_{0},\ldots,g_{p},x,q) \in \bb A_p$ and $t'=t+ \sigma(g_{p} \cdots g_{1}, x).$
In view of Lemma \ref{Lem-tau-prior}, 
to obtain inequality \eqref{bound-Wn-at-increase}, it is sufficient to prove the result with $W^{\mathfrak f}_n(a, t)$ replaced by 
$ \bb{E}_a ( t'+M_{n} - \tilde f(\xi_0);\ \tilde \tau^{\mathfrak f}_{t} > n )$.
 We first show that for any $1 \leq p \leq  n \leq m$, 
 $a \in \bb A_p$ and $t \in \bb R$, 
\begin{align}\label{Equ-desired-O2}
& \bb{E}_a\left( t'+M_{n}- \tilde f(\xi_0);\ \tilde \tau^{\mathfrak f}_{t} > n \right)   
 \leq \bb{E}_a\left( t'+M_{ m}- \tilde f(\xi_0);\ \tilde \tau^{\mathfrak f}_{t} > m\right)   \notag\\
& \qquad\qquad\qquad\qquad +  
\bb{E}_a \Bigg(  \sum_{j=\tilde \tau^{\mathfrak f}_{t} -p+1}^{\tilde \tau^{\mathfrak f}_{t}} h_p(\xi_j) - \tilde f(\xi_{ \tilde \tau^{\mathfrak f}_{t} });
     \  n + 1 \leq \tilde \tau^{\mathfrak f}_{t} \leq  m \Bigg). 
 \end{align}
Indeed, since $\bb E_a (M_{m}) = \bb E_a (M_{ n})$, it holds that 
\begin{align*}
& \bb{E}_a \left( t'+M_{m}- \tilde f(\xi_0);\ \tilde \tau^{\mathfrak f}_{t} \leq m \right) 
-  \bb{E}_a \left( t'+M_{n}- \tilde f(\xi_0);\ \tilde \tau^{\mathfrak f}_{t} \leq n \right)   \notag\\
& =   \bb{E}_a \left( t'+M_{n}- \tilde f(\xi_0);\ \tilde \tau^{\mathfrak f}_{t} > n \right)
  - \bb{E}_a\left( t'+M_{ m}- \tilde f(\xi_0);\ \tilde \tau^{\mathfrak f}_{t} > m  \right), 
\end{align*}
so that  \eqref{Equ-desired-O2} is equivalent to the following inequality:
\begin{align}\label{Equ-00a-O2}
& \bb{E}_a\left( t'+M_{m}- \tilde f(\xi_0);\ \tilde \tau^{\mathfrak f}_{t} \leq m \right)  
 \leq  \bb{E}_a\left( t'+M_{n}- \tilde f(\xi_0);\ \tilde \tau^{\mathfrak f}_{t} \leq n \right)  \notag\\
& \qquad\qquad\qquad\qquad  
 +  \bb{E}_a \Bigg(  \sum_{j=\tilde \tau^{\mathfrak f}_{t} -p+1}^{\tilde \tau^{\mathfrak f}_{t}} h_p(\xi_j) - \tilde f(\xi_{ \tilde \tau^{\mathfrak f}_{t} });\  
 n + 1 \leq \tilde \tau^{\mathfrak f}_{t} \leq m \Bigg). 
\end{align}
We shall prove \eqref{Equ-00a-O2} using an induction argument. 
Let $k \in [n + 1, m]$ and we write
\begin{align}\label{E-Mh-tau-ft-001}
& \bb{E}_a \left( t' + M_k - \tilde f(\xi_0);\ \tilde \tau^{\mathfrak f}_{t} \leq k \right)   \notag\\
& =  \bb{E}_a \left( t' + M_k - \tilde f(\xi_0);\ \tilde \tau^{\mathfrak f}_{t} \leq k - 1 \right)
   + \bb{E}_a \left( t' + M_k - \tilde f(\xi_0);\ \tilde \tau^{\mathfrak f}_{t} = k \right). 
\end{align}
As the event $\{ \tilde \tau^{\mathfrak f}_{t} \leq k -1 \}$ is in $\mathscr G_{k-1}$,
by the martingale property, we get
\begin{align}\label{E-Mh-tau-ft-002}
\bb{E}_a \left( t' + M_k - \tilde f(\xi_0);\ \tilde \tau^{\mathfrak f}_{t} \leq k -1 \right)
= \bb{E}_a \left( t' + M_{k -1} - \tilde f(\xi_0);\ \tilde \tau^{\mathfrak f}_{t} \leq k -1\right). 
\end{align}
Besides, 
by the definition of $\tilde \tau^{\mathfrak f}_{t}$ and $M_n$ (cf.\ \eqref{def-tau-f-y} and \eqref{def-Mn-martingel-001}),
on the set $\{\tilde \tau^{\mathfrak f}_{t} = k \}$, we have $t' + M_{k-p} + \tilde f(\xi_k) - \tilde f(\xi_0)<0$, which leads to
\begin{align}\label{E-Mh-tau-ft-003}
\bb{E}_a \left( t' + M_k - \tilde f(\xi_0);\ \tilde \tau^{\mathfrak f}_{t} = k \right)
\leq \bb{E}_a \Bigg( \sum_{j=k-p+1}^k h_p(\xi_j) - \tilde f(\xi_k)  ;\ \tilde \tau^{\mathfrak f}_{t} = k \Bigg).
\end{align} 
Hence, combining \eqref{E-Mh-tau-ft-001}, \eqref{E-Mh-tau-ft-002} and \eqref{E-Mh-tau-ft-003} gives that for any $k \in [n+1, m]$,
\begin{align*}
& \bb{E}_a \left( t' + M_k - \tilde f(\xi_0);\ \tilde \tau^{\mathfrak f}_{t} \leq k \right)   \notag\\
& \leq \bb{E}_a \left( t'+M_{k -1}-\tilde f(\xi_0);\ \tilde \tau^{\mathfrak f}_{t} \leq k -1 \right) 
   +  \bb{E}_a \Bigg( \sum_{j=k-p+1}^k h_p(\xi_j) - \tilde f(\xi_k) ;\ \tilde \tau^{\mathfrak f}_{t} = k \Bigg).
\end{align*}
Summing over $k \in [n+1, m]$, we get
\begin{align*}
& \bb{E}_a \left( t'+M_{m}- \tilde f(\xi_0);\ \tilde \tau^{\mathfrak f}_{t} \leq m \right)   \notag\\
& \leq  \bb{E}_a \left( t'+M_{n}- \tilde f(\xi_0);\ \tilde \tau^{\mathfrak f}_{t} \leq n \right)
+  \sum_{k = n+1}^{ m }   
\bb{E}_a \Bigg(  \sum_{j = k - p + 1}^k h_p(\xi_j) - \tilde f(\xi_k) ;\ \tilde \tau^{\mathfrak f}_{t} = k \Bigg)  \notag\\
& =  \bb{E}_a \left( t'+M_{n}- \tilde f(\xi_0);\ \tilde \tau^{\mathfrak f}_{t} \leq n \right)  
 +   \bb{E}_a \Bigg( \sum_{j=\tilde \tau^{\mathfrak f}_{t}-p+1}^{\tilde \tau^{\mathfrak f}_{t}} h_p(\xi_j) - \tilde f(\xi_{ \tilde \tau^{\mathfrak f}_{t} });
             \  n + 1 \leq \tilde \tau^{\mathfrak f}_{t} \leq  m  \Bigg), 
\end{align*}
thus proving \eqref{Equ-00a-O2}. 
Since \eqref{Equ-desired-O2} is equivalent to \eqref{Equ-00a-O2},
the conclusion of the lemma follows by Lemma \ref{Exponential-Y-tau}. 
\end{proof}

By integration over $(g_{0},\ldots,g_{p})$ such that $a=(g_{0},\ldots,g_{p},x,q) \in \bb A_p$ in Lemma \ref{lem-U is increasing-000},
 we obtain the following bound for $U^{\mathfrak f}_n(t)$. 

\begin{lemma} \label{lem-U is increasing-1001}
There exist constants $\ee, \beta, c>0$ such that for any $1\leq n\leq m$, $1 \leq p \leq  n^{\ee}$, 
 $t \in \bb R$ and any sequence of $\mathscr A_p$-measurable functions $\mathfrak f = (f_n)_{n \geq 0}$
satisfying the moment condition \eqref{exp mom for g 002}, we have 
\begin{align*}
U^{\mathfrak f}_n(t) \leq  U^{\mathfrak f}_{m}(t) 
 + c \left( 1 + \max \{t,0\}  \right) n^{-\beta} C_{\alpha}(\mathfrak f), 
\end{align*}
where $\alpha$ is the exponent from the moment condition \eqref{exp mom for g 002}. 
\end{lemma}

\begin{proof}
This is obtained from Lemma \ref{lem-U is increasing-000} together with \eqref{MAIN_GOAL-002}, \eqref{Expect-E_x001} and 
Corollary \ref{Lem-bound-tau-p}.
Recall that $\bb E \sigma(g_{p} \cdots g_{1}, x)^2 \leq c p$ by the moment assumption \eqref{exp mom for f 001} 
and the centering assumption \eqref{centering-001}.  
Since $a=(g_0,\ldots,g_{p},x,0) \in \bb A_p$, 
integrating both sides of \eqref{bound-Wn-at-increase}
 in Lemma \ref{lem-U is increasing-000} (with $\alpha/4$ instead of $\alpha$) with respect to the measure $\mu(dg_0) \ldots \mu(dg_{p})  \nu(dx)$ and applying  \eqref{Expect-E_x001} 
yields that for any $t \in \bb R$, $n\geq 1$ and $1\leq p \leq n^{\ee/2}$, 
\begin{align} \label{Expect-E_x001-002-a}  
\int_{\bb X} \bb  E \left( t+S_{n+p}^x ; \tau^{\mathfrak f}_{x, t}>n \right) \nu(dx) 
& \leq  
   \int_{\bb X} \bb E \left( t+S_{m+p}^x ; \tau^{\mathfrak f}_{x, t}> m \right) \nu(dx)  \notag  \\
& \quad + \frac{c_{\ee}}{n^{\ee/2}}    \left(1+ \max \{t,0\} \right)  C_{\alpha}(\mathfrak f),
\end{align}
where we have used the fact that $\nu$ is $\mu$-stationary, as in \eqref{stationary-nu-001}. 
Using Lemma \ref{Lem-bound-tau-p}, we can slightly modify the expectations in both sides of \eqref{Expect-E_x001-002-a},
leading to 
\begin{align*}
U_n^{\mathfrak f}(t)
&= \int_{\bb X} \bb E \left( t+S_{n}^x + f_n( T^n (\omega,x) ) - f_0(\omega,x) ; \tau^{\mathfrak f}_{x, t}>n \right) \nu(dx) \notag  \\
& \leq    U_{m}^{\mathfrak f}(t) + \frac{c_{\ee}}{n^{\ee/2}}    \left(1+ \max \{t,0\} \right)  C_{\alpha}(\mathfrak f).
\end{align*}
This completes the proof of Lemma \ref{lem-U is increasing-1001}.
\end{proof}

We conclude this subsection by proving the following quasi-increasing behaviour of the sequence $(U^{\mathfrak f}_n)_{n\geq 1}$, 
even when the sequence $\mathfrak f = (f_n)_{n \geq 0}$ is not necessarily $\mathscr A_p$-measurable.

\begin{proposition} \label{lem-U is increasing-1002}
Suppose that the cocycle $\sigma$ admits finite exponential moments \eqref{exp mom for f 001}
and is centered \eqref{centering-001}.  
We also suppose that the effective central limit theorem \eqref{BEmart-001} is satisfied. 
Assume that $\mathfrak f = (f_n)_{n \geq 0}$ is a sequence of  measurable functions on $\Omega \times \bb X$ 
satisfying the moment condition \eqref{exp mom for g 002} and the approximation property \eqref{approxim rate for gp-002}.
Then, there exists a constant $\ee>0$ with the following property:  
for any $\gamma >0$, 
there exists a constant $c>0$ such that, for any $1 \leq n\leq m$ and $t \in \bb R$,  
\begin{align*}
U^{\mathfrak f}_n(t) \leq  U^{\mathfrak f}_{m}(t + c n^{-\gamma}) 
 + c n^{-\ee} \left( \max \{t,0\}  + C_{\alpha}(\mathfrak f)  \right)   C_{\alpha}(\mathfrak f) D_{\alpha,\beta}(\mathfrak f), 
\end{align*}
where $\alpha$ is from the moment condition \eqref{exp mom for g 002}. 
\end{proposition}

\begin{proof}
Fix $k$ to be a large integer whose value will be determined later.   
Let $h$ be the least integer with the property that $h^{k} > n$ and
let $l$ be the least integer with the property that $l^{k} > m$.
We set $m_{h-1}=n$, $m_i = i^{k}$ for $h \leq i \leq l-1$ and  $m_l = m$. 
By Lemma \ref{lem-U is increasing-1001}, there exist constants $\ee, \beta_0, c>0$ such that, whenever $p\leq m_{i-1}^{\ee}$, 
the following inequality holds: 
\begin{align*}
U^{\mathfrak f_{p}}_{m_{i-1}}(t) 
\leq  U^{\mathfrak f_{p}}_{m_{i}}(t) 
 + c \left( 1 + \max \{t,0\}  \right) m_{i-1}^{-\beta_0} C_{\alpha}(\mathfrak f), 
\end{align*}
where $\alpha$ is from the moment condition \eqref{exp mom for g 002}. 
In particular, setting $p_i = [m_{i-1}^{\ee}]$ for $h\leq i\leq l$, we get 
\begin{align} \label{-thbound with m for U-200-001}
U^{\mathfrak f_{p_{i}}}_{m_{i-1}}(t) 
\leq  U^{\mathfrak f_{p_i}}_{m_{i}}(t) 
 + c \left( 1 + \max \{t,0\}  \right) m_{i-1}^{-\beta_0} C_{\alpha}(\mathfrak f). 
\end{align}
The approximation property \eqref{approxim rate for gp-002} implies that there exists $\beta_1$ such that $D_{\alpha,\beta_1}(\mathfrak f) <\infty.$ 
Since the approximation property
also holds with $\beta=\min\{\beta_0,\beta_1\}$,  applying Proposition \ref{Prop g approx 002} 
 with $\delta=\ee$, there exists a constant $c>0$ 
 such that, for any  $ h\leq i \leq l$ and any $\gamma>0$,
\begin{align}\label{BBBB002-001}
 U^{\mathfrak f}_{m_{i-1}} (t) 
 \leq U^{\mathfrak f_{p_i}}_{m_{i-1}} (t+ m_{i-1}^{-\gamma}) 
+ c \left(\max \{t,0\} + C_{\alpha}(\mathfrak f) \right) e^{ -\beta  m_{i-1}^{\ee}} D_{\alpha,\beta}(\mathfrak f) 
\end{align}
and 
\begin{align}\label{BBBB002-002}
U^{\mathfrak f_{p_i}}_{m_{i}} (t) 
\leq  U^{\mathfrak f}_{m_{i}} (t+ m_{i}^{-\gamma}) 
+ c \left(\max \{t,0\} + C_{\alpha}(\mathfrak f) \right) e^{ -\beta m_{i}^{\ee}} D_{\alpha,\beta}(\mathfrak f),
\end{align}
Combining inequalities \eqref{-thbound with m for U-200-001}, \eqref{BBBB002-001} and \eqref{BBBB002-002}, 
and noting that  $m_{i} \geq m_{i-1}$, we get
\begin{align*}
 U^{\mathfrak f}_{m_{i-1}} (t) 
 \leq U^{\mathfrak f}_{m_{i}} (t+ 2 m_{i-1}^{-\gamma}) 
 + c  \left(e^{ -\beta  m_{i-1}^{\ee}} + m_{i-1}^{-\beta} \right) 
\left( \max \{t,0\} + C_{\alpha}(\mathfrak f)\right) 
C_{\alpha}(\mathfrak f) D_{\alpha,\beta}(\mathfrak f).  
\end{align*}
By consecutively applying these inequalities and using the definition of $m_i$, we obtain
\begin{align*}
 U^{\mathfrak f}_{n} (t) 
 \leq  U^{\mathfrak f}_{m} \bigg(t+ 2 \sum_{i=h-1}^{l-1} i^{-\gamma k} \bigg) 
+ c \left(\max \{t,0\} + C_{\alpha}(\mathfrak f) \right)  
\sum_{i=h-1}^{l-1}\left(e^{ -\beta  i^{\ee k}} + i^{-\beta k}\right) C_{\alpha}(\mathfrak f) D_{\alpha,\beta}(\mathfrak f).  
\end{align*}
Finally, letting $\gamma>0$ be arbitrary  
and choosing  $k>\max \{ \gamma^{-1}, \beta^{-1}\}$, we obtain 
\begin{align*}
 U^{\mathfrak f}_{n} (t) 
\leq  U^{\mathfrak f}_{m} \left(t+ c n^{- \gamma + k^{-1}} \right) + c n^{-\beta + k^{-1}}  \left( \max \{t,0\} + C_{\alpha}(\mathfrak f) \right) 
 C_{\alpha}(\mathfrak f) D_{\alpha,\beta}(\mathfrak f). 
\end{align*}
The assertion of the proposition follows.  
\end{proof}

\subsection{Proof of the quasi-increasing behaviour with twist function} \label{subsec:quasi-increasing with twist-001}

We actually need a version of Proposition \ref{lem-U is increasing-1002} in the presence of a non-negative twist function
$\theta \in L^{\infty}(\Omega, \bb P)$. 
We assume that $\theta$ satisfies the property \eqref{approx property of theta-001}, 
and we have denoted $\|\theta\|_{\infty} = {\rm esssup} |\theta|$. 
Recall that, for $t \in \bb R$ and $n \geq 1$, 
\begin{align*}
U^{\mathfrak f, \theta}_n(t)
= \int_{\bb X} \bb E \left((t + S^{x}_n + f_n(T^n (\omega,x))- f_0(\omega,x))  \theta(\omega) ;  \tau^{\mathfrak f}_{x,t} > n \right) \nu(dx). 
\end{align*}

We now prove the following lemma, which is a version of Lemma \ref{lem-U is increasing-1001} incorporating the twist function.
\begin{lemma} \label{lem-U is increasing-2001}
There exist constants $\ee, \beta, c>0$ such that for any $1\leq n\leq m$, $1\leq p \leq  n^{\ee}$, 
 $t \in \bb R$, any sequence $\mathfrak f = (f_n)_{n \geq 0}$ of $\mathscr A_p$-measurable functions $f_n:\Omega\times \bb X \to \bb R$ satisfying 
 the moment condition \eqref{exp mom for g 002}, 
 and any non-negative  $\mathscr A_p$-measurable function  $\theta\in L^{\infty}(\Omega, \bb P)$, we have
 \begin{align*}
U^{\mathfrak f, \theta}_n(t) \leq  U^{\mathfrak f, \theta}_{m}(t) 
 + c n^{-\beta}  \left( 1 + \max \{t,0\}  \right) \|\theta\|_{\infty} C_{\alpha}(\mathfrak f). 
\end{align*}
\end{lemma}
\begin{proof}
We write the following extension of formula \eqref{Expect-E_x001}: for any $x \in \bb X$ and $t \in \bb R$,
\begin{align}\label{integral-formula-app}
 &\int_{\bb X} \bb E \left( \left(t + S^{x}_{n+p}\right) \theta(\omega) ;  \tau^{\mathfrak f}_{x, t} > n \right) \nu(dx) \notag\\
 & = \int_{\bb X}  \int _{\bb G^{\{0,\ldots, p\}} }  W^{\mathfrak f}_{n} ((g_{0},\ldots,g_{p}, x, 0), t) \theta(g_1,\ldots,g_p)  \mu(dg_{0}) \ldots \mu(dg_{p}) \nu(dx), 
\end{align}
where $W^{\mathfrak f}_{n}$ is defined by \eqref{EXPECT-E_x-001}.
Then Lemma \ref{lem-U is increasing-2001} is obtained from Lemma \ref{lem-U is increasing-000} in the same way as
 Lemma \ref{lem-U is increasing-1001}.
\end{proof}

We now proceed with an extension of Lemma \ref{lem-U is increasing-2001} 
to sequences $\mathfrak f = (f_n)_{n \geq 0}$ that are not necessarily $\mathscr A_p$-measurable.

\begin{proposition} \label{lem-UfA is increasing-1002-app}
Suppose that the cocycle $\sigma$ admits finite exponential moments \eqref{exp mom for f 001}
and is centered \eqref{centering-001}.  
Then, for any $\alpha,\beta >0$ and $B \geq 1$, there exist constants $\ee, \gamma, c > 0$ with the following property. 
Assume that $\mathfrak f = (f_n)_{n \geq 0}$ is a sequence of  measurable functions on $\Omega \times \bb X$ 
satisfying the moment condition \eqref{exp mom for g 002} and the approximation property \eqref{approxim rate for gp-002}
with $C_{\alpha}(\mathfrak f) \leq B$ and $D_{\alpha,\beta}(\mathfrak f) \leq B$. 
Assume that $\theta \in L^{\infty}(\Omega, \bb P)$ with $\|\theta\|_{\infty} \leq B$ and $N_{\beta}(\theta) \leq B$. 
Then, for any  $1\leq n\leq m$ and $t \in \bb R$, we have
\begin{align*}
U^{\mathfrak f, \theta}_n(t) \leq  U^{\mathfrak f, \theta}_{m}(t + c n^{-\gamma}) 
 + c n^{-\ee} \left( 1 + \max \{t,0\}   \right). 
\end{align*}
\end{proposition}

\begin{proof}[Proof of the first part of Theorem \ref{Pro-Appendix-Final2-Inequ}]
Inequality \eqref{bound with m for U-105-01-2} is the conclusion of Proposition \ref{lem-UfA is increasing-1002-app}. 
\end{proof}


\section{Quasi-decreasing behaviour of the sequence $(U_n^{\mathfrak f})_{n\geq 1}$}\label{sec-quasi-decreasing property}

In this concluding section, we derive inequality \eqref{bound with m for U-105-01-3} of Theorem \ref{Pro-Appendix-Final2-Inequ}.
The proof of this second bound 
is more intricate than that of the bound \eqref{bound with m for U-105-01-2}, established in Section \ref{sec quasi-increasingness}.
While the proof of  \eqref{bound with m for U-105-01-2} relied primarily on the straightforward  
submartingale property of the martingale $(M_n)_{n\geq 0}$ killed at the stopping time $\tau_{x,t}$,
the proof of the converse bound \eqref{bound with m for U-105-01-3} requires a more subtle approach. 
This approach is built upon the proposition stated below, which can be seen as an extension of Proposition \ref{Prop-Vn-001}.

\begin{proposition} \label{AA-Prop iterative 002-app}
Suppose that the cocycle $\sigma$ admits finite exponential moments \eqref{exp mom for f 001}
and is centered \eqref{centering-001}. 
We also suppose that the effective central limit theorem \eqref{BEmart-001} is satisfied.  
Then, for any $\beta > 0$ and $B \geq 1$, there exist constants $\ee, c>0$ with the following property. 
Assume that $\mathfrak f = (f_n)_{n\geq0}$ is a sequence of  measurable functions on $\Omega \times \bb X$ 
satisfying the moment condition \eqref{exp mom for g 002} and the approximation property \eqref{approxim rate for gp-002}
with $C_{\alpha}(\mathfrak f) \leq B$ and $D_{\alpha,\beta}(\mathfrak f) \leq B$. 
Assume that $\theta \in L^{\infty}(\Omega, \bb P)$ is non-negative with $\|\theta\|_{\infty} \leq B$ and $N_{\beta}(\theta) \leq B$. 
We have that, for any $n \geq 1$ and $t \in \bb R,$
\begin{align}\label{AA-bound-M_n-002-app}
U^{\mathfrak f, \theta}_{n} (t) 
  \leq \left( 1+\frac{c }{n^{\ee }}\right)  
   U^{\mathfrak f, \theta}_{[n^{1-\ee }] } (t+ 2 n^{-\ee})    + c n^{-\ee/2} ( 1 + \max \{t,0\}).  
\end{align}
\end{proposition}

The proof of this proposition will be presented at the end of this section. 
We will employ a methodology based on the finite-size approximation of perturbations. 
Specifically, we use the approximation sequence $\mathfrak f_p=(f_{n,p})_{n\geq0}$, as defined in \eqref{def-approxi-fnp}, to substitute for the original sequence $\mathfrak f=(f_n)_{n\geq0}$. This substitution enabled the introduction of the Markov chain $(\xi_i)_{i \geq 0}$ 
in Subsection \ref{Subsec-Markov-chian}. 
We handle this Markov chain using an approach akin to the one 
employed for the Markov chain on the space $\bb X$ in the proof of Proposition \ref{Prop-Vn-001}.
 
\subsection{Properties of the random walk near the level $n^{1/2}$} \label{sec-control-high-level}
In this subsection, let $p\geq 1$ be an integer and let $\mathfrak f=(f_{n})_{n\geq0}$ be a sequence of measurable functions
$f_{n}: \bb G^{\{1,\ldots,p \}}\times \bb X \to \bb R$. 
We retain the notation from Section \ref{Subsec-Markov-chian}.
In particular, by \eqref{new representation for S^x_n 001}, we have denoted
\begin{align*}
S^x_n := \sigma (g_n \cdots g_1,x) =  \sum_{i=1}^{n} \sigma_p(\xi_i).  
\end{align*}
For $\ee >0$, $n \geq 1$, $x\in \bb X$ and $t\in \bb R$, 
on the measurable space $\Omega= \bb G^{\bb N^*}$,
consider the exit time (introduced in \eqref{nu n}) 
\begin{align*}
\nu_{n,x,t}=\min \left\{ k\geq 1: \left|t + S_{k}^x  \right|  \geq 2 n^{1/2-\ee }\right\}.  
\end{align*}
For any $a \in \bb A_p$, $n \geq 1$ and $t\in \bb R$, on $\Omega'_p =\bb A_p^{\bb N}$, 
introduce the following exit time 
\begin{align} \label{nu n02}
\tilde \nu_{n,t}=\min \left\{ k\geq p+1: \left| t  + \sum_{i=1}^k \sigma_p(\xi_i) -\tilde f(\xi_0)  \right|  \geq 2 n^{1/2-\ee }\right\},   
\end{align}
where $\tilde f$ is defined by \eqref{function f tilde001}. 
It is clear that $\tilde \nu_{n,t}$ is a $(\mathscr G_n)_{n\geq 0}$-stopping time on $\Omega'_p =\bb A_p^{\bb N}$.
The two exit times $\nu_{n,x,t}$ and $\tilde \nu_{n,t}$ are related as follows.

\begin{remark}\label{Remark-001}
For any $a=(g_0,\ldots,g_{p},x, q) \in \bb A_p$ and $t\in \bb R$, the law of $\tilde \nu_{n,t}-p$ with respect to $\bb P_a$ 
is equal to the law of $\nu_{n,x',t'}$ with respect to $\bb P$. Specifically, 
for any $k \geq 0$,
\begin{align*}
\bb P_a ( \tilde \nu_{n,t} = k + p ) =  \bb P ( \nu_{n,x',t'} = k ),
\end{align*}
where $x' = g_p \cdots g_1 x$ and $t' = t + \sigma(g_p \cdots g_1, x) - f_q(g_1, \ldots, g_p, x)$, 
and we use the fact that 
$\tilde f(\xi_0)= \tilde f(a) = f_q(g_1, \ldots, g_p, x)$ almost surely under the probability measure $\bb P_a$.
\end{remark}

Below, we compile a set of lemmas, all related to the stopping time $\tilde \nu_{n,t}$. 
These lemmas will assist in establishing the important monotonicity bound detailed in the subsequent subsection.

\begin{lemma} \label{Lemma 2-app}
There exists a constant $\beta >0$ 
such that for any $\ee \in (0,\frac{1}{2})$,  $n\geq 1$, $p\geq 1$, $\ell \geq 1$,  
$a \in \bb A_p$ and $t\in \bb R$, 
\begin{align*}
\mathbb{P}_a\left( \tilde \nu_{n,t}> \ell \right) 
\leq  2 \exp \left(- \frac{\beta (\ell -p)}{n^{1-2\ee}} \right).
\end{align*}
In particular, when $0\leq p \leq n^{1-2\ee}$, 
for any $a \in \bb A_p$ and $t \in \bb R$, we have
\begin{align*}
\mathbb{P}_a\left( \tilde \nu_{n,t}>n^{1-\ee }\right) 
\leq  2 \exp \left(- \beta n^{\ee}\right).
\end{align*}
\end{lemma}

\begin{proof}
Using the same notation as in Remark  \ref{Remark-001}, we have  
\begin{align*}
\mathbb{P}_a\left( \tilde \nu_{n,t}>\ell \right) =  \mathbb{P} \left( \nu_{n,x',t'}> \ell-p\right) 
= \mathbb{P} \bigg( \sup_{ 1\leq k \leq \ell - p} |t' + S_k^{x'} |  \leq  2 n^{1/2-\ee  }  \bigg).
\end{align*}
By Lemma \ref{lemma-sup  S_k -001}, there exists a constant $\beta>0$ such that, for any $x'\in \bb X$ and $t'\in \bb R$,
\begin{align*}
\mathbb{P} \bigg( \sup_{ 1\leq k \leq \ell -p} |t' + S_k^{x'} |  \leq  2 n^{1/2-\ee  }  \bigg)
\leq 2 \exp \left(-\beta \frac{\ell- p}{4 n^{1-2\ee}}\right). 
\end{align*}
The assertion follows.
\end{proof}

\begin{lemma}\label{Lem-nu-less-n}
For any $\ee \in (0,\frac{1}{4})$, there exist constants $c, c_\epsilon >0$ 
such that for any $p\geq 1$, $a= (g_0,\ldots, g_{p},x, q) \in \bb A_p$ and $t\in \bb R$
satisfying $t + \sigma(g_{p} \cdots g_{1}, x) - \tilde f(a) \geq -n^{\ee}$,  
\begin{align*}
\mathbb{P}_a \left( \tilde \nu_{n,t} \leq n^{1/2-\ee } -  t - \sigma(g_{p} \cdots g_{1},x)  + \tilde f(a) - p \right) 
\leq c \exp \left( -c_{\ee }n^{\ee/2}   \right).
\end{align*}
\end{lemma}

\begin{proof}
By Remark \ref{Remark-001}, it holds that, 
with $x'=g_p\cdots g_1 x$ and $t' = t + \sigma(g_p\cdots g_1, x)  - \tilde f(a)$, 
\begin{align*}
\mathbb{P}_a \left( \tilde \nu_{n,t} \leq n^{1/2-\ee } - t' - p \right) 
=  \mathbb{P} \left( \nu_{n,x',t'} \leq n^{1/2-\ee } - t' \right). 
\end{align*}
Now the assertion follows from Lemma \ref{Lem-nu-less-n-0}. 
\end{proof}

\begin{lemma} \label{Lemma-nu-tau}
There exist constants $c, \beta, \ee > 0$  such that, 
for any $a=(g_0,\ldots, g_{p}, x, q) \in \bb A_p$, $t \in \bb R$, $n\geq 1$ and $p\leq n^{\ee} $,  
\begin{align*}
\bb P_a \left( \tilde \tau^{\mathfrak f}_{t} > \tilde \nu_{n,t}   \right) 
\leq 
\frac{c}{n^{\beta }} \bigg( \max \{t,0\} + | \sigma(g_p\cdots g_1, x)|  + |\tilde f(a)| + \frac{1}{n} \sum_{k = p+1}^n \mathcal F_k(a) \bigg). 
\end{align*}
\end{lemma}

\begin{proof}
Let $x'=g_p\cdots g_1 x$ and $t' = t +  \sigma(g_p\cdots g_1, x)$, where 
$a=(g_0,\ldots, g_{p}, x,q) \in \bb A_p$ and $t \in \bb R$. 
We also fix $\ee>0$, with its value to be determined later.

Let us first assume that $t' - \tilde f(a) \geq -n^{\ee}$, so that we are able to apply Lemma \ref{Lem-nu-less-n}.
In this case, since the intersection 
$\{ \tilde \nu_{n,t} > n^{1/2-\ee } - t' + \tilde f(a) \} \cap \{ n^{1/2-\ee } - t' + \tilde f(a) > \tilde \tau^{\mathfrak f}_{t} \}$
is contained in the set $\{ \tilde \tau^{\mathfrak f}_{t} \leq \tilde \nu_{n,t} \}$, 
we have 
\begin{align}\label{a001}
\bb P_a( \tilde \tau^{\mathfrak f}_{t} > \tilde \nu_{n,t}  ) 
\leq \bb P_a ( \tilde \nu_{n,t} \leq n^{1/2-\ee } - t' + \tilde f(a) ) + \bb P_a ( n^{1/2-\ee } - t' + \tilde f(a) \leq \tilde \tau^{\mathfrak f}_{t}). 
\end{align}
For the first term, using Remark \ref{Remark-001} and Lemma \ref{Lem-nu-less-n}, we get
\begin{align}\label{a002}
\bb P_a ( \tilde \nu_{n,t} \leq n^{1/2-\ee } - t' + \tilde f(a) )
= \bb P ( \nu_{n,x',t'} \leq n^{1/2-\ee } - t' + \tilde f(a) - p )
\leq c \exp \left( -c_{\ee }n^{\ee/2}   \right). 
\end{align}
Now we handle the second term in \eqref{a001}. Assume that $t' - \tilde f(a)\in [-n^{\ee},  \frac{1}{2} n^{1/2-\ee}]$. 
Then, by Lemma \ref{Lem-tau-prior}, there  exist constants $\beta>0$ and $c>0$ such that, 
for $p\leq \frac{1}{4}n^{1/2-\ee}$,
\begin{align*}
&\bb P_a ( n^{1/2-\ee } - t' + \tilde f(a) \leq \tilde \tau^{\mathfrak f}_{t})
\leq \bb P_a \bigg( \tilde \tau^{\mathfrak f}_{t} \geq \frac{1}{2} n^{1/2-\ee }  \bigg)  \notag\\
&\leq c \frac{\max \{t,0\} +  |\sigma(g_p\cdots g_1,x)|  +|\tilde f(a)| + \log( \frac{1}{2} n^{1/2-\ee }) }{(\frac{1}{2} n^{1/2-\ee }-p)^{\beta}}
+  \frac{c }{n^{4 - 8\ee}} \sum_{k = p+1}^n \mathcal F_k(a)  \notag\\
&\leq  
\frac{c \log n}{\left(n^{1/2-\ee }-2p\right)^{\beta}}  \left( \max \{t,0\} +  |\sigma(g_p\cdots g_1,x)| +|\tilde f(a)| \right) 
+  \frac{c }{n^{4 - 8\ee}} \sum_{k = p+1}^n \mathcal F_k(a) \notag\\
&\leq  
\frac{c}{n^{\beta/4 }}  \left( \max \{t,0\} + |\sigma(g_p\cdots g_1,x)| +|\tilde f(a)| \right) 
+  \frac{c }{n^{4 - 8\ee}} \sum_{k = p+1}^n \mathcal F_k(a).  
\end{align*}
If $t' - \tilde f(a) > \frac{1}{2} n^{1/2-\ee}$, then we have 
\begin{align}\label{a003}
\bb P_a ( n^{1/2-\ee } - t' +  \tilde f(a) \leq \tilde \tau^{\mathfrak f}_{t}) \leq 1 
\leq 2 \frac{ t' - \tilde f(a) }{n^{1/2-\ee}} \leq 2 \frac{ \max \{t,0\} + |\sigma(g_p\cdots g_1,x)|  +|\tilde f(a)| }{n^{1/2-\ee}}. 
\end{align}
Therefore, in the case $t' - \tilde f(a) \geq -n^{\ee}$, the assertion follows from 
\eqref{a001}, \eqref{a002} and \eqref{a003} by taking $\ee$ small enough.

It remains to deal with the case $t' - \tilde f(a) < -n^{\ee}$.
In this case, we set 
\begin{align*}
A = \left\{ \xi\in \Omega_p : |\tilde f(\xi_{p+1})| \leq \frac{n^{\ee}}{2}, \  |\sigma_p\left(\xi_{p+1} \right)| \leq \frac{n^{\ee}}{2}
\right\}.
\end{align*}
In view of the definition \eqref{trans-prob-xi} of the transition kernel $\{ P_a: a \in \bb A_p \}$, 
we have, for $p\leq n^{\ee/2}$,
\begin{align} \label{aboundAc001}
\bb P_a\left(A^c\right) 
&\leq \bb P_a \left( |\tilde f(\xi_{p+1})| > \frac{n^{\ee}}{2} \right) 
+  \sup_{x\in \bb X} \bb P \left(|\sigma (g_1,x)| \geq \frac{n^{\ee}}{2} \right) \notag\\
&\leq  e^{- \alpha \frac{n^{\ee}}{2}} \mathcal F_{p+1}(a) + c e^{- \alpha \frac{n^{\ee}}{2}}
\leq c' e^{- \alpha \frac{n^{\ee}}{2}} \mathcal F_{p+1}(a), 
\end{align}
where we have used \eqref{exp mom for f 001}.
On the set $A$, it holds, $\bb P_a$-almost surely, 
\begin{align*} 
t +\sum_{i=1}^{p+1}  \sigma_p (\xi_i) + \tilde f(\xi_{p+1})  - \tilde f(\xi_0)
= t' -  \tilde f(a)+ \sigma_p (\xi_{p+1}) + \tilde f(\xi_{p+1}) 
<  -n^{\ee} + n^{\ee} =0. 
\end{align*}
Hence, by the definitions \eqref{def-tau-f-y} and \eqref{nu n02} of the stopping times $\tilde \tau_t^{\mathfrak f}$ and $\tilde\nu_{n,t}$,  
it holds that $\tilde \tau_t^{\mathfrak f} \leq p+1 \leq \tilde\nu_{n,t}$, $\bb P_a$-almost surely on the set $A$. 
Therefore, when $t' - \tilde f(a) < -n^{\ee}$, we obtain
\begin{align*}
\bb P_a \left( \tilde\nu_{n,t} <  \tilde \tau_t^{\mathfrak f}   \right) \leq \bb P_a \left( A^c  \right), 
\end{align*}
and the conclusion follows from \eqref{aboundAc001}.
\end{proof}

\begin{lemma} \label{Lemma martingale-001}
There exist constants $c, \beta, \ee > 0$  such that, 
for any $a=(g_0,\ldots, g_{p}, x, q) \in \bb A_p$, $t \in \bb R$, $n\geq 1$ and $p\leq n^{\ee} $,  
\begin{align}\label{sum-hp-nu-nt-001}
&\bb E_a \bigg( \bigg| \sum_{i=\tilde \nu_{n,t}-p+1}^{\tilde \nu_{n,t}} h_p(\xi_i) \bigg|; \; 
 \tilde \tau^{\mathfrak f}_{t} >\tilde \nu_{n,t}, \tilde \nu_{n,t} \leq n  \bigg)  \notag\\
&\leq 
\frac{c}{n^{\beta }}   \bigg( \max \{t,0\} + | \sigma(g_p\cdots g_1, x)| +|\tilde f(a)| +   \frac{1}{n} \sum_{k = p+1}^n \mathcal F_k(a) \bigg) 
\end{align}
and
\begin{align}\label{sum-hp-nu-nt-002}
&\bb E_a \bigg( |\tilde f(\xi_{\tilde \nu_{n,t}})|; \; \tilde \tau^{\mathfrak f}_{t} > \tilde \nu_{n,t},  \tilde \nu_{n,t} \leq n  \bigg) 
 \notag\\
&\leq 
\frac{c}{n^{1 + \beta }}  \left( 1 + \max \{t,0\} + | \sigma(g_p\cdots g_1, x)| +|\tilde f(a)|  \right)  \sum_{k = p+1}^n \mathcal F_k(a). 
\end{align}
\end{lemma}

\begin{proof} 
We first prove \eqref{sum-hp-nu-nt-001}. 
By H\"older's inequality, with $\eta >1$ and $\frac{1}{\eta} +\frac{1}{\eta'}=1$,
\begin{align} \label{app-lem-10-001}
 & \bb E_a \Bigg( \Bigg| \sum_{i=\tilde \nu_{n,t}-p+1}^{\tilde \nu_{n,t}} h_p(\xi_i) \Bigg|; \; \tilde \tau^{\mathfrak f}_{t} >\tilde \nu_{n,t}, \tilde \nu_{n,t} \leq n  \Bigg) \notag\\
 & \leq 
  \bb E_a^{1/\eta} \Bigg( \Bigg| \sum_{i=\tilde \nu_{n,t}-p+1}^{\tilde \nu_{n,t}} h_p(\xi_i) \Bigg|^\eta; \tilde \nu_{n,t} \leq n  \Bigg) 
   \bb P_a^{1/\eta'} \left(  \tilde \tau^{\mathfrak f}_{t} >\tilde \nu_{n,t}  \right).
\end{align}
For the first expectation on the right-hand side of \eqref{app-lem-10-001}, in view of \eqref{nu n02}, we write
\begin{align*}
\bb E_a \Bigg( \Bigg| \sum_{i=\tilde \nu_{n,t}-p+1}^{\tilde \nu_{n,t}} h_p(\xi_i) \Bigg|^\eta; \tilde \nu_{n,t} \leq n  \Bigg)
= \sum_{\ell =p+1}^{n} 
\bb E_a \Bigg( \Bigg| \sum_{i=\ell-p+1}^{\ell} h_p(\xi_i)  \Bigg|^\eta; \tilde \nu_{n,t}=\ell  \Bigg)
\leq  c p^{\eta} n, 
\end{align*}
where the last inequality holds due to the moment assumption \eqref{exp mom for f 001}. 
Substituting this into \eqref{app-lem-10-001} and using Lemma \ref{Lemma-nu-tau}, 
we get that there exist constants $c, \beta, \ee > 0$  such that, 
for any $a \in \bb A_p$, $t \in \bb R$, $n\geq 1$ and $p\leq n^{\ee} $,  
\begin{align*}
& \bb E_a \Bigg( \Bigg| \sum_{i=\tilde \nu_{n,t}-p+1}^{\tilde \nu_{n,t}} h_p(\xi_i) \Bigg|; \; \tilde \tau^{\mathfrak f}_{t} >\tilde \nu_{n,t}, \tilde \nu_{n,t} \leq n  \Bigg)\\
 &\leq  c \left( p^{\eta} n    \right)^{1/\eta}  \bb P_a^{1/\eta'} \left(  \tilde \tau^{\mathfrak f}_{t} >\tilde \nu_{n,t}  \right)\\
 &\leq      \frac{ c \left(   p^{\eta} n   \right)^{1/\eta} }{n^{\beta \eta'}}
 \bigg( \max \{t,0\} + | \sigma(g_p\cdots g_1, x)| + |\tilde f(a)| + \frac{1}{n} \sum_{k = p+1}^n \mathcal F_k(a) \bigg)^{1/\eta'}   \notag\\
 &\leq  
\frac{c}{n^{\beta/2} }  \bigg( \max \{t,0\} + | \sigma(g_p\cdots g_1, x)| + |\tilde f(a)| + \frac{1}{n} \sum_{k = p+1}^n \mathcal F_k(a) \bigg), 
\end{align*}
where $\mathcal F_{\ell}(a)$ is defined by \eqref{def-Fka}, 
and in the last inequality we take $\eta' > 1$ to be sufficiently close to $1$. 
This concludes the proof of \eqref{sum-hp-nu-nt-001}.

We next prove \eqref{sum-hp-nu-nt-002}. 
By H\"older's inequality, with $\eta >1$ and $\frac{1}{\eta} +\frac{1}{\eta'}=1$,
\begin{align} \label{app-lem-second-ine001}
 \bb E_a \left( |\tilde f(\xi_{\tilde \nu_{n,t}})|; \; \tilde \tau^{\mathfrak f}_{t} >\tilde \nu_{n,t}, \tilde \nu_{n,t} \leq n  \right)
 \leq  \bb E_a^{1/\eta} \left( |\tilde f(\xi_{\tilde \nu_{n,t}})|^{\eta};  \tilde \nu_{n,t} \leq n   \right)
   \bb P_a^{1/\eta'} \left(  \tilde \tau^{\mathfrak f}_{t} >\tilde \nu_{n,t}  \right).
\end{align}
In view of \eqref{nu n02}, we write
\begin{align*}
\bb E_a \left( |\tilde f(\xi_{\tilde \nu_{n,t}})|^\eta; \tilde \nu_{n,t} \leq n  \right)
= \sum_{\ell =p+1}^{n} 
\bb E_a \left( |\tilde f(\xi_{\ell})|^\eta; \tilde \nu_{n,t} = \ell  \right)
\leq \sum_{\ell =p+1}^{n}  \mathcal F_{\ell}(a),  
\end{align*}
where $\mathcal F_{\ell}(a)$ is defined by \eqref{def-Fka}. 
Then, by \eqref{app-lem-second-ine001} and Lemma \ref{Lemma-nu-tau},
\begin{align*}
& \bb E_a \left( |\tilde f(\xi_{\tilde \nu_{n,t}})|; \; \tilde \tau^{\mathfrak f}_{t} >\tilde \nu_{n,t},  \tilde \nu_{n,t} \leq n  \right) \notag\\
 &\leq   \bigg( \sum_{\ell =p+1}^{n}  \mathcal F_{\ell}(a)   \bigg)^{1/\eta}  
  \bb P_a^{1/\eta'} \left(  \tilde \tau^{\mathfrak f}_{t} >\tilde \nu_{n,t}  \right)  \notag\\
 &\leq   \bigg(  \sum_{\ell =p+1}^{n}  \mathcal F_{\ell}(a)   \bigg)^{1/\eta}  
  \frac{c}{n^{\beta \eta'}} \bigg( \max \{t,0\} + | \sigma(g_p\cdots g_1, x)| + |\tilde f(a)| + \frac{1}{n} \sum_{k = p+1}^n \mathcal F_k(a) \bigg)^{1/\eta'}.
\end{align*}
By taking $\eta'$ to be sufficiently close to $1$, we get
\begin{align*}
& \bb E_a \bigg( |\tilde f(\xi_{\tilde \nu_{n,t}})|; \; \tilde \tau^{\mathfrak f}_{t} >\tilde \nu_{n,t}, \tilde \nu_{n,t} \leq n  \bigg)  \notag\\
 &  \leq  
\frac{c}{n^{1+\beta/2} }  \left( 1 + \max \{t,0\} + | \sigma(g_p\cdots g_1, x)| + |\tilde f(a)|  \right)  
   \sum_{k = p+1}^n \mathcal F_k(a). 
\end{align*}
This ends the proof of the second inequality \eqref{sum-hp-nu-nt-002} of the lemma.
\end{proof}

\begin{lemma} \label{Lemma martingale-002}
There exist constants $c, \beta, \ee > 0$  such that, 
for any $a=(g_0,\ldots, g_{p}, x, q) \in \bb A_p$, $t \in \bb R$, $n\geq 1$ and $p\leq n^{\ee} $,  
\begin{align*}
I: & = \bb{E}_a  \Bigg( \Bigg| \sum_{j=\tilde \tau^{\mathfrak f}_{t} -p+1}^{\tilde \tau^{\mathfrak f}_{t}} h_p(\xi_j) \Bigg| 
 + |\tilde f(\xi_{ \tilde \tau^{\mathfrak f}_{t} })|; 
\  \tilde \nu_{n,t}  < \tilde \tau^{\mathfrak f}_{t} \leq  [n^{1-\ee }]   \Bigg) \\ 
&  \leq   \frac{c}{n^{1 + \beta }} \left( 1 + \max \{t,0\} + | \sigma(g_p\cdots g_1, x)| + |\tilde f(a)|  \right) \sum_{k = p+1}^n \mathcal F_k(a). 
\end{align*}
\end{lemma}

\begin{proof} 
By H\"older's inequality, with $\eta,\eta' >1$ and $\frac{1}{\eta} +\frac{1}{\eta'}=1$, we have 
\begin{align} \label{CSH-LLL-001}
I \leq 
   \bb E_a^{1/\eta} \Bigg( \Bigg| \sum_{j=\tilde \tau^{\mathfrak f}_{t} -p+1}^{\tilde \tau^{\mathfrak f}_{t}} h_p(\xi_j) \Bigg|^{\eta} + |\tilde f(\xi_{ \tilde \tau^{\mathfrak f}_{t} })|^\eta; 
  p+1< \tilde \tau^{\mathfrak f}_{t} \leq  [n^{1-\ee }] \Bigg)
   \bb P_a^{1/\eta'} \left(  \tilde \tau^{\mathfrak f}_{t} >\tilde \nu_{n,t}  \right).
\end{align}
Applying \eqref{def of func h-001} and the moment assumption \eqref{exp mom for f 001}, we get
\begin{align*}
\bb E_a \Bigg( \Bigg| \sum_{j = \tilde \tau^{\mathfrak f}_{t}-p+1}^{\tilde \tau^{\mathfrak f}_{t}} h_p(\xi_j) \Bigg|^\eta; \ p+1<  \tilde \tau^{\mathfrak f}_{t} \leq  [n^{1-\ee }]  \Bigg) 
& =  \sum_{\ell =p+2}^{ [n^{1-\ee}] } \bb E_a \Bigg( \Bigg| \sum_{j = \ell-p+1}^{\ell} h_p(\xi_j)  \Bigg|^\eta; \tilde \tau^{\mathfrak f}_{t}=\ell  \Bigg)\\
&  \leq \sum_{\ell =p+2}^{ [n^{1-\ee}] } \bb E_a \Bigg( \Bigg| \sum_{j = \ell-p+1}^{\ell} h_p(\xi_j)  \Bigg|^\eta  \Bigg)
\leq c  n p^{\eta}. 
\end{align*}
In view of \eqref{def-Fka}, it holds that
\begin{align*}
\bb E_a \left( |\tilde f(\xi_{ \tilde \tau^{\mathfrak f}_{t} })|^\eta; \ p +1 < \tilde \tau^{\mathfrak f}_{t} \leq  [n^{1-\ee }]  \right)
= \sum_{\ell =p+2}^{ [n^{1-\ee}] } \bb E_a \Bigg( |\tilde f(\xi_{ \ell })|^\eta; \tilde \tau^{\mathfrak f}_{t}=\ell  \Bigg)
\leq c \sum_{\ell =p+1}^{n} \mathcal F_{\ell}(a).
\end{align*}
The last two bounds give
\begin{align*}
& \bb E_a^{1/\eta} \Bigg( \Bigg|  \sum_{j=\tilde \tau^{\mathfrak f}_{t} -p+1}^{\tilde \tau^{\mathfrak f}_{t}} h_p(\xi_j)  \Bigg|^{\eta} 
+ |\tilde f(\xi_{ \tilde \tau^{\mathfrak f}_{t} })|^\eta; \ p +1 < \tilde \tau^{\mathfrak f}_{t} \leq  [n^{1-\ee }]   \Bigg)  \notag\\
&\leq  c \bigg(  n p^{\eta}    +  \sum_{\ell =p+1}^{n} \mathcal F_{\ell}(a)    \bigg)^{1/\eta}  
\leq  c n^{1/\eta}  \bigg( p  + \frac{1}{n} \sum_{\ell =p+1}^{n} \mathcal F_{\ell}(a)  \bigg)^{1/\eta}.
\end{align*}
Using Lemma \ref{Lemma-nu-tau} and taking $\eta'$ to be sufficiently close to $1$, we get
\begin{align*}
\bb P_a^{1/\eta'} \left( \tilde \tau^{\mathfrak f}_{t} >\tilde \nu_{n,t}  \right)
 \leq \frac{c}{n^{\beta/2 }} \bigg( \max \{t,0\} + | \sigma(g_p\cdots g_1, x)| + |\tilde f(a)| + \frac{1}{n} \sum_{k = p+1}^n \mathcal F_k(a) \bigg)^{1/\eta'} . 
\end{align*}
Therefore, by \eqref{CSH-LLL-001} and taking into account that $\eta > 1$ is sufficiently large, we obtain that for $p\leq n^{\ee} $,  
\begin{align*}
I &\leq c n^{1/\eta}  \bigg( p  + \frac{1}{n} \sum_{\ell =p+1}^{n} \mathcal F_{\ell}(a)  \bigg)^{1/\eta} \\
&\quad \times \frac{1}{n^{\beta/2 }}  \bigg( \max \{t,0\} + | \sigma(g_p\cdots g_1, x)| + |\tilde f(a)| + \frac{1}{n} \sum_{k = p+1}^n \mathcal F_k(a)  \bigg)^{1/\eta'}  \\
& \leq \frac{c}{n^{1 + \beta/4 }} \left( \max \{t,0\} + | \sigma(g_p\cdots g_1, x)|  + |\tilde f(a)|  +  1 \right) \sum_{k = p+1}^n \mathcal F_k(a),
\end{align*}
which concludes the proof of the lemma.
\end{proof}

\begin{lemma} \label{Lemma 3}
There exists a constant $\beta > 0$ with the following property. 
For any $\ee \in (0, \frac{1}{2})$, there is a constant $c>0$ such that for any $n \geq 1$, $p\leq n^{\ee}$, $a=(g_0,\ldots,g_{p},x,q) \in \bb A_p$ and $t \in \bb R$,
\begin{align*}
&  \bb{E}_a \left( \left\vert t + \sigma(g_{p} \cdots g_{1}, x) + \sum_{i=1}^n h_p(\xi_i) \right\vert; 
\tilde \tau^{\mathfrak f}_{t} >n, \tilde \nu_{n,t} >n^{1-\ee }\right)   \notag\\
& \leq  
c \frac{  \max \{t,0\} + \Big|\sigma(g_p\cdots g_1,x) \Big| + |\tilde f(a)|}{n^{\beta}} 
\left(  \mathcal F_n(a)  + |\tilde f(a)| \right)
 +  |\tilde f(a)| \frac{ \sum_{k = p+1}^n \mathcal F_k(a) }{n^3}. 
\end{align*}
\end{lemma}

\begin{proof}
Using Lemma \ref{Lem-bound-tau_p-new}, we get 
\begin{align*}
&  \bb{E}_a \left( \left\vert t + \sigma(g_{p} \cdots g_{1}, x) + \sum_{i=1}^n h_p(\xi_i) \right\vert; 
\tilde \tau^{\mathfrak f}_{t} >n, \tilde \nu_{n,t} >n^{1-\ee }\right)   \notag\\
&\leq  \bb{E}_a \left(  t + \sigma(g_{n} \cdots g_{1}, x) + \tilde f(\xi_n) - \tilde f(\xi_0); 
\tilde \tau^{\mathfrak f}_{t} >n, \tilde \nu_{n,t} >n^{1-\ee }\right)   \notag\\
&\quad + c \frac{  \max \{t,0\} + \Big|\sigma(g_p\cdots g_1,x) \Big| + |\tilde f(a)|}{n^{\beta}} 
\left(  \mathcal F_n(a)  + |\tilde f(a)| \right)
 +  |\tilde f(a)| \frac{ \sum_{k = p+1}^n \mathcal F_k(a) }{n^3}.
\end{align*}
Then, by  the Cauchy-Schwarz inequality, we get that, for any $t \in \bb R,$
\begin{align*}
& \bb{E}_a \left(  t + \sigma(g_{n} \cdots g_{1}, x) + \tilde f(\xi_n) - \tilde f(\xi_0); 
\tilde \tau^{\mathfrak f}_{t} >n, \tilde \nu_{n,t} >n^{1-\ee }\right) \notag\\
& \leq \bb{E}_a^{1/2} \left[ \Big(  t + \sigma(g_{n} \cdots g_{1}, x) + \tilde f(\xi_n) - \tilde f(\xi_0) \Big) ^{2};  \tilde \tau^{\mathfrak f}_{t} >n  \right] 
     \mathbb{P}_a^{1/2}\left( \tilde \nu_{n,t}>n^{1-\ee }\right). 
\end{align*}%
By Lemma \ref{Lem-trick}, we have
\begin{align*}
& \bb{E}_a \left(  t + \sigma(g_{n} \cdots g_{1}, x) + \tilde f(\xi_n) - \tilde f(\xi_0); 
\tilde \tau^{\mathfrak f}_{t} >n, \tilde \nu_{n,t} >n^{1-\ee }\right)   \notag\\
&  \leq \bb{E}_a^{1/2}\left( \Big(  t + \sigma(g_{n} \cdots g_{1}, x) + \tilde f(\xi_n) - \tilde f(\xi_0) \Big) ^{2};  \tilde \tau^{\mathfrak f}_{t} >n  \right) 
     \mathbb{P}_a^{1/2}\left( \tilde \nu_{n,t}>n^{1-\ee }\right) \notag\\ 
& \leq \left[ \max \{t,0\} +  \bb{E}_a^{1/2}\left( \Big( \sigma(g_{n} \cdots g_{1}, x) + \tilde f(\xi_n) - \tilde f(\xi_0) \Big) ^{2}\right) 
   \right]
     \mathbb{P}_a^{1/2}\left( \tilde \nu_{n,t}>n^{1-\ee }\right) \notag \\ 
& \leq  c \left( \max \{t,0\} +  n + \sigma(g_{p} \cdots g_{1}, x) + \mathcal F_n(a)  + | \tilde f(a) |  \right)  
    e^{- \beta n^{\ee}}, 
\end{align*}
where in the last bound we used Minkowski's inequality and Lemma \ref{Lemma 2-app}. 
\end{proof}


\subsection{A monotonicity bound for the killed martingale $(M_{n})_{n\geq 1}$} \label{sec-bound for MC}

In this subsection, the functions $f_n$  depend only on the first $p$ coordinates. 
We aim to establish a monotonicity property for the following sequence, which was introduced previously in  
Subsection \ref{Subsec-Markov-chian} (see \eqref{EXPECT-E_x-002} and \eqref{EXPECT-E_x-001}): 
\begin{align} \label{recall-Wfn-at001}
W^{\mathfrak f}_{n} (a, t) 
= \bb{E}_a \left( t'+M_{n};\ \tilde \tau^{\mathfrak f}_{t} >n\right)
= \bb E_a \bigg( t + \sum_{i=1}^{n+p} \sigma_p(\xi_i); \tilde \tau^{\mathfrak f}_{t} >n \bigg), 
\end{align}
where $a = (g_0,\ldots, g_{p}, x, q) \in \bb A_p$ and $t' = t + \sigma(g_{p} \cdots g_{1}, x)$.
The lemma presented below plays a key role in establishing the main results of this paper. 
Specifically, it plays a crucial role in proving Propositions \ref{Prop 001-app} and \ref{Prop 001-app2}. 
Although we will apply a conceptual approach similar to that employed in establishing Proposition \ref{Prop-Vn-001}, 
it is important to note that this proof is significantly more complex and technically demanding.

\begin{lemma} \label{Lemma 4}
There exists a constant $\ee_0 > 0$  such that  
for any $\ee \in (0, \ee_0)$, the following holds:
there exists a constant $c_{\ee} > 0$  such that for any $n \geq 1$, $1 \leq p \leq n^{\ee}$, 
$a=(g_0,\ldots, g_{p}, x, q) \in \bb A_p$, $t \in \bb R$ 
and any sequence $\mathfrak f = (f_n)_{n \geq 0}$ of $\scr A_p$-measurable functions, 
\begin{align}\label{bound-M_n-001}
W^{\mathfrak f}_{n} (a, t) 
 &\leq \left( 1+\frac{c_{\ee}}{n^{\ee }}\right)  
   W^{\mathfrak f}_{[n^{1-\ee }]} (a, t)  + \frac{c_{\ee}}{n^{\ee}}  ( 1+  |\sigma(g_{p} \cdots g_{1}, x)| )
   \bigg( |\tilde f(a)| + \frac{1}{n} \sum_{j = p+1}^n \mathcal F_j(a)  \bigg)
   \notag  \\
& \qquad\qquad\qquad  \times  
\bigg(\max \{t,0\}  +  |\sigma(g_{p} \cdots g_{1}, x)| + |\tilde f(a)| + \frac{1}{n} \sum_{j = p+1}^n \mathcal F_j(a)  \bigg).
\end{align}
\end{lemma}

\begin{proof}
We shall first prove that there exist constants $c, c', \ee_0 > 0$ such that for any 
$n\geq 1$, $\ee \in \left( 0,\ee _{0}\right) $, $1\leq p\leq n^{\ee_0}$, 
 $t\geq n^{1/2-\ee }$ and $a=(g_0,\ldots, g_{p}, x, q) \in \bb A_p$,  
\begin{align}\label{Un-xy-Bound-001}
W^{\mathfrak f}_{n} (a, t) 
&\leq \left( 1+ \frac{c}{n^{\ee}} \right) t  +  |\sigma(g_{p} \cdots g_{1}, x)| \notag \\
& \quad + c e^{- c' n^{1/3} } \Big(1+ t + |\sigma(g_{p} \cdots g_{1}, x)| \Big)   
\sum_{k = p+1}^n \mathcal F_k(a). 
\end{align}
Let $a \in \bb A_p$ and $t' = t + \sigma(g_p \cdots g_1, x)$. 
By \eqref{def-Mn-martingel-001} and Lemma \ref{Lemma-Martingale001}, the sequence of functions 
\begin{align*}
M_0: = 0,  \quad 
 M_n : \xi \mapsto  \sum_{i=1}^n h_p(\xi_i),  \quad n \geq 1
\end{align*}
is a martingale with respect to the probability measure $\bb P_a$ and the natural sequence of  $\sigma$-algebras $(\mathscr G_n)_{n \geq 0}$. 
By \eqref{recall-Wfn-at001} and the optional stopping theorem, we get
\begin{align}\label{bound G 000}
W^{\mathfrak f}_{n} (a, t) 
&=\bb{E}_a \left(t'+M_{n}\right) -\bb{E}_a \left( t'+M_{n};\ \tilde \tau^{\mathfrak f}_{t} \leq n\right)  \notag \\
&=t'-\bb{E}_a \left( t'+M_{\tilde \tau^{\mathfrak f}_{t}};\  \tilde \tau^{\mathfrak f}_{t}  \leq n\right)  \notag \\
& =  t' - \bb{E}_a \left(  t'+M_{\tilde \tau^{\mathfrak f}_{t}};\ \tilde \tau^{\mathfrak f}_{t} \leq p \right)
   - \bb{E}_a \left(  t'+M_{\tilde \tau^{\mathfrak f}_{t}};\ p < \tilde \tau^{\mathfrak f}_{t} \leq n \right)  \notag\\
& \leq t' -  \bb{E}_a \left(  t'+M_{p};\ \tilde \tau^{\mathfrak f}_{t} \leq p \right)
  + \bb{E}_a \left( \left\vert t' + M_{\tilde \tau^{\mathfrak f}_{t}} \right\vert ;\ p < \tilde \tau^{\mathfrak f}_{t} \leq n \right). 
\end{align}
For the first expectation on the right-hand side of \eqref{bound G 000}, we have 
\begin{align}\label{Second-Expect}
 t' - \bb{E}_a \left(  t'+M_{p};\ \tilde \tau^{\mathfrak f}_{t} \leq p \right)
& =   \bb{E}_a \left(  t'+M_{p};\ \tilde \tau^{\mathfrak f}_{t} > p \right)   \notag\\
& =  t' \bb P_a \left( \tilde \tau^{\mathfrak f}_{t} > p \right) +  \bb{E}_a \left(  M_{p};\ \tilde \tau^{\mathfrak f}_{t} > p \right)  \notag\\
& \leq \max\{t', 0\} + c p  \leq  \left| t + \sigma(g_{p} \cdots g_{1}, x) \right| + cp. 
\end{align}
For the second expectation on the right-hand side of \eqref{bound G 000},
by Lemma \ref{Lemma 1}, 
there exist constants $c, c', \ee_0 >0$
such that for any $\ee\in (0,\ee_0)$, $a \in \bb A_p$ and $t\geq  n^{1/2-\ee },$
\begin{align}\label{First-Expect}
\bb{E}_a \left( \left\vert t' + M_{\tilde \tau^{\mathfrak f}_{t}} \right\vert ;\ p < \tilde \tau^f_{t} \leq n \right)
\leq \frac{t}{n^{\ee }} + 
c e^{- c' n^{1/3} } \Big(1+ \Big| t + \sigma(g_p\cdots g_1,x) \Big|\Big)   \sum_{k = p+1}^n \mathcal F_k(a), 
\end{align}
where we have used the fact that $p\leq n^{\ee_0}$.
By substituting \eqref{Second-Expect} and \eqref{First-Expect} into \eqref{bound G 000}, we obtain the desired inequality \eqref{Un-xy-Bound-001}.

Now we proceed to \eqref{bound-M_n-001}. To do this, we will apply the Markov property 
and the bound \eqref{Un-xy-Bound-001}. 
First, we note that, by \eqref{recall-Wfn-at001}, for any $a\in \bb A_p$, $t \in \bb R$, $n \geq 1$ and $\ee >0,$ 
\begin{align}\label{bound J1+J2}
W^{\mathfrak f}_{n} (a, t)
  = \bb{E}_a \Big( t'+M_{n};\ \tilde \tau^{\mathfrak f}_{t} >n, \tilde \nu_{n,t} >  n^{1-\ee } \Big)  
 + \bb{E}_a \Big( t'+M_{n};\ \tilde \tau^{\mathfrak f}_{t} >n,  \tilde \nu_{n,t}  \leq n^{1-\ee }  \Big), 
\end{align}
where $\tilde \nu_{n,t}$ is defined by \eqref{nu n02}. 
For the first term, by Lemma \ref{Lemma 3},
for any $\ee \in \left( 0,\ee_{0}\right)$,  
\begin{align}\label{bound-J2-ay}
& \bb{E}_a \left( |t'+M_{n}|;\ \tilde \tau^{\mathfrak f}_{t} >n,   \tilde \nu_{n,t} > n^{1-\ee }\right)   \notag\\
&\leq c \frac{  \max \{t,0\} + \left| \sigma(g_p\cdots g_1,x) \right| + |\tilde f(a)|}{n^{\beta}} 
\left(  \mathcal F_n(a)  + |\tilde f(a)| \right)
 +  |\tilde f(a)| \frac{ \sum_{k = p+1}^n \mathcal F_k(a) }{n^3}.
\end{align}
For the second term on the right-hand side of \eqref{bound J1+J2}, we decompose it into two parts: 
\begin{align}\label{bound J1 001}
& \bb{E}_a \left( t'+M_{n};\ \tilde \tau^{\mathfrak f}_{t} > n,\ \tilde \nu_{n,t} \leq n^{1-\ee } \right)  \notag\\
& =  \bb{E}_a \left( t'+M_{n};\ \tilde \tau^{\mathfrak f}_{t} > n,\ \tilde \nu_{n,t} \leq n^{1-\ee },  |\tilde f(\xi_{\tilde \nu_{n,t}} )| >  n^{1/2-\ee} \right) \notag\\
& \quad +   \bb{E}_a \left( t'+M_{n};\ \tilde \tau^{\mathfrak f}_{t} > n,\ \tilde \nu_{n,t} \leq n^{1-\ee },  |\tilde f(\xi_{\tilde \nu_{n,t}} )| \leq  n^{1/2-\ee} \right)  \notag\\
& =:  I_1 (a, t) + I_2 (a, t). 
\end{align}
For the first term $I_1 (a, t)$,  by the Cauchy-Schwarz inequality, \eqref{def-Fka} and Markov's inequality, we have 
\begin{align} \label{bound J12 001}
I_1 (a, t)
& = \sum_{k=p+1}^{ [n^{1-\ee }] }  
 \bb{E}_a \left( t'+M_{n};\ \tilde \tau^{\mathfrak f}_{t} > n,\ \tilde \nu_{n,t} = k,  |\tilde f(\xi_{k} )| >   n^{1/2-\ee} \right) \notag\\
& \leq \sum_{k=p+1}^{ [n^{1-\ee }] }  
 \bb{E}_a^{1/2} \left( |t'+M_{n}|^2 \right)    
  \mathbb{P}_a^{1/2} \left(  |\tilde f(\xi_{k} )| >   n^{1/2-\ee} \right) \notag\\
 & \leq c'  \left( \max\{t, 0\} + |\sigma(g_{p} \cdots g_{1}, x)| + \sqrt{n} \right) e^{-\frac{\alpha}{2} n^{1/2-\ee}} 
  \bigg(  \sum_{k=p+1}^{ n }  \mathcal F_k(a) \bigg)^{1/2} \notag\\
 & \leq c' \left( 1+\max\{t, 0\} + |\sigma(g_{p} \cdots g_{1}, x)| \right) e^{-c n^{1/2-\ee}} 
  \bigg(  \sum_{k=p+1}^{ n }  \mathcal F_k(a) \bigg)^{1/2} . 
\end{align}
For the second term $I_2 (a, t)$ in \eqref{bound J1 001}, 
we note that, by \eqref{def-Mn-martingel-001} 
and the fact that $t' = t + \sigma(g_p \cdots g_1, x) = t + \sum_{i=1}^p \sigma_p (\xi_i)$, 
it holds that $t' + M_n = t +  \sum_{i=1}^{n+p} \sigma_p(\xi_i)$
and 
\begin{align}\label{equality-hp-sigmap}
t' + \sum_{i=1}^{k-p} h_p(\xi_i) = t + \sum_{i=1}^{k} \sigma_p(\xi_i). 
\end{align}
Hence, applying the Markov property (see Lemma \ref{Lemma-Markov}), we get 
\begin{align} \label{bound J1 002}
I_2 (a, t) 
& = \sum_{k=p+1}^{ [n^{1-\ee }] }
\bb{E}_a \left( t' + M_n; \ \tilde \tau^{\mathfrak f}_{t} > n,\ \tilde \nu_{n,t}=k,  |\tilde f(\xi_k )| \leq  n^{1/2-\ee} \right)  \notag\\
& = \sum_{k=p+1}^{ [n^{1-\ee }] }
  \bb E_a \bigg[ W^{\mathfrak f}_{n-k} \bigg( \xi_k, t' + \sum_{i=1}^{k-p} h_p(\xi_i) + \tilde f(\xi_k)- \tilde f(\xi_0) \bigg); 
      A_{n,k} \bigg], 
\end{align}
where 
\begin{align}\label{def-Ank}
A_{n,k} = \left\{ \tilde \tau^{\mathfrak f}_{t} >k, \ \tilde \nu_{n,t}=k, |\tilde f(\xi_k )| \leq  n^{1/2-\ee}  \right\}. 
\end{align}
Using the definition of the stopping times $\tilde \nu_{n, t}$ and $\tilde\tau^{\mathfrak f}_{t}$
(see \eqref{nu n02} and \eqref{def-tau-f-y})
and \eqref{equality-hp-sigmap}, 
one can verify that on the event $A_{n, k}$, 
we have simultaneously
$| t' + \sum_{i=1}^{k-p} h_p(\xi_i)  - \tilde f(\xi_0) | \geq 2 n^{1/2-\ee}$,
$ t' + \sum_{i=1}^{k-p} h_p(\xi_i) + \tilde f(\xi_k ) - \tilde f(\xi_0)\geq 0 $
and $|\tilde f(\xi_k )| \leq n^{1/2-\ee}$. All these bounds imply that
$ t' + \sum_{i=1}^{k-p} h_p(\xi_i) + \tilde f(\xi_k)- \tilde f(\xi_0)   \geq  n^{1/2-\ee}$.
Thus, we have the inclusion: 
\begin{align*}
A_{n,k }  \subseteq \bigg\{  t' + \sum_{i=1}^{k-p} h_p(\xi_i) + \tilde f(\xi_k)- \tilde f(\xi_0) \geq   n^{1/2-\ee }\bigg\}.
\end{align*}
On the event $A_{n,k}$, 
applying \eqref{Un-xy-Bound-001}, there exist constants $c, c', \ee_0 > 0$ such that for any 
$n\geq 1$, $\ee \in \left( 0,\ee _{0}\right) $, $1\leq p\leq n^{\ee_0}$, $k \in [p + 1, [n^{1-\ee }]]$, 
 $t\geq n^{1/2-\ee }$ and $a \in \bb A_p$, 
\begin{align*}
& W^{\mathfrak f}_{n-k} \bigg( \xi_k, t' + \sum_{i=1}^{k-p} h_p(\xi_i) + \tilde f(\xi_k)- \tilde f(\xi_0)  \bigg)  \notag\\
&\leq \bigg( 1+ \frac{c}{(n-k)^{\ee}} \bigg) \bigg(t' + \sum_{i=1}^{k-p} h_p(\xi_i)+ \tilde f(\xi_k)- \tilde f(\xi_0)  \bigg)  + |\sigma(g_{p} \cdots g_{1}, x)| \notag\\ 
& \quad + c e^{- c' (n-k)^{1/3} } 
 \bigg(1+t' + \sum_{i=1}^{k-p} h_p(\xi_i)+ \tilde f(\xi_k)- \tilde f(\xi_0)+ |\sigma(g_{p} \cdots g_{1}, x)|  \bigg) 
   \sum_{j = p+1}^{n} \mathcal F_j(a)  \notag\\
&\leq \bigg( 1+ \frac{c}{n^{\ee}} \bigg)  \bigg(t' + \sum_{i=1}^{k-p} h_p(\xi_i)+ \tilde f(\xi_k)- \tilde f(\xi_0) \bigg)  + |\sigma(g_{p} \cdots g_{1}, x)| \notag\\
& \quad + c e^{- c' n^{1/4} } 
\bigg(1+t' + \sum_{i=1}^{k-p} h_p(\xi_i)+ \tilde f(\xi_k)- \tilde f(\xi_0)+ |\sigma(g_{p} \cdots g_{1}, x)|  \bigg)  
  \sum_{j = p+1}^{n} \mathcal F_j(a), 
\end{align*}
which, together with the definition of $A_{n, k}$ (cf.\ \eqref{def-Ank}), implies that 
\begin{align}\label{bound E001}
& \bb E_a \bigg[ W^{\mathfrak f}_{n-k} \bigg( \xi_k, t' + \sum_{i=1}^{k-p} h_p(\xi_i) + \tilde f(\xi_k)- \tilde f(\xi_0)  \bigg); 
      A_{n,k}  \bigg]  \notag\\ 
& \leq \bigg( 1+ \frac{c}{n^{\ee}} \bigg) 
\bb E_a \bigg( t' + \sum_{i=1}^{k-p} h_p(\xi_i)- \tilde f(\xi_0); \;  A_{n,k}  \bigg) 
 + c \bb E_a  \left( |\tilde f(\xi_k)|; \tilde \tau^{\mathfrak f}_{t} >k, \tilde \nu_{n,t}=k  \right)   \notag \\
& \quad +  |\sigma(g_{p} \cdots g_{1}, x)|  \bb P_a \Big( \tilde \tau^{\mathfrak f}_{t} >k, \tilde \nu_{n,t}=k  \Big) 
\notag \\
& \quad + c  e^{- c' n^{1/4} } \bigg(  \sum_{j = p+1}^{n} \mathcal F_j(a) \bigg) \notag\\
& \qquad \times \bb E_a \bigg(  
  1 + t' + \sum_{i=1}^{k-p} | h_p(\xi_i) | +|\tilde f(\xi_0)|+|\tilde f(\xi_k)|  +  |\sigma(g_{p} \cdots g_{1}, x)|; \;  A_{n,k}  \bigg).
\end{align} 
Substituting \eqref{bound E001} into \eqref{bound J1 002}, we obtain 
\begin{align}\label{inequality-J11-hh}
I_2 (a, t)
&\leq \left( 1+ \frac{c}{n^{\ee}} \right)   \sum_{k=p+1}^{ [ n^{1-\ee }] }
 \bb E_a \bigg( t' + \sum_{i=1}^{k-p} h_p(\xi_i)- \tilde f(\xi_0); \; \tilde \tau^{\mathfrak f}_{t} >k, \tilde \nu_{n,t}=k,  |\tilde f(\xi_k )| \leq  n^{1/2-\ee}  \bigg) \notag \\
& \quad  
+ c \bb E_a \Big(|f(\xi_{\tilde \nu_{n,t}})|; \tilde \tau^{\mathfrak f}_{t} >\tilde \nu_{n,t}, \tilde \nu_{n,t} \leq n^{1-\ee}  \Big) 
+   |\sigma(g_{p} \cdots g_{1}, x)|  \bb P_a \left( \tilde \tau^{\mathfrak f}_{t} > \tilde \nu_{n,t}  \right) 
\notag \\
&\quad + c   e^{- c' n^{1/4} } \bigg(  \sum_{j = p+1}^{n} \mathcal F_j(a) \bigg)  \notag\\
& \qquad \times \sum_{k=p+1}^{ [ n^{1-\ee }] }  
 \bb E_a \bigg(  1+ t' + \sum_{i=1}^{k - p} | h_p(\xi_i) |  
  + |\tilde f(\xi_0)|+|\tilde f(\xi_k)| +  |\sigma(g_{p} \cdots g_{1}, x)|; \;  A_{n,k} \bigg).
\end{align}
For the second term on the right-hand side of \eqref{inequality-J11-hh}, 
by the second inequality \eqref{sum-hp-nu-nt-002} of Lemma \ref{Lemma martingale-001}, 
there exist constants $c, \beta, \ee > 0$  such that, 
for any $a \in \bb A_p$, $t \in \bb R$, $n\geq 1$ and $p\leq n^{\ee}$,  
\begin{align} \label{xi_nu bound 001}
&\bb E_a \left( |\tilde f(\xi_{\tilde \nu_{n,t}})|; \; \tilde \tau^{\mathfrak f}_{t} > \tilde \nu_{n,t},  \tilde \nu_{n,t} \leq n  \right) \notag\\
&\leq 
c  \left( 1 + \max \{t,0\} + | \sigma(g_p\cdots g_1, x)| +|\tilde f(a)|  \right) \frac{1}{n^{1 + \beta }} \sum_{j = p+1}^n \mathcal F_j(a). 
\end{align}
For the third term on the right-hand side of \eqref{inequality-J11-hh}, by Lemma \ref{Lemma-nu-tau},
we get that for any $1\leq p\leq n^{\ee} $ and $t \in \bb R$, 
\begin{align} \label{SecondinJ11-001}
\bb P_a \left( \tilde \tau^{\mathfrak f}_{t} > \tilde \nu_{n,t}   \right) 
\leq 
\frac{c}{n^{\beta }} \bigg( \max \{t,0\} + | \sigma(g_p\cdots g_1, x)|  + |\tilde f(a)| + \frac{1}{n} \sum_{j = p+1}^n \mathcal F_j(a) \bigg). 
\end{align}
For the last term on the right-hand side of \eqref{inequality-J11-hh}, 
using the definition of $A_{n, k}$ (cf.\ \eqref{def-Ank}), 
the fact that $t' = t + \sigma(g_p \cdots g_1, x)$ and 
the moment assumption \eqref{exp mom for f 001}, 
we obtain
\begin{align}  \label{ThirdinJ11-001}
& \sum_{k=p+1}^{ [ n^{1-\ee }] }  
 \bb E_a \bigg(  1+ t' + \sum_{i=1}^{k - p} | h_p(\xi_i) |  
  + |\tilde f(\xi_0)|+|\tilde f(\xi_k)| +  |\sigma(g_{p} \cdots g_{1}, x)|; \;  A_{n,k} \bigg)  \notag\\
 & \leq \sum_{k=p+1}^{ [ n^{1-\ee }] }  
 \bb E_a \bigg(  1+\max \{t,0\} + \sum_{i=1}^{k - p} | h_p(\xi_i) |+|\tilde f(\xi_0)|+|\tilde f(\xi_k)| + |\sigma(g_{p} \cdots g_{1}, x)|;  \notag\\
& \qquad\qquad\qquad \tilde \tau^{\mathfrak f}_{t} > k , \tilde \nu_{n,t}=k \bigg) \notag \\
& \leq 
 \bb E_a \bigg(  1+ \max \{t,0\} + \sum_{i=1}^{\tilde \nu_{n,t}-p} | h_p(\xi_i) |+|\tilde f(\xi_0)| 
  + |\tilde f(\xi_{\tilde \nu_{n,t}})| 
 + |\sigma(g_{p} \cdots g_{1}, x)|; \; \tilde \nu_{n,t} \leq n^{1-\ee} \bigg) \notag  \\
& \leq \bigg(  1  + \max \{t,0\} + cn 
  + |\tilde f(a)|  + c \sum_{j = p+1}^n \mathcal F_j(a) + |\sigma(g_{p} \cdots g_{1}, x)|  \bigg) \notag\\
& \leq \bigg( \max \{t,0\} + |\tilde f(a)|  + |\sigma(g_{p} \cdots g_{1}, x)|  + c \sum_{j = p+1}^n \mathcal F_j(a) \bigg),  
\end{align}
where in the last inequality we used the fact that $p \in [1, n^{\ee}]$ and $\mathcal F_j(a) > 1$ for any $j \geq 1$. 
Substituting \eqref{xi_nu bound 001}, \eqref{SecondinJ11-001} and \eqref{ThirdinJ11-001} 
into \eqref{inequality-J11-hh},  
we obtain
\begin{align}\label{inequality-J11-hh-002}
I_2 (a, t)
&\leq \left( 1+ \frac{c}{n^{\ee}} \right)   \sum_{k=p+1}^{ [ n^{1-\ee }] }
 \bb E_a \bigg( t' + \sum_{i=1}^{k-p} h_p(\xi_i) - \tilde f(\xi_0); \; \tilde \tau^{\mathfrak f}_{t} >k, \tilde \nu_{n,t}=k,  |\tilde f(\xi_k )| \leq  n^{1/2-\ee}  \bigg) \notag \\
& \quad  
+  \frac{c}{n^{1+\beta}}   \left( 1 + |\sigma(g_p \cdots g_1, x)| \right)   \bigg(\sum_{j = p+1}^n \mathcal F_j(a) \bigg) \notag \\ 
& \qquad
 \times \bigg( |\tilde f(a)|+\max \{t,0\} +  |\sigma(g_{p} \cdots g_{1}, x)|  + \frac{1}{n}\sum_{j = p+1}^n \mathcal F_j(a)  \bigg). 
\end{align}
Now we proceed by completing the sum $t' + \sum_{i=1}^{k-p} h_p(\xi_i) - \tilde f(\xi_0)$ in
the first term on the right-hand side of \eqref{inequality-J11-hh-002} to obtain the full martingale expression
$t' + \sum_{i=1}^{k} h_p(\xi_i) - \tilde f(\xi_0)$.  
To achieve this, we will address the difference between the two sums, which is exactly 
$\sum_{i=k-p+1}^{k} h_p(\xi_i).$ 
Additionally, we want to transition from the event $\{|\tilde f(\xi_k )| \leq  n^{1/2-\ee}\}$ to its complement,  thus introducing an extra term. 
This leads us to decompose the first term on the right-hand side of \eqref{inequality-J11-hh-002} as follows: 
\begin{align} \label{decom-J11-123}
& \bb E_a \bigg( t' + \sum_{i=1}^{k-p} h_p(\xi_i) - \tilde f(\xi_0); \; \tilde \tau^{\mathfrak f}_{t} >k, \tilde \nu_{n,t}=k,  |\tilde f(\xi_k )| \leq  n^{1/2-\ee}  \bigg) \notag \\
& = \bb E_a \bigg( t' + \sum_{i=1}^{k} h_p(\xi_i) - \tilde f(\xi_0); \; \tilde \tau^{\mathfrak f}_{t} >k, \tilde \nu_{n,t}=k \bigg) \notag \\
& \quad - \bb E_a \bigg( \sum_{i=k-p+1}^{k} h_p(\xi_i); \; \tilde \tau^{\mathfrak f}_{t} >k, \tilde \nu_{n,t}=k \bigg) \notag \\
& \quad -  \bb E_a \bigg( t' + \sum_{i=1}^{k-p} h_p(\xi_i) - \tilde f(\xi_0); \; \tilde \tau^{\mathfrak f}_{t} >k, \tilde \nu_{n,t}=k, 
  |\tilde f(\xi_k )| >  n^{1/2-\ee}  \bigg) \notag \\
& =: J_1(k) - J_2(k) - J_3(k).
\end{align}
The first term $J_1(k)$ contributes directly to one part of the desired result. 
For the second term $J_2(k)$, by the first inequality \eqref{sum-hp-nu-nt-001} of Lemma \ref{Lemma martingale-001}, we have
\begin{align} \label{bound-J-112-001}
\sum_{k= p+1}^{ [ n^{1-\ee }] }  |J_2(k)|  
\leq  \frac{c}{n^{\beta }}  \left( \max \{t,0\} + | \sigma(g_p\cdots g_1, x)| +  |\tilde f (a)| +   \frac{1}{n}\sum_{j = p+1}^n \mathcal F_j(a) \right). 
 \end{align}
For the third term $J_3(k)$, 
in view of \eqref{def-tau-f-y} and \eqref{equality-hp-sigmap}, 
it holds that $\tilde \tau^{\mathfrak f}_{t} >k$ implies
$t' + \sum_{i=1}^{k-p} h_p(\xi_i) + \tilde f(\xi_k ) - \tilde f(\xi_0) \geq 0$, 
so that
\begin{align*}
- J_3(k) 
& \leq \bb E_a \left( \tilde f(\xi_k); \; \tilde \tau^{\mathfrak f}_{t} >k, \tilde \nu_{n,t}=k,  |\tilde f(\xi_k )| >  n^{1/2-\ee}  \right)  \notag\\
& \leq  \bb E_a \left( |\tilde f(\xi_k)|; \;   |\tilde f(\xi_k )| >  n^{1/2-\ee}  \right) 
  \notag\\
& \leq  c e^{ - \frac{\alpha}{2} n^{1/2-\ee} } \mathcal F_k(a). 
\end{align*}
Hence
\begin{align}\label{bound-J-113-001}
\sum_{k=p+1}^{ [ n^{1-\ee }] }  - J_3(k)  
\leq  c e^{ - \frac{\alpha}{2} n^{1/2-\ee} } \sum_{j = p+1}^n \mathcal F_j(a)
\leq  \frac{c}{n^{1+\beta}} \sum_{j = p+1}^n \mathcal F_j(a).
\end{align}
Since $\mathcal F_j(a) > 1$ and $1 \leq p \leq n^{\ee}$, it holds that $\sum_{j = p+1}^n \mathcal F_j(a) \geq \frac{n}{2}$. 
Therefore, 
combining \eqref{inequality-J11-hh-002}, \eqref{decom-J11-123}, \eqref{bound-J-112-001} and \eqref{bound-J-113-001}, 
and noting that $M_k= \sum_{i=1}^{k} h_p(\xi_i)$ (by \eqref{def-Mn-martingel-001}), 
we obtain 
\begin{align} \label{Start induction-001}
I_2 (a, t)
& \leq \left( 1+\frac{c}{n^{\ee }}\right)  \sum_{k=p+1}^{ [ n^{1-\ee } ] }
 \bb E_a \left( t' + M_k - \tilde f(\xi_0); 
      \tilde \tau^{\mathfrak f}_{t} >k, \tilde \nu_{n,t}=k \right)  \notag\\
   & \quad     +  \frac{c}{n^{1+\beta}}  ( 1  +  |\sigma(g_{p} \cdots g_{1}, x)| ) \bigg( \sum_{j = p+1}^n \mathcal F_j(a) \bigg)  \notag \\ 
& \quad\quad  \times  \bigg( |\tilde f(a)| + \max \{t,0\} +  |\sigma(g_{p} \cdots g_{1}, x)|  + \frac{1}{n} \sum_{j = p+1}^n \mathcal F_j(a)  \bigg). 
  \end{align}
Now we deal with the first expectation on the right-hand side of \eqref{Start induction-001}.
We shall prove that for $k\in [p+1, [n^{1-\ee }] ]$, 
\begin{align}\label{Equ-desired}
&      \bb{E}_a\left( t' + M_{k} - \tilde f(\xi_0);\  \tilde \tau^{\mathfrak f}_{t} > k,  \tilde \nu_{n,t} = k \right)   \notag\\
& \leq \bb{E}_a\left( t'+M_{ [n^{1-\ee }]} - \tilde f(\xi_0);\ \tilde \tau^{\mathfrak f}_{t} > [n^{1-\ee }],  \tilde \nu_{n,t}=k\right)   \notag\\
& \quad +  \bb{E}_a \Bigg( - \sum_{i = \tilde \tau^{\mathfrak f}_{t} -p+1}^{\tilde \tau^{\mathfrak f}_{t}} h_p(\xi_i) - \tilde f(\xi_{ \tilde \tau^{\mathfrak f}_{t} });
                               \  k + 1 \leq \tilde \tau^{\mathfrak f}_{t} \leq  [n^{1-\ee }],  \tilde \nu_{n,t} = k \Bigg). 
 \end{align}
Indeed, since the event $\{ \tilde \nu_{n,t} = k \}$ is in $\mathscr G_{k}$,
we have $\bb E_a (M_{ [n^{1-\ee }]}; \tilde \nu_{n,t} = k) = \bb E_a (M_{ k}; \tilde \nu_{n,t} = k)$, 
so that
\begin{align} \label{intermediate ident-001}
& \bb{E}_a \left( t' + M_{ [n^{1-\ee }]} - \tilde f(\xi_0);\ \tilde \tau^{\mathfrak f}_{t} \leq [n^{1-\ee }],  \tilde \nu_{n,t}=k\right) 
-  \bb{E}_a \left( t'+M_{k} - \tilde f(\xi_0);\ \tilde \tau^{\mathfrak f}_{t} \leq k,  \tilde \nu_{n,t}=k\right)   \notag\\
& =   \bb{E}_a \left( t'+M_{k}- \tilde f(\xi_0);\ \tilde \tau^{\mathfrak f}_{t} > k,  \tilde \nu_{n,t}=k\right)  \notag\\
& \qquad  - \bb{E}_a\left( t'+M_{ [n^{1-\ee }]} - \tilde f(\xi_0);\ \tilde \tau^{\mathfrak f}_{t} > [n^{1-\ee }],  \tilde \nu_{n,t}=k\right). 
\end{align}
Using \eqref{intermediate ident-001}, 
we see that proving \eqref{Equ-desired} is equivalent to showing the following inequality: 
\begin{align}\label{Equ-00a}
& \bb{E}_a\left( t'+M_{ [n^{1-\ee }]} - \tilde f(\xi_0);\ \tilde \tau^{\mathfrak f}_{t} \leq [n^{1-\ee }],  \tilde \nu_{n,t}=k\right)   \notag\\
& \leq  \bb{E}_a\left( t'+M_{k} - \tilde f(\xi_0);\ \tilde \tau^{\mathfrak f}_{t} \leq k,  \tilde \nu_{n,t}=k\right)  \notag\\
& \quad  +  \bb{E}_a \Bigg( - \sum_{i = \tilde \tau^{\mathfrak f}_{t} -p+1}^{\tilde \tau^{\mathfrak f}_{t}} h_p(\xi_i) - \tilde f(\xi_{ \tilde \tau^{\mathfrak f}_{t} });\  k + 1 \leq \tilde \tau^{\mathfrak f}_{t} \leq  [n^{1-\ee }],  \tilde \nu_{n,t} = k \Bigg). 
\end{align}
Now we prove \eqref{Equ-00a} by an induction argument. 
Let $j \in [k + 1, [n^{1-\ee }]]$ be an integer. 
Then
\begin{align*}
& \bb{E}_a \left( t' + M_j - \tilde f(\xi_0);\ \tilde \tau^{\mathfrak f}_{t} \leq j,  \tilde \nu_{n,t} = k \right)   \notag\\
& =  \bb{E}_a \left( t' + M_j - \tilde f(\xi_0);\ \tilde \tau^{\mathfrak f}_{t} \leq j -1,  \tilde \nu_{n,t} = k \right)
   + \bb{E}_a \left( t' + M_j - \tilde f(\xi_0);\ \tilde \tau^{\mathfrak f}_{t} = j,  \tilde \nu_{n,t} = k \right).  
\end{align*}
As the event $\{ \tilde \tau^{\mathfrak f}_{t} \leq  j - 1,  \tilde \nu_{n,t}=k \}$ is in $\mathscr G_{j-1}$,
by the martingale property, we have
\begin{align*}
\bb{E}_a \left( t' + M_j - \tilde f(\xi_0);\ \tilde \tau^{\mathfrak f}_{t} \leq j -1,  \tilde \nu_{n,t} = k \right)
= \bb{E}_a \left( t' + M_{j -1} - \tilde f(\xi_0);\ \tilde \tau^{\mathfrak f}_{t} \leq  j -1,  \tilde \nu_{n,t} = k \right). 
\end{align*}
Besides, by \eqref{def-tau-f-y} and \eqref{def-Mn-martingel-001}, on the set $\{\tilde \tau^{\mathfrak f}_{t} = j \}$, 
it holds that $t' + M_{j - p} + \tilde f(\xi_j) - \tilde f(\xi_0)<0$, which leads to
\begin{align*}
\bb{E}_a \left( t' + M_j - \tilde f(\xi_0);\ \tilde \tau^{\mathfrak f}_{t} = j,  \tilde \nu_{n,t}=k\right)
\leq \bb{E}_a \bigg( - \sum_{i = j - p + 1}^j h_p(\xi_i) - \tilde f(\xi_j); \ \tilde \tau^{\mathfrak f}_{t} = j,  \tilde \nu_{n,t}=k \bigg).
\end{align*}
Hence, for any integer $j \in [k+1, [n^{1-\ee }]]$,
\begin{align*}
& \bb{E}_a \left( t' + M_j - \tilde f(\xi_0);\ \tilde \tau^{\mathfrak f}_{t} \leq j,  \tilde \nu_{n,t} = k \right)   \notag\\
& \leq \bb{E}_a \left( t'+M_{j -1} - \tilde f(\xi_0);\ \tilde \tau^{\mathfrak f}_{t} \leq j -1,  \tilde \nu_{n,t} = k \right)  \notag\\
& \quad   +  \bb{E}_a \bigg( - \sum_{i = j - p + 1}^j h_p(\xi_i) - \tilde f(\xi_j);\ \tilde \tau^{\mathfrak f}_{t} = j,  \tilde \nu_{n,t} = k \bigg).
\end{align*}
Summing these bounds up in $j \in [k+1, [n^{1-\ee }]]$, we get
\begin{align*}
& \bb{E}_a \left( t'+M_{ [n^{1-\ee }]} - \tilde f(\xi_0);\ \tilde \tau^{\mathfrak f}_{t} \leq [n^{1-\ee }],  \tilde \nu_{n,t} = k \right)   \notag\\
& \leq  \bb{E}_a \left( t'+M_{k} - \tilde f(\xi_0);\ \tilde \tau^{\mathfrak f}_{t} \leq k,  \tilde \nu_{n,t}=k\right)   \notag\\
& \quad  +  \sum_{j = k+1}^{ [n^{1-\ee }] }   
\bb{E}_a \bigg( - \sum_{i = j-p+1}^h h_p(\xi_i) - \tilde f(\xi_j); \ \tilde \tau^{\mathfrak f}_{t} = j,  \tilde \nu_{n,t} = k \bigg)  \notag\\
& =  \bb{E}_a \left( t'+M_{k} - \tilde f(\xi_0);\ \tilde \tau^{\mathfrak f}_{t} \leq k,  \tilde \nu_{n,t}=k\right)   \notag\\ 
& \quad  +   \bb{E}_a \Bigg( - \sum_{i = \tilde \tau^{\mathfrak f}_{t} - p+1}^{\tilde \tau^{\mathfrak f}_{t}} h_p(\xi_i) - \tilde f(\xi_{ \tilde \tau^{\mathfrak f}_{t} });
             \  k + 1 \leq \tilde \tau^{\mathfrak f}_{t} \leq  [n^{1-\ee }],  \tilde \nu_{n,t} = k \Bigg). 
\end{align*}
This shows \eqref{Equ-00a} and therefore \eqref{Equ-desired} holds.

Summing over $k$ in \eqref{Equ-desired}, we obtain
\begin{align}\label{decom-H-H1H2}
H (a, t): & = \sum_{k=p+1}^{ [ n^{1-\ee } ] } \bb{E}_a\left( t'+M_{k} - \tilde f(\xi_0);\ \tilde \tau^{\mathfrak f}_{t} > k,  \tilde \nu_{n,t}=k\right)    \notag\\
& \leq  \bb{E}_a\left( t'+M_{ [n^{1-\ee }]} - \tilde f(\xi_0);\ \tilde \tau^{\mathfrak f}_{t} > [n^{1-\ee }],  \tilde \nu_{n,t} \leq [n^{1-\ee }]  \right)  
    \notag\\
& \quad 
+ \bb{E}_a \Bigg( \Bigg| \sum_{i = \tilde \tau^{\mathfrak f}_{t} -p+1}^{\tilde \tau^{\mathfrak f}_{t}} h_p(\xi_i) \Bigg| + |\tilde f(\xi_{ \tilde \tau^{\mathfrak f}_{t} })|; \  \tilde \nu_{n,t} + 1 \leq \tilde \tau^{\mathfrak f}_{t} \leq  [n^{1-\ee }] \Bigg). 
\end{align}
For the first term on the right-hand side of \eqref{decom-H-H1H2}, 
by the definition of $W^{\mathfrak f}_{n}(a,t)$ 
(cf.\ \eqref{EXPECT-E_x-002})
and the fact that $\tilde f(\xi_0) = \tilde f(a)$, we have 
\begin{align} \label{bound-H1-abc}
& \bb{E}_a \left( t' + M_{ [n^{1-\ee }]} - \tilde f(\xi_0);\ \tilde \tau^{\mathfrak f}_{t} > [n^{1-\ee }],  \tilde \nu_{n,t} \leq [n^{1-\ee }]  \right) \notag \\
& = W^{\mathfrak f}_{[n^{1-\ee }] } (a, t) 
- \tilde f(a) \bb P_a(\tilde \tau^{\mathfrak f}_{t} > [n^{1-\ee }],  \tilde \nu_{n,t} \leq [n^{1-\ee }] )  + H (a, t)  \notag \\
& \leq  W^{\mathfrak f}_{[n^{1-\ee }] } (a, t) + 
|\tilde f(a)| \bb P_a \left( \tilde \tau^{\mathfrak f}_{t} > [n^{1-\ee }] \right) + K (a, t),  
\end{align}
where, for short, we denote
\begin{align*}
K (a, t) = - \bb{E}_a \left( t' + M_{ [n^{1-\ee }]};\ \tilde \tau^{\mathfrak f}_{t} > [n^{1-\ee }],  \tilde \nu_{n,t} > [n^{1-\ee }]  \right). 
\end{align*}
For the second term on the right-hand side of \eqref{bound-H1-abc}, 
by Lemma \ref{Lem-tau-prior}, 
there exist constants $c, \beta, \ee > 0$  such that, 
for any $a \in \bb A_p$, $t \in \bb R$, $n\geq 1$ and $p\leq n^{\ee}$,  
\begin{align} \label{bound Q_n 001}
 \bb P_a \left( \tilde \tau^{\mathfrak f}_{t} > [n^{1-\ee }] \right) 
 \leq  \frac{c}{n^{\beta}}
\bigg( \max \{t,0\} +  |\sigma(g_p\cdots g_1,x)|   +|\tilde{f}(a)| 
 + \frac{1}{n} \sum_{k = p+1}^n \mathcal F_k(a) \bigg). 
\end{align}
For the last term $K (a, t)$ on the right-hand side of \eqref{bound-H1-abc}, 
by \eqref{def-tau-f-y}, \eqref{def-Mn-martingel-001} and 
the fact that $t' = t + \sigma(g_p \cdots g_1, x) = t + \tilde f(\xi_0)$, on the set $\{ \tilde \tau^{\mathfrak f}_{t} > [n^{1-\ee }] \}$,  
we have $t' + M_{ [n^{1-\ee }] - p }  + \tilde f(\xi_{[n^{1-\ee }]} ) \geq 0$, 
so that
 \begin{align*}
K(a, t) \leq   \bb{E}_a \Bigg(  \sum_{i = [n^{1-\ee }] - p + 1}^{ [n^{1-\ee }] } |h_p(\xi_i)|  +  |\tilde f(\xi_{[n^{1-\ee }]} )|;  
   \ \tilde \tau^{\mathfrak f}_{t} > [n^{1-\ee }],  \tilde \nu_{n,t} > [n^{1-\ee }]  \Bigg). 
\end{align*}
By the Cauchy-Schwarz inequality, the moment assumption \eqref{exp mom for f 001}
 and Lemma \ref{Lemma 2-app}, it follows that there exist constants $\beta, \ee, c_{\ee} > 0$  such that, 
for any $a \in \bb A_p$, $t \in \bb R$, $n \geq 1$ and $p \leq n^{\ee}$,  
\begin{align} \label{bound for R_n 001}
K(a, t) & \leq   \bb{E}_a^{1/2}
   \Bigg( \sum_{i = [n^{1-\ee }] - p + 1}^{ [n^{1-\ee }] } |h_p(\xi_i)| +  |\tilde f(\xi_{[n^{1-\ee }]} )|  \Bigg)^2 
   \mathbb{P}_a^{1/2}\left( \tilde \nu_{n,t} > [n^{1-\ee }]  \right)  \notag\\
& \leq  c_{\ee}  e^{-\beta n^{\ee}}  \mathcal F_{[n^{1-\ee}]} (a). 
\end{align}
For the last term on the right-hand side of \eqref{decom-H-H1H2}, 
by Lemma \ref{Lemma martingale-002}, 
there exist constants $c, \beta, \ee > 0$  such that, 
for any $a \in \bb A_p$, $t \in \bb R$, $n\geq 1$ and $p\leq n^{\ee}$, 
\begin{align}\label{bound H2-001}
& \bb{E}_a \Bigg( \Bigg| \sum_{i = \tilde \tau^{\mathfrak f}_{t} -p+1}^{\tilde \tau^{\mathfrak f}_{t}} h_p(\xi_i) \Bigg| 
 + |\tilde f(\xi_{ \tilde \tau^{\mathfrak f}_{t} })|; \  \tilde \nu_{n,t}  + 1 \leq \tilde \tau^{\mathfrak f}_{t} \leq  [n^{1-\ee }]  \Bigg)\notag  \\
&  \leq   \frac{c}{n^{1 + \beta }} \left( 1 + \max \{t,0\} + | \sigma(g_p\cdots g_1, x)| + |\tilde f(a)|  \right) \sum_{k = p+1}^n \mathcal F_k(a). 
\end{align}
Substituting \eqref{bound-H1-abc}, \eqref{bound Q_n 001}, \eqref{bound for R_n 001} and \eqref{bound H2-001} 
into \eqref{decom-H-H1H2}, 
and taking into account that $\sum_{j = p+1}^n \mathcal F_j(a) \geq \frac{n}{2}$ for $1 \leq p \leq n^{\ee}$, 
we get
\begin{align*}
& H (a, t)   \leq  W^{\mathfrak f}_{[n^{1-\ee }] } (a, t)   \notag\\
&  +  \frac{c}{n^\beta} \bigg( |\tilde f(a)| +  \frac{1}{n} \sum_{k = p+1}^n \mathcal F_k(a) \bigg)
\bigg( \max \{t,0\} +  |\sigma(g_p\cdots g_1,x)| + |\tilde f(a)| +  \frac{1}{n} \sum_{k = p+1}^n \mathcal F_k(a) \bigg). 
\end{align*}
Recalling the definition of $H(a, t)$ (cf.\ \eqref{decom-H-H1H2})
and substituting the latter bound into \eqref{Start induction-001} gives 
\begin{align*} 
I_2 (a, t)
&\leq \left( 1+\frac{c_{\ee } }{n^{\ee }}\right)  
   W^{\mathfrak f}_{[n^{1-\ee }] } (a, t)  \notag  
    + \frac{c}{n^{\beta}}  ( 1 +  |\sigma(g_{p} \cdots g_{1}, x)|) 
     \bigg(  |\tilde f(a)| +  \frac{1}{n} \sum_{j = p+1}^n \mathcal F_j(a)  \bigg) \notag  \\
& \qquad\qquad \times   \bigg( \max \{t,0\} +  |\sigma(g_{p} \cdots g_{1}, x)| + |\tilde f(a)| 
  + \frac{1}{n} \sum_{j = p+1}^n \mathcal F_j(a) \bigg). 
  \end{align*}
This, together with \eqref{bound J1 001} and \eqref{bound J12 001}, implies that
\begin{align*}
& \bb{E}_a \left( t'+M_{n};\ \tilde \tau^{\mathfrak f}_{t} > n,\ \tilde \nu_{n,t} \leq n^{1-\ee } \right)  \notag\\
& \leq \left( 1+\frac{c_{\ee } }{n^{\ee }}\right)  
   W^{\mathfrak f}_{[n^{1-\ee }] } (a, t)  + \frac{c}{n^{\beta}}  ( 1 +  |\sigma(g_{p} \cdots g_{1}, x)|) 
  \bigg( |\tilde f(a)|  +  \frac{1}{n} \sum_{j = p+1}^n \mathcal F_j(a) \bigg)      \notag  \\
&  \qquad\qquad  \times    
\bigg( \max \{t,0\} +  |\sigma(g_{p} \cdots g_{1}, x)| +  |\tilde f(a)| + \frac{1}{n} \sum_{j = p+1}^n \mathcal F_j(a) \bigg) .
\end{align*}
Implementing this and the bound \eqref{bound-J2-ay} into \eqref{bound J1+J2}, 
we obtain 
\begin{align*}
W^{\mathfrak f}_{n} (a, t) 
 &\leq \left( 1+\frac{c_{\ee } }{n^{\ee }}\right)  
   W^{\mathfrak f}_{[n^{1-\ee }] } (a, t)  
 + \frac{c}{n^{\beta}}  ( 1 +  |\sigma(g_{p} \cdots g_{1}, x)| ) 
 \bigg( |\tilde f(a)| + \frac{1}{n} \sum_{j = p+1}^n \mathcal F_j(a)  \bigg) \notag  \\ 
 & \qquad\qquad  \times 
 \bigg( \max \{t,0\} +  |\sigma(g_{p} \cdots g_{1}, x)| + |\tilde f(a)| + \frac{1}{n} \sum_{j = p+1}^n \mathcal F_j(a) \bigg).
\end{align*}
This finishes the proof of Lemma \ref{Lemma 4}.
\end{proof}

\subsection{Proof of the quasi-decreasing behaviour without twist function} \label{AAsubsec-proof Prop 001}

We now proceed to establish inequality \eqref{bound with m for U-105-01-3} of Theorem \ref{Pro-Appendix-Final2-Inequ}. 
We first address the case where 
the twist function $\theta$ equals $1$
and the functions $f_n$  depend only on the first $p$ coordinates.

\begin{proposition} \label{Prop 001-app}
Suppose that the cocycle $\sigma$ admits finite exponential moments \eqref{exp mom for f 001}
and is centered \eqref{centering-001}. 
We also suppose that the effective central limit theorem \eqref{BEmart-001}  is satisfied. 
Then, there exist constants $\ee > 0$ and $c>0$ such that for any $n \geq 1$, $1\leq p\leq n^{\ee/2}$, 
$t \in \bb R,$ 
and any sequence $\mathfrak f=(f_n)_{n\geq0}$ of measurable functions on $ \bb G^{\{1,\ldots,p \}} \times \bb  X$, 
\begin{align*}
U^{\mathfrak f}_{n} (t) 
 \leq \left( 1+\frac{c }{n^{\ee }}\right)  
   U^{\mathfrak f}_{[n^{1-\ee }] } (t)  + c (1+ \max\{t, 0\}) \frac{1}{n^{\ee/2 }} 
	 C_{\alpha}(\mathfrak f).
\end{align*}
\end{proposition}

\begin{proof}
The assertion of Proposition \ref{Prop 001-app} 
is obtained from \eqref{MAIN_GOAL-002}, \eqref{Expect-E_x001}, 
Corollary \ref{Lem-bound-tau-p} and Lemma \ref{Lemma 4}.
Note that $\bb E \sigma(g_{p} \cdots g_{1}, x)^2 \leq c p$, by the moment assumption \eqref{exp mom for f 001} 
and the centering assumption \eqref{centering-001}.  
Since $a=(g_0,\ldots,g_{p},x,0) \in \bb A_p$, 
integrating both sides of \eqref{bound-M_n-001} in Lemma \ref{Lemma 4} (with $\alpha/4$ instead of $\alpha$) 
with respect to the measure $\mu(dg_0) \ldots \mu(dg_{p})  \nu(dx)$ and using  \eqref{Expect-E_x001} 
yields that, for any $t \in \bb R$, $n\geq 1$ and $1\leq p \leq n^{\ee/2}$, 
\begin{align} \label{Expect-E_x001-002}
\int_{\bb X} \bb  E \left( t+S_{n+p}^x ; \tau^{\mathfrak f}_{x, t}>n \right) \nu(dx) 
& \leq 
\left( 1+\frac{c_{\ee}}{n^{\ee }}\right)  
   \int_{\bb X} \bb E \left( t+S_{[n^{1-\ee}]+p}^x ; \tau^{\mathfrak f}_{x, t}> [n^{1-\ee}] \right) \nu(dx)  \notag  \\
& \quad + \frac{c_{\ee}}{n^{\ee/2}}    \left(1+ \max \{t,0\} \right)  C_{\alpha}(\mathfrak f),
\end{align}
where we have used the fact that the measure $\nu$ is $\mu$-stationary, as in \eqref{stationary-nu-001}. 
Using Lemma \ref{Lem-bound-tau-p}, we can slightly modify the expectations on both sides of \eqref{Expect-E_x001-002} to get
\begin{align*} 
U_n^{\mathfrak f}(t)&= \int_{\bb X} \bb E \left( t+S_{n}^x + f_n( T^n (\omega,x) ) - f_0(\omega,x) ; \tau^{\mathfrak f}_{x, t}>n \right) \nu(dx) \notag  \\
& \leq 
\left( 1+\frac{c_{\ee}}{n^{\ee }}\right)  
  U_{[n^{1-\ee}]}^{\mathfrak f}(t) 
   + \frac{c_{\ee}}{n^{\ee/2}}    \left(1+ \max \{t,0\} \right) 
     C_{\alpha}(\mathfrak f).
\end{align*}
This concludes the proof of Proposition \ref{Prop 001-app}.
\end{proof}

Using Proposition \ref{Prop 001-app} and the approximation by finite-size 
perturbations from Proposition \ref{Prop g approx 002}, we now obtain the following result for functions $f_n$ depending on infinitely many coordinates. 

\begin{proposition} \label{AA-Prop iterative 002}
Suppose that the cocycle $\sigma$ admits finite exponential moments \eqref{exp mom for f 001}
and is centered \eqref{centering-001}. 
We also suppose that the effective central limit theorem \eqref{BEmart-001} is satisfied. 
Assume that $\mathfrak f = (f_n)_{n \geq 0}$ is a sequence of  measurable functions on $\Omega \times \bb X$ 
satisfying the moment condition \eqref{exp mom for g 002} and the approximation property \eqref{approxim rate for gp-002}.
Then, there exist constants $\ee > 0$ and $c>0$ such that for any $n \geq 1$ and $t \in \bb R$,  
\begin{align*}
U^{\mathfrak f}_{n} (t) 
 &\leq \left( 1+\frac{c }{n^{\ee }}\right)  
   U^{\mathfrak f}_{[n^{1-\ee }] } (t+ 2 n^{-\ee})  
   + \frac{c}{n^{\ee/2 }} \Big( \max \{t,0\} + C_{\alpha}(\mathfrak f)\Big)  
   \Big( C_{\alpha}(\mathfrak f)  + D_{\alpha,\beta}(\mathfrak f) \Big).
\end{align*}
\end{proposition}

\begin{proof} 
Let $\ee>0$ be as in Proposition \ref{Prop 001-app}.
Pick $n\geq 1$ and set $p=[n^{\ee/2}]$.
Using the upper bound from Proposition \ref{Prop g approx 002} with $\gamma=\ee$ and $\delta=\ee/2$, we get
\begin{align} \label{AA-proof-propapprox-001}
U^{\mathfrak f}_n (t) 
\leq U^{\mathfrak f_p}_n (t+ n^{-\ee}) 
+ c \left( \max \{t,0\} + C_{\alpha}(\mathfrak f) \right) e^{ -\beta n^{\ee/2}} D_{\alpha,\beta}(\mathfrak f).
\end{align}
By Proposition \ref{Prop 001-app}, we have
\begin{align*} 
U^{\mathfrak f_p}_{n} (t+ n^{-\ee}) 
 &\leq \left( 1+\frac{c }{n^{\ee }}\right)  
   U^{\mathfrak f_p}_{[n^{1-\ee }] } (t+ n^{-\ee})  + c (1+ \max \{t,0\}) \frac{1}{n^{\ee/2 }} C_{\alpha}(\mathfrak f_p).
\end{align*}
By Jensen's inequality, we have $C_{\alpha}(\mathfrak f_p) \leq C_{\alpha}(\mathfrak f)$. 
Hence, using the lower bound from Proposition \ref{Prop g approx 002}, we obtain 
\begin{align*}
U^{\mathfrak f_p}_{n} (t + n^{-\ee}) 
 \leq \left( 1+\frac{c }{n^{\ee }}\right)  
   U^{\mathfrak f}_{[n^{1-\ee }] } (t+ 2n^{-\ee})  + c (\max \{t,0\} + C_{\alpha}(\mathfrak f)) \frac{1}{n^{\ee/2 }} D_{\alpha,\beta}(\mathfrak f).
\end{align*}
Combining this with \eqref{AA-proof-propapprox-001}, we deduce the assertion of Proposition \ref{AA-Prop iterative 002}.
\end{proof}

Thanks to the general technique from Lemma \ref{lem-iter-proc-001}, we get the following assertion, which
proves inequality \eqref{bound with m for U-105-01-3} of Theorem \ref{Pro-Appendix-Final2-Inequ}
in the case where $\theta=1$.  

\begin{proposition}\label{Pro-Appendix-Final-Inequ}
Suppose that the cocycle $\sigma$ admits finite exponential moments \eqref{exp mom for f 001}
and is centered \eqref{centering-001}.  
We also suppose that the effective central limit theorem \eqref{BEmart-001} is satisfied.  
For any $B >0$, there exist constants $A, \ee, \gamma >0$ with the following property. 
Assume that $\mathfrak f = (f_n)_{n \geq 0}$ is a sequence of  measurable functions on $\Omega \times \bb X$ 
satisfying the moment condition \eqref{exp mom for g 002} and the approximation property \eqref{approxim rate for gp-002}
with $C_{\alpha}(\mathfrak f) \leq B$ and $D_{\alpha,\beta}(\mathfrak f) \leq B$. 
Then, for any $1 \leq n\leq m$ and $t \in \bb R$, we have 
\begin{align*}
U^{\mathfrak f}_{m} (t) \leq   
U^{\mathfrak f}_{n} \left( t + A n^{-\gamma} \right) + A n^{-\ee} (1+\max \{t,0\}). 
\end{align*}
\end{proposition}

\begin{proof}
This is a direct consequence of Propositions \ref{lem-U is increasing-1002} and \ref{AA-Prop iterative 002}, 
by using  Lemma \ref{lem-iter-proc-001}.
Note that the property \eqref{eq-bound for V_1-001} in Lemma \ref{lem-iter-proc-001} holds due to Lemma \ref{Lem-trick}
and our moment assumption. 
\end{proof}

\subsection{Proof of the quasi-decreasing behaviour with twist function}

We finally prove inequality \eqref{bound with m for U-105-01-3} of Theorem  \ref{Pro-Appendix-Final2-Inequ} in full generality,
by considering an arbitrary twist function $\theta$. 
We begin with an analog of Proposition \ref{Prop 001-app} that includes the twist function $\theta$, 
where the functions $f_n$ depend on the first $p$ coordinates.

\begin{proposition} \label{Prop 001-app2}
Suppose that the cocycle $\sigma$ admits finite exponential moments \eqref{exp mom for f 001}
and is centered \eqref{centering-001}. 
We also suppose that the effective central limit theorem \eqref{BEmart-001} is satisfied.  
Then, there exist constants $\ee > 0$ and $c>0$ such that for any $n \geq 1$, $1\leq p\leq n^{\ee/2}$, $t \in \bb R,$ 
any sequence $\mathfrak f = (f_n)_{n\geq0}$ of measurable functions on $ \bb G^{\{1,\ldots,p \}} \times \bb  X$, 
and any non-negative measurable function $\theta$ on $ \bb G^{\{1,\ldots,p \}}$, we have 
\begin{align*}
U^{\mathfrak f, \theta}_{n} (t) 
 \leq \left( 1+\frac{c }{n^{\ee }}\right)  
   U^{\mathfrak f, \theta}_{[n^{1-\ee }] } (t)  + c n^{- \ee/2} (1+ \max\{t, 0\}) 
	 C_{\alpha}(\mathfrak f) \|\theta\|_{\infty}.
\end{align*}
\end{proposition}

\begin{proof}
The demonstration is derived using a similar approach to that of Proposition \ref{Prop 001-app}.
This involves multiplying the inequality from Lemma \ref{Lemma 4} by $\theta(g_1, \ldots, g_p)$
 and then applying the integral formula \eqref{integral-formula-app}.
\end{proof}

By using the technique based on finite-size approximation 
of perturbations, we extend Proposition \ref{Prop 001-app2} to the case of 
perturbation functions $f_n$ depending on infinitely many coordinates.

\begin{proof}[Proof of Proposition \ref{AA-Prop iterative 002-app}]
The bound \eqref{AA-bound-M_n-002-app} follows from Propositions \ref{Prop 001-app2} and \ref{Prop UfA g approx 002-app}. 
\end{proof}

\begin{proof}[Proof of the second part of Theorem \ref{Pro-Appendix-Final2-Inequ}]
Inequality \eqref{bound with m for U-105-01-3} follows from 
Propositions  \ref{AA-Prop iterative 002-app} and \ref{lem-UfA is increasing-1002-app}, 
by using Lemmas \ref{Lem-trick} and \ref{lem-iter-proc-001} as in the proof of Proposition \ref{Pro-Appendix-Final-Inequ}.  
\end{proof}

\section*{Acknowledgments}  

Ion Grama was supported by 
the ANR project "Rawabranch" number ANR-23-CE40-0008, 
the France 2030 framework program, Centre Henri Lebesgue ANR-11-LABX-0020-01,
and Hui Xiao was supported by the National Natural Science Foundation of China (Grant Nos. 12595283 and 12288201). 
We would like to thank the referees for their helpful comments and remarks.


\end{document}